\documentclass[submit,hidelinks,onefignum,onetabnum]{siamart250211}

\usepackage{lipsum}
\usepackage{amsfonts}
\usepackage{graphicx}
\usepackage{epstopdf}
\usepackage{algorithmic}
\ifpdf
  \DeclareGraphicsExtensions{.eps,.pdf,.png,.jpg}
\else
  \DeclareGraphicsExtensions{.eps}
\fi

\usepackage{multirow}

\makeatletter
\@mparswitchfalse        
\makeatother
\normalmarginpar

\newsiamremark{remark}{Remark}
\newsiamremark{hypothesis}{Hypothesis}
\newsiamremark{as}{Assumption}
\newsiamremark{setting}{Setting}

\crefname{hypothesis}{Hypothesis}{Hypotheses}
\newsiamthm{claim}{Claim}
\newsiamremark{fact}{Fact}
\crefname{fact}{Fact}{Facts}

\def\N{\mathbb{N}}

\def\P{\mathbb{P}}
\definecolor{myblue}{rgb}{0.2,0,0.9}
\definecolor{blue_violet}{rgb}{0.54, 0.17, 0.89}
\definecolor{darkgreen}{rgb}{0,0.35,0}

\allowdisplaybreaks

\makeatletter
\DeclareRobustCommand*\cal{\@fontswitch\relax\mathcal}
\makeatother

\theoremstyle{plain}
\numberwithin{equation}{section}

\DeclareMathOperator*{\argmin}{arg\,min}

\DeclareMathOperator*{\esssup}{ess\,sup}
\DeclareMathOperator*{\essinf}{ess\,inf}

\usepackage{amssymb}

\headers{Robust exploratory stopping under ambiguity in RL}{J. Ye, H.Y. Wong, and K. Park}

\title{
    ROBUST EXPLORATORY STOPPING UNDER AMBIGUITY IN REINFORCEMENT LEARNING
\thanks{Submitted to the editors \today.
\funding{H.\;Y.\;Wong acknowledges the support from the Research Grants Council of Hong Kong (grant DOI:  GRF14308422). K.\;Park acknowledges the support from the National
Research Foundation of Korea (grant DOI: RS-2025-02633175). }
}
}

\author{
Junyan Ye\thanks{Department of Statistics and Data Science, The Chinese University of Hong Kong, Shatin, N.T., Hong Kong 
  (\email{junyanye@link.cuhk.edu.hk}, \email{hywong@cuhk.edu.hk}).}
\and Hoi Ying Wong\footnotemark[2]
  \and
  Kyunghyun Park\thanks{Division of Mathematical Sciences,
Nanyang Technological University, Singapore
  (\email{kyunghyun.park@ntu.edu.sg}).}
  }

\usepackage{amsopn}

\begin{document}

\maketitle

\begin{abstract}
    We propose and analyze a continuous-time robust reinforcement learning framework for optimal stopping under ambiguity. In this framework, an agent chooses a robust exploratory stopping time motivated by two objectives: {robust decision-making} under ambiguity and {learning} about the unknown environment. 
	Here, ambiguity refers to considering multiple probability measures~dominated by a reference measure, reflecting the agent’s awareness that the reference measure representing her learned belief about the environment would be {erroneous}. 
	Using the $g$-expectation framework, we reformulate the optimal stopping problem under ambiguity as a robust exploratory control problem with Bernoulli distributed controls. 
    We then characterize the optimal Bernoulli distributed control via backward stochastic differential equations and, based on this, construct the robust exploratory stopping time that approximates the optimal stopping time under ambiguity. Last, we establish a policy iteration theorem and implement it as a reinforcement learning algorithm.  Numerical experiments demonstrate the convergence, robustness, and scalability of our reinforcement learning algorithm across different levels of ambiguity and exploration. 
\end{abstract}

\begin{keywords}
optimal stopping, exploratory control, reinforcement learning, ambiguity, robust optimization, $g$-expectation, policy iteration.
\end{keywords}

\begin{MSCcodes}
60G40, 60H10, 68T07, 49L20
\end{MSCcodes}

\section{Introduction}\label{sec:intro}
Optimal stopping is a class of decision problems in which one seeks to choose a time to take a certain action so as to maximize an expected reward. It is applied in various fields, for instance to analyze the optimality of the sequential probability ratio test in statistics (e.g.,\;\cite{wald1948optimum}), to study consumption habits in economics  (e.g.,\;\cite{dybvig1995dusenberry})
and notably to derive American option pricing (e.g.,\;\cite{peskir2006optimal}). A common challenge arising in all these fields is finding the best model to describe the underlying process or probability measure, which is usually {\it unknown}. Although significant efforts have been made to propose and analyze general stochastic models with improved estimation techniques, a margin of error in estimation inherently exists. 

In response to such model misspecification and estimation errors, recent works, Dai et al.\;\cite{dai2024learning} and Dong\;\cite{DongRLstopping}, have cast optimal stopping problems within the continuous time reinforcement learning (RL) framework of Wang et al.\;\cite{wang2020reinforcement} and Wang and Zhou\;\cite{wang2020continuous}. Arguably, the exploratory (or randomized) optimal stopping framework is viewed as {\it model-free}, since agents, even without knowledge of the true model or underlying dynamics of the environment, can learn from observed data and determine a stopping rule that yields the best outcome. In this sense, the framework provides a systematic way to balance exploration and exploitation in optimal~stopping. 



Although model-free RL is sometimes interpreted as enabling direct deployment without an explicit model, this view can be misleading in practice. Limited and noisy training data may lead to inaccurate learned dynamics or value surrogates, and environmental conditions can change between training and execution; in physical or safety-critical systems, such discrepancies can entail severe consequences. 
Moreover, model-free learning is typically data-intensive, so one often relies on simulators or limited real data, thereby reintroducing {\it ambiguity} (or Knightian uncertainty) and creating a trade-off between safety during training and generalization at deployment \cite{christiano2016transfer,morimoto2005robust}. 
In addition, real-world training is costly, time-consuming, and potentially unsafe, which often restricts data collection to a narrow set of scenarios and may result in overfitting and poor generalization under distributional shifts \cite{pinto2017robust}. These considerations motivate robust RL approaches that seek policies whose performance remains reliable under model ambiguity and related disturbances. 

This motivation is further reinforced by the observation that learning need not eliminate ambiguity at the level of beliefs. In particular, as argued in \cite{epstein2003recursive} and emphasized by \cite{marinacci1999limit}, the link between empirical frequencies and asymptotic (learned) beliefs can be weakened by the degree of ambiguity in an agent’s prior beliefs, so that ambiguity may persist even with extensive learning. Consequently, the reliability of purely model-free exploratory frameworks can be fundamentally limited under ambiguity and related forms of model misspecification. In this sense, robust RL is both practically desirable and conceptually natural. Arguably, worst-case or minimax formulations provide a canonical framework for robust RL \cite{garcia2015comprehensive-safeRL}. 
\color{black}

The aim of this article is to propose and analyze a continuous-time RL framework for optimal stopping under ambiguity. Our framework starts with revisiting an optimal stopping problem under $g$-expectation (
see 
Coquet et al.\;\cite{coquet2002filtration}, Peng\;\cite{peng1997backward}):
Let ${\cal T}_{t}$ be the set of all stopping times with values in $[t,T]$. Denote by ${\cal E}_t^g[\cdot]$ the conditional $g$-expectation with a driver $g:\Omega\times[0,T]\times \mathbb{R}^d\to \mathbb{R}$ 
, representing a worst-case expectation (see Definition~\ref{def:g_expect}, Remark \ref{rem:represent_g_exp}). 
Then, the optimal stopping problem~is 
\begin{align}\label{eq:intro_OST_g}
V_t^{x}:= \esssup_{\tau \in {\cal T}_t}{\cal E}^g_t\bigg[\int_t^\tau e^{-\int_t^s\beta_udu} r(X_s^{x})ds + e^{-\int_t^\tau \beta_udu}R(X_\tau^{x}) \bigg], 
\end{align}
where $(\beta_t)_{t\in[0,T]}$ is a discount rate, $r:\mathbb{R}^d\to \mathbb{R}$ and $R:\mathbb{R}^d\to \mathbb{R}$ are reward~functions, and $(X_t^x)_{t\in[0,T]}$ is an It\^o-semimartingale with an initial state 
$x\in \mathbb{R}^d$; see \eqref{eq:refer_semi}. 

We then combine the penalization method of \cite{el1997backward,lepeltier2005penalization,peng2005smallest} (used to establish the well-posedness of reflected backward stochastic differential equations (BSDEs) characterizing \eqref{eq:intro_OST_g}) with the entropy regularization framework of \cite{wang2020reinforcement,wang2020continuous} to study the following robust exploratory control problem:
\begin{align}\label{eq:intro_explore_robust_problem}
	\begin{aligned}
	\overline V_t^{x;N,\lambda}:=\esssup_{\pi \in \Pi}{\cal E}^g_t&\bigg[\int_t^Te^{-\int_t^s(\beta_u+N\pi_u)du}\big(r(X_s^x)+R(X_s^x)\, N\pi_s -\lambda {\cal H}(\pi_s) \big)ds\\
	&\quad+e^{-\int_t^T(\beta_u+N\pi_u)du} R(X_T^x)\bigg], 
	\end{aligned}
\end{align}
where $\Pi$ is the set of all 
$[0,1]$-valued, Bernoulli-distributed controls randomizing stopping rules; see Remark~\ref{rem:bernoulli},
$-{\cal H}(a)$ denotes the binary differential entropy for $a\in[0,1]$; see \eqref{eq:entropy}, \color{black} $\lambda>0$ represents the level of exploration, and $N\in \mathbb{N}$ represents the penalization level used for approximation of \eqref{eq:intro_OST_g}.

In Theorem \ref{thm:verification}, we show that if the characteristics of $(X_t^x)_{t\in[0,T]}$ are sufficiently integrable (see Assumption\;\ref{as:refer_semi}), $r$ and $R$ satisfy certain regularity and growth conditions, and $(\beta_t)_{t\in[0,T]}$ is uniformly bounded (see Assumption \ref{as:reward}), then the robust exploratory control problem $\overline V^{x;N,\lambda}$ in \eqref{eq:intro_explore_robust_problem} can be characterized by a BSDE (see~\eqref{eq:g_BSDE_explore_optimal}). Moreover, its optimal Bernoulli-distributed control is given~by 
\begin{align}\label{eq:intro_optimal_explore}
	\pi^{*,x;N,\lambda}_t:=\operatorname{logit}\bigg(\frac{N}{\lambda}(R(X_t^x)-\overline V_t^{x;N,\lambda})\bigg)\in \Pi,
\end{align}
where $\operatorname{logit}(x):=(1+\exp(-x))^{-1}$, $x\in\mathbb{R}$, denotes the standard logistic function. 

Next, using $\pi^{*,x;N,\lambda}$ in \eqref{eq:intro_optimal_explore}, we define 
a stopping time by
\begin{align}\label{eq:intro_optimal_explore_stopping}    
    {\tau}^{x;N,\lambda}_t :=\inf \bigg\{s \ge t : \pi^{*,x;N,\lambda}_s\geq \frac{1}{1+2^{N(T-s)}}\bigg\} \wedge T\in {\cal T}_t.
\end{align}
Theorem \ref{thm:explore_stop_conv} shows that for any sequences $(N_n,\lambda_n)_{n\in\mathbb{N}}$ such that $N_n\uparrow \infty$ and $\lambda_n \downarrow 0$ as $n\to \infty$, ${\tau}^{x;N_n,\lambda_n}_t$ converges to the optimal stopping time of $V^x_t$ in \eqref{eq:intro_OST_g}. Moreover, the corresponding $g$-expectation value converges to $V_t^x$. In other words, under sufficiently large $N$ and sufficiently small $\lambda$, the stopping time ${\tau}^{x;N,\lambda}_t $ in \eqref{eq:intro_optimal_explore_stopping}, induced by the optimal Bernoulli-distributed control in \eqref{eq:intro_optimal_explore}, approximates the optimal stopping time of $V_t^x$. Hence, we refer to ${\tau}^{x;N,\lambda}_t $ as a {\it robust exploratory stopping time}.

It is noteworthy that a logistic form for exploratory controls as in \eqref{eq:intro_optimal_explore} already appears in the non-robust exploratory stopping problem in \cite{dai2024learning}. However, the natural question---namely, to what extent the optimal stopping time of the original stopping problem can be approximated from the exploratory controls---remains open therein. Theorem~\ref{thm:explore_stop_conv} provides an answer to this question, even in a robust framework. In particular, our result also applies to the non-robust setting. 

\color{black}
Last, under the same assumptions on $X^x,r,R$, and $\beta$, Theorem \ref{thm:policy_improvement} establishes a policy iteration result. Specifically, at each step we evaluate the $g$-expectation value function under the control $\pi \in \Pi$ from the previous iteration and then update the control in the logistic form driven by this evaluated $g$-expectation value as in~\eqref{eq:intro_optimal_explore}. This iterative process ensures that the resulting sequence of value functions and controls converge to the solution of \eqref{eq:intro_explore_robust_problem} as the number of iterations goes to~infinity.  

As an application of Theorem \ref{thm:policy_improvement}, 
we devise an RL algorithm (see Algorithm\;\ref{alg:exact_policy_iteration}) in which policy evaluation at each iteration, characterized by a PDE (see Corollary~\ref{coro:Markovian_of_policy_iter_closed_loop}),  can be implemented by the deep splitting method of Beck et al.\;\cite{Becker-Jentzen-Neufeld-Deep-Splitting-21}. 

Using our RL algorithm, Section~\ref{sec:experiment} presents several numerical examples, including robust American put- and call-option valuation. In the one-dimensional setting, we observe policy improvement and convergence under several ambiguity degrees. Moreover, stability analysis for our exploratory BSDEs solution is conducted with respect to ambiguity degree~$\varepsilon $, exploration level $\lambda $ and penalization level $ N $ in the put-type stopping problem, while robustness is illustrated in the call-type problem under different levels of dividend rate misspecification. Finally, we further investigate the policy improvement and convergence for a put-type Basket option valuation, and robustness for a call-type geometric option in multi-dimensional settings, thereby supporting the scalability of our algorithm. 

Beyond the valuation of American options, our robust exploratory stopping framework and RL algorithm can be applied in a broad class of optimal stopping problems under ambiguity, for instance to determine early termination rules in clinical trials under unknown treatment response distributions (e.g., \cite{pocock1977group}), to analyze Gittins-type indices in multi-armed bandit problems under unknown transition dynamics (e.g., \cite{gittins2011multi}), and to inform maintenance decisions such as repair or replacement of deteriorating systems under imperfectly observed failure signals (e.g., \cite{pierskalla1976survey}), among others.

Our main contributions can be summarized as follows. Beyond the standard BSDE formulation of the robust exploratory control framework, we uncover a structural link between robust exploratory controls and optimal stopping times under ambiguity, and prove---via a priori BSDE estimates and a convergence of monotone hitting times---that the robust exploratory stopping time approximates the optimal stopping time under ambiguity. Moreover, we establish a policy improvement theorem in a general setting that allows unbounded payoffs and arbitrary initial policy, while accommodating an ambiguity-induced driver in the generator; see Remarks \ref{rem:policy_iteration_contribution},\;\ref{rem:markovian_closure_PI} for comparisons with related works \cite{dai2024learning,DongRLstopping}. Finally, in our RL algorithm, we propose a new policy evaluation scheme based on the deep splitting method, combined with a quadratic variation–based procedure using the observed data to handle the driver term induced by the $g$-expectation; see~\eqref{eq:loss_minimizat} in Section~\ref{sec:policy_iter}. Numerical experiments demonstrate the convergence, robustness, and scalability of the proposed algorithm.

\color{black}
\vspace{-0.5em}
\subsection{Related literature}
Sutton and Barto \cite{SuttonBarto1998} opened up the field of RL~which has since gained significant attention with notable applications 
\cite{Guo2025DeepSeekR1-nature,Mnih2015HumanLevel,autodriver,Silver2016Mastering,Silver2017MasteringWithout}. In continuous-time settings, \cite{wang2020reinforcement,wang2020continuous} introduced an RL framework based on exploratory or relaxed controls, motivating subsequent development of RL schemes~\cite{huang2025convergence,Policy-evaluation,Policy-gradient,small-q}, applications and extensions~\cite{DaiLearning2023,DaiRecursive20203,han2023choquet,exploratoryHJB,Liwu2024}. 
A substantial literature on robust RL exists in a discrete-time setting \cite{blanchet2023double, panaganti2022robust, RobustRL-model-mismatch, RobustRL-deep}, while a continuous-time robust RL remains less explored. \color{black}

Our robust exploratory stopping problem under ambiguity aligns with, and can be viewed as, a robust analog of \cite{dai2024learning,DongRLstopping}. 
Recently, an exploratory stopping-time framework based on a singular control formulation has also been proposed by \cite{dianetti2024exploratory}. 
A related recent work of \cite{frikha2025entropy} studies an exploratory optimal stopping framework under discrete stopping times but without ambiguity. Lastly,  we refer to \cite{becker2019deep,becker2021solving,reppen2025neural} for machine learning (ML) approaches to optimal stopping.

Moving away from the continuous-time RL (or ML) results to the literature on continuous-time optimal stopping under ambiguity, we refer to \cite{bayraktar2011optimal_a,bayraktar2011optimal_b,nutz2015optimal,park2023irreversible,PW23,riedel2009optimal}. 
More recently, \cite{tankov2025optimalstoppingdivestmenttiming} proposes a framework for optimal stopping that incorporates both ambiguity and learning. Rather than adopting a worst-case approach, it employs the smooth ambiguity framework of Klibanoff et al.~\cite{klibanoff2005smooth} and the Bayesian learning. 
\subsection{Notations and preliminaries}\label{sec:notation}
Fix $d\in \mathbb{N}$. We endow $\mathbb{R}^d$ and $\mathbb{R}^{d\times d}$ with the Euclidean inner product $\langle \cdot,\cdot \rangle$ and the Frobenius inner product $\langle \cdot,\cdot\rangle_{\operatorname{F}}$, respectively. Moreover, we denote by $|\cdot|$ the Euclidean norm and denote by $\|\cdot\|_{\operatorname{F}}$ the Frobenius~norm.

Let $(\Omega,{\cal F}, \mathbb{P})$ be a probability space and let $B:=(B_t)_{t\geq 0}$ be a $d$-dimensional standard Brownian motion starting with $B_0=0$. Fix $T>0$ a finite time horizon, and let $\mathbb{F}:=({\cal F}_t)_{t\in [0,T]}$ be the usual augmentation of the natural filtration generated by $B$, i.e., ${\cal F}_t:=\sigma(B_s;s\leq t)\vee {\cal N}$, where ${\cal N}$ is the set of all $\mathbb{P}$-null subsets. 

For any probability measure $\mathbb{Q}$ on $(\Omega,{\cal F})$, we write $\mathbb{E}^{\mathbb{Q}}[\cdot]$ for the expectation under~$\mathbb{Q}$ and $\mathbb{E}^{\mathbb{Q}}_t[\cdot]:=\mathbb{E}^{\mathbb{Q}}[\cdot|{\cal F}_t]$ for the conditional expectation under $\mathbb{Q}$ with respect to ${\cal F}_t$ at time $t\geq 0$. Moreover, we set $\mathbb{E}[\cdot]:=\mathbb{E}^{\mathbb{P}}[\cdot]$ and $\mathbb{E}_t[\cdot]:=\mathbb{E}^{\mathbb{P}}_t[\cdot]$ for $t\geq 0$. For any $p\geq 1$, $k\in \mathbb{N}$ and $t\in[0,T]$, consider the following sets:
\begin{itemize}
	\item  $L^p({\cal F}_t;\mathbb{R}^k)$ is the set of all $\mathbb{R}^k$-valued, ${\cal F}_t$-measurable random variables $\xi$ such that $\|\xi \|_{L^p}^p:=\mathbb{E}[|\xi|^p]<\infty$; 
	\item  $\mathbb{L}^p(\mathbb{R}^k)$ is the set of all $\mathbb{R}^k$-valued, $\mathbb{F}$-predictable processes $Z=(Z_t)_{t\in[0,T]}$ such that $\|Z\|^p_{\mathbb{L}^p}:=\mathbb{E}[\int_0^T|Z_t|^pdt]<\infty$;  
	\item  
    $\mathbb{S}^p(\mathbb{R}^k)$ is the set of all $\mathbb{R}^k$-valued, $\mathbb{F}$-progressively measurable processes $Y=(Y_t)_{t\in[0,T]}$ such that $\|Y\|_{\mathbb{S}^p}^p:=\mathbb{E}[\sup_{t\in[0,T]}|Y_t|^p]<\infty$; 
	\item  ${\cal T}_{t}$ is the set of all $\mathbb{F}$-stopping times $\tau$ with values in $[t,T]$.  
\end{itemize}

\section{Optimal stopping under ambiguity}\label{sec:robust_stopping}
Consider the optimal stopping of an agent facing ambiguity, where the agent is {\it ambiguity-averse} and his/her stopping time is determined by observing an {ambiguous} underlying state process in a continuous-time environment. We model the agent's preference and the environment by using the $g$-expectation ${\cal E}^g[\cdot]$ (see \cite{coquet2002filtration,peng1997backward}) defined as follows.
\begin{definition}\label{def:g_expect}
	Let the driver term $g:\Omega\times[0,T]\times \mathbb{R}^d\to \mathbb{R}$ be a mapping such that the following conditions hold: 
	\begin{itemize}
		\item [(i)] for $z\in \mathbb{R}^d$, $(g(t,z))_{t\in[0,T]}$ is $\mathbb{F}$-progressively measurable with $\|g(\cdot,z)\|_{\mathbb{L}^2}<\infty$;
		\item [(ii)] there exists some constant $\kappa> 0$ such that for every $(\omega,t)\in \Omega\times[0,T]$ and $z,z'\in \mathbb{R}^d$ $\big|g(\omega,t,z)-g(\omega,t,z')\big| \leq \kappa |z-z'|;$
		\item [(iii)] for every $(\omega,t)\in \Omega\times[0,T]$, $g(\omega,t,\cdot):\mathbb{R}^d\to\mathbb{R}$ is concave, $g(\omega,t,0)=0$, 
        and satisfies that $g(\omega,t,cz)= cg(\omega,t,z)$ for all $z\in \mathbb{R}^d$ and $c\in(0,1)$.\color{black}
	\end{itemize}
	Then we define ${\cal E}^{g}: L^2({\cal F}_T;\mathbb{R})\ni \xi\to {\cal E}^g[\xi]\in \mathbb{R}$~as ${\cal E}^{g}[\xi]:=Y_0,$
	where $(Y,Z)\in \mathbb{S}^2(\mathbb{R})\times \mathbb{L}^2(\mathbb{R}^d)$ is the unique solution of the following BSDE (see \cite[Theorem~3.1]{pardoux1990adapted}):
	\begin{align*}
		Y_t= \xi + \int_t^T g(s,Z_s) ds -\int_t^T Z_s dB_s,\quad t\in[0,T],
	\end{align*}
	where $(B_t)_{t\in [0,T]}$ is the fixed $d$-dimensional Brownian motion on $(\Omega,{\cal F},\mathbb{P})$. Moreover, its conditional $g$-expectation with respect to ${\cal F}_t$ 
    is defined by
	\(
	{\cal E}^{g}_t[\xi]:=Y_t\) for $t\in [0,T]$, 
	which can be extended into $\mathbb{F}$-stopping times $\tau\in {\cal T}_0$, i.e., ${\cal E}^{g}_{\tau}[\xi]:=Y_\tau.$ 
\end{definition}

\begin{remark}\label{rem:represent_g_exp}
    The $g$-expectation defined above coincides with a variational representation in the following sense (see \cite[Proposition 3.6]{el1997backward}, \cite[Proposition A.1]{ferrari2022optimal}): Define $\hat g: \Omega\times[0,T]\times \mathbb{R}^d\ni (\omega,t,\hat z)\to  \hat g(\omega,t,\hat z):=\sup_{z\in \mathbb{R}^d}\big(g(\omega,t,z)-\langle z, \hat z \rangle \big)\in \mathbb{R},$
	i.e., the convex conjugate function of $g(\omega,t,\cdot)$. Denote by ${\cal B}^g$ the set of all $\mathbb{F}$ progressively measurable processes $\vartheta=(\vartheta_t)_{t\in[0,T]}$ such that $\|\hat g(\cdot,\vartheta_\cdot )\|_{\mathbb{L}^2}<\infty$.  
	
	For any $\tau \in {\cal T}_t$ and $t\in[0,T]$, the following representation holds:
	\begin{align*}
		{\cal E}_t^g[\xi] = \essinf_{\vartheta \in {\cal B}^g} \mathbb{E}_t^{\mathbb{P}^\vartheta}\bigg[\xi+\int_t^\tau \hat g(s, \vartheta_s)ds\bigg]\quad \mbox{for}\;\; \xi \in L^2({\cal F}_{\tau};\mathbb{R}^d), 
	\end{align*} 
	where $\mathbb{P}^\vartheta$ is defined on $(\Omega,{\cal F}_T)$ through $\frac{d \mathbb{P}^\vartheta}{d \mathbb{P}|_{{\cal F}_T}}:= \exp(-\frac{1}{2}\int_0^T|\vartheta_s|^2ds+\int_0^T\vartheta_s dB_s).$
\end{remark}

For (sufficiently integrable) $\mathbb{F}$-predictable processes $(b_s^o)_{s\in[0,T]}$ and $(\sigma_s^o)_{s\in[0,T]}$ taking values in $\mathbb{R}^d$ and $\mathbb{R}^{d\times d}$ respectively, we consider an It\^o $(\mathbb{F},\mathbb{P})$-semimartingale 
\begin{align}\label{eq:refer_semi}
	X_t^{x}:=
	x+ \int_0^tb^o_s ds +\int_0^t \sigma^o_sdB_s,\quad  t\in[0,T],
\end{align}
where $x\in \mathbb{R}^d$ is fixed and does not depend on $b^o$ and $\sigma^o$.

We note that $b^o$ and $\sigma^o$ correspond to the baseline parameters (e.g., the estimators) and $X^{x}$ corresponds to the reference underlying state process. We assume the certain integrability condition on the baseline parameters. To that end, for any $p\geq 1$, let $\mathbb{L}^p(\mathbb{R}^d)$ be defined as in Section \ref{sec:notation} and let $\mathbb{L}_{\operatorname{F}}^p(\mathbb{R}^{d\times d})$ be the set of all $\mathbb{R}^{d\times d}$-valued, $\mathbb{F}$-predictable processes $H=(H_t)_{t\in[0,T]}$ such that $\|H\|^p_{\mathbb{L}_{\operatorname{F}}^p}:=\mathbb{E}[(\int_0^T\|H_t\|_{\operatorname{F}}^2dt)^{\frac{p}{2}}]<\infty$.
\begin{as}\label{as:refer_semi}
	$b^o\in \mathbb{L}^p(\mathbb{R}^d)$ and $\sigma^o\in \mathbb{L}_{\operatorname{F}}^p(\mathbb{R}^{d\times d})$ for some $p\geq 2$.
\end{as}
\begin{remark}\label{rem:suff_refer_semi}
	Either one of the following conditions is sufficient for Assumption \ref{as:refer_semi} to hold true \cite[Lemma\;2.3]{bartl2023sensitivity}:
	\begin{itemize}
		\item [(i)] $b^o$ and $\sigma^o$ are uniformly bounded, i.e., there exists some constant $C_{b,\sigma}>0$ such that $|b^o_t|+\|\sigma_t^o\|_{\operatorname{F}}\leq C_{b,\sigma}$ $\mathbb{P}\otimes dt$-a.e..
		\item [(ii)] 
		$b^o$ and $\sigma^o$ are of the following form: $b_t^o= \widetilde{b}^o(t,X_t^{x}),$ $\sigma_t^o= \widetilde{\sigma}^o(t,X_t^{x})$ $\mathbb{P}\otimes dt$-a.e.,
		where
		$\widetilde{b}^o:[0,T]\times \mathbb{R}^d\rightarrow \mathbb{R}^d$ and $\widetilde{\sigma}^o:[0,T] \times \mathbb{R}^d\rightarrow \mathbb{R}^{d\times d}$ are Borel  functions satisfying that $\lvert \widetilde{b}^o(t,y) - \widetilde{b}^o(t,\hat y) \lvert + \lVert \widetilde{\sigma}^o(t,y)- \widetilde{\sigma}^o(t,\hat y) \rVert_{\operatorname{F}} 
		\leq C_{\widetilde{b},\widetilde{\sigma}} \lvert y-\hat y \rvert$ and $\lvert \widetilde{b}^o(t,y) \lvert + \lVert \widetilde{\sigma}^o(t,y) \rVert_{\operatorname{F}} 
		\leq C_{\widetilde{b},\widetilde{\sigma}} (1+\lvert y\rvert)$ for every $t\in[0,T]$ and $y,\hat y\in \mathbb{R}^d$, 
		with some constant $C_{\widetilde{b},\widetilde{\sigma}}>0$.
	\end{itemize}
\end{remark}

\begin{remark}\label{rem:refer_semi} 
	\begin{itemize}
		\item [(i)] Under Assumption \ref{as:refer_semi}, a straightforward application of  the Burkholder Davis Gundy (BDG) inequality shows that $\|X^x\|_{\mathbb{S}^p}<\infty$. 
		\item [(ii)] In fact, both sufficient conditions given in Remark \ref{rem:suff_refer_semi} ensure that Assumption~\ref{as:refer_semi} holds for all $p\geq 2$ (see \cite[Theorems 2.3.1 and 2.4.1]{mao2007stochastic})
	\end{itemize}
\end{remark}

Having completed the descriptions of the $g$-expectation and underlying process, we describe the decision-maker's optimal stopping problem $V^x:=(V_t^x)_{t\in[0,T]}$ under ambiguity: for every $t\in [0,T]$, 
\begin{align}\label{eq:robust_stopping}
    \begin{aligned}
	&V_t^{x}:= \esssup_{\tau \in {\cal T}_t}{\cal E}^{g}_t[\operatorname{I}_t^{x;\tau}],\quad\operatorname{I}_t^{x;\tau}:=\int_t^\tau e^{-\int_t^s\beta_udu} r(X_s^{x})ds + e^{-\int_t^\tau \beta_udu}R(X_\tau^{x}),
    \end{aligned}
\end{align}
where both $r:\mathbb{R}^d\to \mathbb{R}$ and  $R:\mathbb{R}^d\to \mathbb{R}$
are some Borel functions (representing the intermediate and stopping reward functions), and $(\beta_u)_{u\in [0,T]}$ is an $\mathbb{F}$-progressively measurable process taking positive values (representing the subjective discount rate). 

\begin{as}\label{as:reward} The following conditions on $R,r$ and $\beta$ hold:
	\begin{itemize}
		\item [(i)] $R$ is continuous. Moreover, there exists some constant $C_{r,R}>0$ such that for every $y\in \mathbb{R}^d$, $|r(y)|+|R(y)|\leq C_{r,R}(1+|y|)$.
		\item [(ii)] There is some $C_\beta>0$ such that $0\leq \beta_t(\omega) \leq C_\beta$ for all $(\omega,t)\in\Omega\times [0,T]$.
	\end{itemize}
\end{as}

\begin{remark}\label{rem:wellposed_stopping}
    Under Assumptions \ref{as:refer_semi} and \ref{as:reward}, it holds for every $t\in [0,T]$ and $\tau\in {\cal T}_t$ that the integrand $\operatorname{I}_t^{x;\tau}$ given in \eqref{eq:robust_stopping} is in $L^2({\cal F}_\tau;\mathbb{R})$. Indeed, by the triangle inequality and the positiveness of $(\beta_u)_{u\in [0,T]}$, \(\mathbb{E}[|\operatorname{I}_t^{x;\tau}|]
    \leq C_{r,R}(T+1) \|X^x\|_{\mathbb{S}^1};\)
	see also Assumption \ref{as:reward}. Moreover, since $\|X^x\|_{\mathbb{S}^p}<\infty$ with the exponent $p\geq 2$ (see Remark~\ref{rem:refer_semi}\;(i)), an application of the Jensen's inequality with exponent $2$ ensures the claim to hold. As a direct consequence, 
    $V^x$ in \eqref{eq:robust_stopping} is well-defined.
\end{remark}

{\color{blue} \color{black} In the following, we remark} that $V^x$ given in \eqref{eq:robust_stopping} corresponds to a reflected BSDE with an obstacle. To this end, set for every  $(\omega,t,y,z)\in\Omega\times [0,T]\times \mathbb{R}\times \mathbb{R}^d$ by
\begin{align}\label{eq:benchmark_gen}
F_t^x(\omega,y,z):= r(X_t^x(\omega)) - \beta_t(\omega)y+ g(\omega,t,z),
\end{align}
where $g:\Omega\times[0,T]\times \mathbb{R}^d\to \mathbb{R}$ is defined as in Definition \ref{def:g_expect}, $(X_t^x)_{t\in[0,T]}$ is given in\;\eqref{eq:refer_semi}, and $(\beta_t)_{t\in[0,T]}$ is the discount rate appearing in~\eqref{eq:robust_stopping}.

Denote by $(Y^x,Z^x,K^x)$ a triplet of processes 
satisfying that 
\begin{align}\label{eq:reflected_gBSDE}
	Y_t^{x}= R (X_T^x) +  \int_t^T
	F_s^x(Y_s^x,Z_s^x)ds  -\int_t^T Z_s^x dB_s + K_T^x-K_t^x,\quad t\in[0,T].
\end{align}

We then introduce the notion of the reflected BSDE (see 
\cite[Section~2]{el1997reflected}). For this, recall the sets $\mathbb{S}^2(\mathbb{R})$ and $\mathbb{L}^2(\mathbb{R}^d)$ given in Section \ref{sec:notation}.
\begin{definition}\label{def:reflected_gBSDE}
	A triplet $(Y^x,Z^x,K^x)$ is said to be a solution
	to the reflected BSDE \eqref{eq:reflected_gBSDE} with the obstacle $(R(X_t^x))_{t\in[0,T]}$ if the following conditions hold:  
	\begin{itemize}
		\item [(i)] $Y^x\in \mathbb{S}^2(\mathbb{R})$, $Z^x\in \mathbb{L}^2(\mathbb{R}^d)$, and $K^x\in \mathbb{S}^2(\mathbb{R})$ is a nondecreasing, continuous process with $K_0^x=0$. Moreover, $(Y^x,Z^x,K^x)$ satisfies \eqref{eq:reflected_gBSDE};
		\item [(ii)]  $Y_t^x\geq R(X_t^x)$ $\mathbb{P}$-a.s., for all $t\geq 0$;
        \item [(iii)] $\int_0^T(Y_{t}^x-R(X_{t}^x))dK_t^x=0$ $\mathbb{P}$-a.s..
	\end{itemize}
\end{definition}

\begin{remark}\label{rem:wellposed_reflected_gBSDE}
    Under Assumptions \ref{as:refer_semi} and \ref{as:reward}, 
    the parameters of the reflected BSDE \eqref{eq:reflected_gBSDE} with the obstacle $(R(X_t^x))_{t\in[0,T]}$ satisfy the conditions 
    (i)--(iv) given in \cite[Section~2]{el1997reflected}.  
    This enables to apply \cite[Theorem~5.2]{el1997reflected} to ensure that there exists a unique solution $(Y^x,Z^x,K^x)$ of the reflected BSDE; see Definition \ref{def:reflected_gBSDE}. 
    
\end{remark}

The following proposition establishes that the solution to the reflected BSDE \eqref{eq:reflected_gBSDE} coincides with the Snell envelope of the optimal stopping problem under ambiguity given in \eqref{eq:robust_stopping}. This result can be seen as a robust analogue of \cite[Proposition\;2.3]{el1997reflected} and \cite[Proposition\;3.1]{lepeltier2005penalization}. Several properties of (conditional) $g$-expectation developed in \cite{coquet2002filtration} are useful in the proof presented in Section \ref{proof:pro:verification}.
\begin{proposition}\label{pro:verification}
	Suppose that Assumptions \ref{as:refer_semi} and \ref{as:reward} hold. Let $V^x$ be given in \eqref{eq:robust_stopping} (see Remark \ref{rem:wellposed_stopping}) and let $Y^x$ be the first component of the unique solution to the reflected BSDE \eqref{eq:reflected_gBSDE} with the obstacle $(R(X_t^x))_{t\in[0,T]}$ (see Remark \ref{rem:wellposed_reflected_gBSDE}). Then, we have that $V_t^x=Y_t^x$ $\mathbb{P}$-a.s.~for all $t\in[0,T]$. In particular, the stopping~time 
	\begin{align}\label{eq:robust_OST}
		\tau_{t}^{*,x}:= \inf\{s\geq t:Y_s^x\leq R(X_s^x)\}\wedge T\in {\cal T}_t 
	\end{align}
	is optimal to $V_t^x$, i.e., ${\cal E}^{g}_t[\operatorname{I}_t^{x;\tau^{*,}_t}]=V_t^x$ $\mathbb{P}$-a.s..
\end{proposition}

In the remainder of this section, we introduce auxiliary robust stochastic control problems arising from the penalization method which is used to establish the existence
of solutions to reflected BSDEs (see, e.g., \cite{el1997backward,lepeltier2005penalization,peng2005smallest}). 

	To that end, define for every $N\in \mathbb{N}$ and $(\omega,t,y,z)\in\Omega\times [0,T]\times \mathbb{R}\times \mathbb{R}^d$ by
    \begin{align}\label{eq:benchmark_gen_N}
        F_t^{x;N}(\omega,y,z):= F_t^{x}(\omega,y,z)+ N\big(R(X_t^x(\omega))-y\big)^+,
    \end{align}
    where $F^x$ is given in \eqref{eq:benchmark_gen} and $(a)^+:=\max\{a,0\}$ for $a\in \mathbb{R}$. Then we denote for every $N\in \mathbb{N}$ by $(Y^{x;N},Z^{x;N})$ a couple of processes satisfying
	\begin{align}\label{eq:g_BSDE}
		Y_t^{x;N}=R (X_T^x) +  \int_t^T
		F_s^{x;N}(Y_s^{x;N},Z_s^{x;N})
		ds-\int_t^T Z_s^{x;N} dB_s,\quad \mbox{$t\in[0,T]$}.
	\end{align}

    \begin{remark}\label{rem:penalized_BSDE}
        Under Assumptions \ref{as:refer_semi} and \ref{as:reward}, the parameters of the BSDE \eqref{eq:g_BSDE} satisfy all the conditions given in \cite[Section 3]{pardoux1990adapted}. Hence, the following hold:
        \begin{itemize}
            \item [(i)] For every $N\in \mathbb{N}$, there exists a unique solution $(Y^{x;N},Z^{x;N})\in \mathbb{S}^2(\mathbb{R})\times \mathbb{L}^2(\mathbb{R}^d)$ of the BSDE \eqref{eq:g_BSDE}; see \cite[Theorem 3.1]{pardoux1990adapted}. 
            \item [(ii)] If we set $K_t^{x;N}:=N\int_0^t(R(X_s^x)-Y_s^{x;N})^+ds$ for $t\in[0,T]$, then it follows from 
            \cite[Section\,6., Eq.\,(16)]{el1997reflected} that there exists some constant $C>0$ such that for every $N\in \mathbb{N}$, $
			\|Y^{x;N}\|_{\mathbb{S}^2}^2+\|Z^{x;N}\|_{\mathbb{L}^2}^2 + \|K_T^{x;N}\|_{L^2}^2 \leq C.$
            \item [(iii)] We recall that $(Y^x,Z^x,K^x)$ is the unique solution to \eqref{eq:reflected_gBSDE}; see Remark~\ref{rem:wellposed_reflected_gBSDE}. We have by \cite[Section 6]{el1997reflected} that as $N\to \infty$, 
            \[
            \|Y^{x;N}-Y^{x}\|_{\mathbb{S}^2}+\|Z^{x;N}-Z^x\|_{\mathbb{L}^2}+\|K^{x;N}-K^x\|_{\mathbb{S}^2}\to 0.
            \]
            In particular, it holds that $ Y_t^{x;N} \uparrow Y_t^{x}$ $\P $-a.s.\;for all $t\in[0,T]$
        \end{itemize}
    \end{remark}

    The following proposition shows that for each $N\in \mathbb{N}$ the solution to the penalized BSDE \eqref{eq:g_BSDE} can be represented by a robust stochastic control problem under ambiguity. The corresponding proof is presented in Section \ref{proof:pro:verification}.
	\begin{proposition}\label{pro:verification_control}
		Suppose that Assumptions \ref{as:refer_semi} and \ref{as:reward} hold. Let $N\in \mathbb{N}$ be given. Denote by $Y^{x;N}$ the first component of the unique solution to \eqref{eq:g_BSDE}. Then $Y^{x;N}$  admits a representation of the robust stochastic control problem in the following sense:
		Let ${\cal A}$ be the set of all $\mathbb{F}$-progressively measurable processes $\alpha=(\alpha_t)_{t\in[0,T]}$ with values in $\{0,1\}$. Set for every $t\in[0,T]$ and $N\in \mathbb{N}$
        \[
            \operatorname{I}_t^{x;N,\alpha}:=\int_t^Te^{-\int_t^s(\beta_u+N\alpha_u)du}\big(r(X_s^x)+R(X_s^x)\, N\alpha_s \big)ds+e^{-\int_t^T(\beta_u+N\alpha_u)du} R(X_T^x).
        \]
        Then it holds for every $t\in [0,T]$ that $Y_t^{x;N}=\esssup_{\alpha \in {\cal A}} {\cal E}_t^g[\operatorname{I}_t^{x;N,\alpha}]={\cal E}_t^g[\operatorname{I}_t^{x;N,\alpha^{*,x;N}}],$ $\mathbb{P}$-a.s.,~where $\alpha^{*,x;N}:=(\alpha^{*,x;N}_t)_{t\in[0,T]}\in {\cal A}$ is the optimizer given by 
		\begin{align}\label{eq:const_alpha_0}
			\alpha^{*,x;N}_t := {\bf 1}_{\{R(X_t^x)>Y_t^{x;N}\}}\quad \mbox{for $t\in[0,T]$}.
		\end{align}
	\end{proposition}

    \section{Exploratory framework: approximation of optimal stopping under ambiguity}\label{subsec:robust_exploratory}
	In the previous section, we show that 
	for sufficiently large~$N\in \mathbb{N}$, the optimal stopping problem $V^x(=Y^x)$ under ambiguity in \eqref{eq:robust_stopping} can be approximated by the robust stochastic control problem $Y^{x;N}$ under ambiguity. 
    
    In this section, we introduce an exploratory version of $Y^{x;N}$, following  \cite{wang2020reinforcement,wang2020continuous}. In particular, we aim to study a robust analogue of the optimal exploratory stopping framework in \cite{dai2024learning}. To that end, let $\Pi$ be the set of all $\mathbb{F}$-progressively measurable processes $\pi=(\pi_t)_{t\in[0,T]}$ taking values in $[0,1]$, i.e., an exploratory version of the $ \{0,1\}$-valued controls set ${\cal A}$ appearing in Proposition \ref{pro:verification_control}.
	
    Then we define ${\cal H}:[0,1]\ni a \mapsto {\cal H}(a)\in \mathbb{R}$ by
    \begin{align}\label{eq:entropy}
    {\cal H}(a):= a \log(a)+(1-a)\log(1-a),
    \end{align}
    with the convention that ${\cal H}(0):=\lim_{a\downarrow 0}{\cal H}(a)=0$ and ${\cal H}(1):=\lim_{a\uparrow 1}{\cal H}(a)=0$.
    Typically, $-{\cal H}(a)$ is called the binary differential entropy. \color{black}
    
    Finally, let $\lambda>  0$ denote the temperature parameter reflecting the trade-off between exploration and exploitation.
	
We can then describe the robust exploratory control problem 
    under ambiguity: 
	for any $N\in \mathbb{N}$ and~$\lambda> 0$:
	\begin{align}\label{eq:robust_explore_problem}
	\overline V_t^{x;N,\lambda}:=\esssup_{\pi \in \Pi} {\cal E}_t^g[\overline{\operatorname{J}}_t^{x;N,\lambda,\pi}],\quad \mbox{for $t\in[0,T]$},
	\end{align}
	where for each $\pi \in \Pi$, the integrand $\overline{\operatorname{J}}_t^{x;N,\lambda,\pi}$ is given by
	\begin{align*}
		\overline{\operatorname{J}}_t^{x;N,\lambda,\pi}:=&\int_t^Te^{-\int_t^s(\beta_u+N\pi_u)du}\big(r(X_s^x)+R(X_s^x)\, N\pi_s -\lambda {\cal H}(\pi_s) \big)ds\\
        &\quad+e^{-\int_t^T(\beta_u+N\pi_u)du} R(X_T^x),
	\end{align*}
	where $X^x$ is given in \eqref{eq:refer_semi} and $(\beta_t)_{t\in[0,T]}$ is the discount rate appearing in\;\eqref{eq:robust_stopping}. 
	
	\begin{remark}\label{rem:entropy}
		We note that ${\cal H}$ in \eqref{eq:entropy} is strictly convex and bounded on~$[0,1]$. Moreover, since every $\pi\in \Pi$ is uniformly bounded by~$[0,1]$, by the same arguments presented for Remark \ref{rem:wellposed_stopping}, we have that $\overline{\operatorname{J}}_t^{x;N,\lambda,\pi}\in L^2({\cal F}_T;\mathbb{R})$ for all $N\in\mathbb{N}$, $\lambda> 0$, and $\pi\in \Pi$. Therefore, $\overline V^{x;N,\lambda}$ given in \eqref{eq:robust_explore_problem} is well-defined for all $N\in\mathbb{N}$ and $\lambda> 0$.
	\end{remark}
    \begin{remark}\label{rem:bernoulli}
	Assume that the probability space $(\Omega,{\cal F},\mathbb{P})$ supports a uniformly distributed random variable $U$ with values in $[0,1]$ which is independent of the fixed Brownian motion $B$. Then we are able to see that every exploratory control $\pi\in \Pi$ generates a Bernoulli-distributed (randomized) process under {\it drift ambiguity}. Indeed, we recall the variational characterization of $g$-expectation in Remark~\ref{rem:represent_g_exp} with the map $\hat g:\Omega \times[0,T]\times \mathbb{R}^d\to \mathbb{R}$ and the set ${\cal B}^g$. Then, for all $N\in\mathbb{N}$, $\lambda> 0$, and $t\in[0,T]$, we can rewrite the conditional $g$-expectation value ${\cal E}_t^g[\overline{\operatorname{J}}_t^{x;N,\lambda,\pi}]$ given in~\eqref{eq:robust_explore_problem} as the following {\it strong formulation} for drift ambiguity under $\mathbb{P}$ (see \cite[Section~5]{bartl2024numerical}):
        \begin{align}\label{eq:strong_repre}
            {\cal E}_t^g[\overline{\operatorname{J}}_t^{x;N,\lambda,\pi}]=\essinf_{\vartheta\in {\cal B}^g}\mathbb{E}_t\bigg[\overline{\operatorname{J}}_t^{x;N,\lambda,\pi,\vartheta}+\int_t^T \hat g(s, \vartheta_s)ds\bigg],
        \end{align}
        where 
        for each $\pi\in \Pi$ and $\vartheta\in {\cal B}^g$, the term $\overline{\operatorname{J}}_t^{x;N,\lambda,\pi,\vartheta}$ is given by 
        \begin{align*}
		\overline{\operatorname{J}}_t^{x;N,\lambda,\pi,\vartheta}:=&\int_t^Te^{-\int_t^s(\beta_u+N\pi_u)du}\big(r(X_s^{x;\vartheta})+R(X_s^{x;\vartheta})\, N\pi_s -\lambda {\cal H}(\pi_s) \big)ds\\
        &\quad+e^{-\int_t^T(\beta_u+N\pi_u)du} R(X_T^{x;\vartheta}),
	\end{align*}
        where $(X^{x;\vartheta}_t)_{t\in[0,T]}$ is 
        given~by $X_t^{x;\vartheta}:=
	x+ \int_0^t\big(b^o_s+ \sigma_s^o \vartheta_s \big) ds +\int_0^t \sigma^o_sdB_s,$ for $t\in[0,T]$, 
        and $(b^o,\sigma^o)$ are the baseline parameters appearing in \eqref{eq:refer_semi}. 

    Then by using the random variable $U$ and its independence with the filtration $\mathbb{F}$ generated by $B$, we can apply the Blackwell--Dubins lemma (see \cite{blackwell1983extension}) to ensure that there exists a randomized process $(\widetilde \alpha_t)_{t\in[0,T]}$ such that for every $t\in[0,T]$, $\mathbb{P}$-a.s.,
    \[
    \mathbb{P}(\widetilde\alpha_t=1\,|\,{\cal F}_t) = \pi_t =1 - \mathbb{P}(\widetilde\alpha_t=0\,|\,{\cal F}_t),
    \]
    i.e., $\widetilde \alpha_t$ is  a Bernoulli distributed random variable with probability $\pi_t$ given ${\cal F}_t$.
    \end{remark}
    
    In order to characterize $\overline V^{x;N,\lambda}$ given in \eqref{eq:robust_explore_problem}, we first collect several preliminary results concerning the following auxiliary BSDE formulations: Recall that $F^x$ is given in~\eqref{eq:benchmark_gen}. Denote for every $\pi \in \Pi$ and $(\omega,t,y,z)\in \Omega \times[0,T]\times \mathbb{R}\times \mathbb{R}^d$ by
		\begin{align}\label{eq:control_gen}
		\overline F^{x;N,\lambda,\pi}_t(\omega,y,z):=F^{x}_t(\omega,y,z)+N(R\big(X_t^x(\omega)\big)-y)\pi_t(\omega)-\lambda {\cal H}(\pi_t(\omega)). 
		\end{align}
    Then, consider the controlled processes $(\overline Y^{x;N,\lambda,\pi},\overline Z^{x;N,\lambda,\pi})$ satisfying 
	\begin{align}\label{eq:g_BSDE_explore}
		\overline Y_t^{x;N,\lambda,\pi}=R (X_T^x) + \int_t^T
		\overline F_s^{x;N,\lambda,\pi}(\overline Y_s^{x;N,\lambda,\pi},\overline Z_s^{x;N,\lambda,\pi})
		ds-\int_t^T \overline Z_s^{x;N,\lambda,\pi} dB_s,
	\end{align}

        \begin{remark}\label{rem:explore_control_BSDE}
            Under Assumptions \ref{as:refer_semi} and \ref{as:reward}, the following statements hold for all $\pi\in\Pi$, $N\in \mathbb{N}$ and $\lambda>0$:
            \begin{itemize}
                \item [(i)] Since $(\pi_t)_{t\in[0,T]}\in \Pi$ and $({\cal H}(\pi_t))_{t\in[0,T]}$ are uniformly bounded (see Remark\;\ref{rem:entropy}), we are able to see that the parameters of \eqref{eq:g_BSDE_explore} satisfy all the conditions given in \cite[Section 3]{pardoux1990adapted}. Therefore, there exists a unique solution $(\overline Y^{x;N,\lambda,\pi},\overline Z^{x;N,\lambda,\pi})\in \mathbb{S}^2(\mathbb{R})\times \mathbb{L}^2(\mathbb{R}^d)$ to \eqref{eq:g_BSDE_explore}. 
                \item [(ii)] Since $\overline Y_t^{x;N,\lambda,\pi}\in L^2({\cal F}_t;\mathbb{R})$ and $\overline{\operatorname{J}}_t^{x;N,\lambda,\pi}\in L^2({\cal F}_T;\mathbb{R})$ (see Remark~\ref{rem:entropy}), we can use the same arguments presented for Proposition \ref{pro:verification_control} to have that 
	\begin{align}\label{eq:g_exp_explore}
			\overline Y_t^{x;N,\lambda,\pi}={\cal E}_t^g[\overline{\operatorname{J}}_t^{x;N,\lambda,\pi}],\quad \mbox{$\mathbb{P}$-a.s. for all $t\in[0,T]$}.
	\end{align}
            \end{itemize}
        \end{remark}
    Moreover, define 
    for every $N\in \mathbb{N}$, $\lambda> 0$, and $(\omega,t,y,z)\in\Omega\times [0,T]\times \mathbb{R}\times \mathbb{R}^d$ by
    \begin{align}\label{eq:optimal_explore_gen}
			\begin{aligned}
				&\overline F_t^{x;N,\lambda}(\omega,y,z):=
				F_t^x(\omega,y,z)+ G_t^{x;N,\lambda}(\omega,y),\\
				&\;\;\mbox{where}\;\;G_t^{x;N,\lambda}(\omega,y):=N\Big(R\big(X_t^x(\omega)\big)-y\Big) +\lambda \log\Big(e^{-\frac{N}{\lambda}\{R(X_t^x(\omega))-y\}}+1\Big).
			\end{aligned}
    \end{align}
    Then consider  the couple of processes $(\overline Y^{x;N,\lambda},\overline Z^{x;N,\lambda})$ satisfying 
	\begin{align}\label{eq:g_BSDE_explore_optimal}
			\overline Y_t^{x;N,\lambda}=&R (X_T^x) + \int_t^T \overline F_s^{x;N,\lambda}(\overline Y_s^{x;N,\lambda},\overline Z_s^{x;N,\lambda})ds-\int_t^T \overline Z_s^{x;N,\lambda} dB_s.
	\end{align}
    
    The following theorem shows that the robust exploratory control problem~$\overline V^{x;N,\lambda}$ is characterized by the BSDE~\eqref{eq:g_BSDE_explore_optimal}. The proof is presented in Section \ref{proof:thm:verification}.
	\begin{theorem}\label{thm:verification}
		Suppose that Assumptions \ref{as:refer_semi} and \ref{as:reward} hold. Recall the logistic function $\operatorname{logit}(\cdot)$ in \eqref{eq:intro_optimal_explore}. 
        The following statements hold for every $N\in\mathbb{N}$ and $\lambda> 0$.
        \begin{itemize}
            \item [(i)] There exists a unique solution $(\overline Y^{x;N,\lambda},\overline Z^{x;N,\lambda})\in \mathbb{S}^2(\mathbb{R})\times \mathbb{L}^2(\mathbb{R}^d)$ of~\eqref{eq:g_BSDE_explore_optimal}.
            \item [(ii)] Moreover, recall $\overline V^{x;N,\lambda}$ is given in \eqref{eq:robust_explore_problem}. Then it holds  for every $t\in[0,T]$ that $\overline Y_t^{x;N,\lambda}=\overline V_t^{x;N,\lambda}={\cal E}_t^g[\overline{\operatorname{J}}_t^{x;N,\lambda,\pi^{*,x;N,\lambda}}]$ $\mathbb{P}$-a.s., where $\pi^{*,x;N,\lambda}\in \Pi$ is given~by
		\begin{align}\label{eq:robust_control}
			\pi^{*,x;N,\lambda}_t:= \operatorname{logit}\Big(\frac{N}{\lambda}(R(X_t^x)-\overline Y_t^{x;N,\lambda})\Big)
            ,\quad t\in[0,T].
		\end{align}
        \end{itemize}
	\end{theorem}

    The following theorem establishes the comparison and stability results between the exploratory and non-exploratory robust control problems characterized in Proposition \ref{pro:verification_control} and Theorem \ref{thm:verification}. The proof is presented in Section \ref{proof:thm:verification}.
	\begin{theorem}\label{thm:stability}
		Suppose that Assumptions \ref{as:refer_semi} and \ref{as:reward} hold. For each $N\in \mathbb{N}$ and $\lambda>0$, let $(Y^{x;N},Z^{x;N})$ and $(\overline Y^{x;N,\lambda},\overline Z^{x;N,\lambda})$ be the unique solution to the BSDEs \eqref{eq:g_BSDE} and \eqref{eq:g_BSDE_explore_optimal}, respectively. Then it holds that for every $N\in \mathbb{N}$ and $\lambda>0$, 
		\begin{align}\label{eq:compare}
		Y_t^{x;N} \leq \overline{Y}_t^{x;N,\lambda},\quad \mbox{$\mathbb{P}$-a.s., for all $t\geq 0$, }
		\end{align}
		In particular, there exists some constant $C>0$ (that does not depend on $N\in \mathbb{N}$ and $\lambda>0$ but on $T>0$) such that for every $N\in \mathbb{N}$ and $\lambda>0$,
		\begin{align}\label{eq:stability}
			\|Y^{x;N}-\overline{Y}^{x;N,\lambda} \|_{\mathbb{S}^2}+\|Z^{x;N}-\overline{Z}^{x;N,\lambda} \|_{\mathbb{L}^2} \leq C \lambda,
		\end{align}
		This implies that for any $N\in \mathbb{N}$, $\overline{Y}^{x;N,\lambda}$ strongly converges to $Y^{x;N}$ in $\mathbb{S}^2(\mathbb{R})$,~as~$\lambda \downarrow 0$. 
	\end{theorem}

    Having completed the characterization of the robust exploratory control problem, we now introduce a stopping time ${\tau}^{x;N,\lambda}_t \in {\cal T}_t$ defined, for every $N\in\mathbb{N}$ and $\lambda>0$, by
     \begin{align}\label{eq:explore_stopping}
        \begin{aligned}
       {\tau}^{x;N,\lambda}_t :=&\inf \bigg\{s \ge t : \pi^{*,x;N,\lambda}_s\geq \frac{1}{1+2^{N(T-s)}}\bigg\} \wedge T\\
       =& \inf \{s \ge t : \overline{Y}_s^{x;N,\lambda} - \lambda (T-s) \log2 \le R(X_s^x)\} \wedge T\quad t\in[0,T],
       \end{aligned}
    \end{align}
    where the equality holds by definition of the standard logistic function $\operatorname{logit}(\cdot)$. 
    We refer to it as a {\it robust exploratory stopping time}. 
    
    The following theorem shows that the robust exploratory stopping time ${\tau}^{x;N,\lambda}_t$ in \eqref{eq:explore_stopping} converges to the optimal stopping time $\tau_t^{*,x}$ in \eqref{eq:robust_OST} along any sequences $(N_n,\lambda_n)$ with $N_n\uparrow \infty$ and $\lambda_n \downarrow 0$. Moreover, it establishes that the corresponding $g$-expectation value under the robust exploratory stopping time converges to the optimal stopping value $V_t^x(=Y_t^x)$ in \eqref{eq:robust_stopping}. The proof can be found in Section \ref{proof:thm:verification}.

    \begin{theorem}
    \label{thm:explore_stop_conv}
    Suppose that Assumptions \ref{as:refer_semi} and \ref{as:reward} hold. 
    Let \((N_n)_{n\in\mathbb N}\subset\mathbb N\) and \((\lambda_n)_{n\in\mathbb N}\subset(0,\infty)\) be sequences satisfying 
    \(
    N_n\uparrow\infty\) and \(\lambda_n\downarrow0
    \) 
    as $n\to \infty$. Then it holds for every \(t\in[0,T]\) that as $n\to \infty$, 
    \[
    \tau_t^{x;N_n,\lambda_n}\to \tau_t^{*,x}\quad \mbox{$\mathbb P$-a.s.}, \quad\mbox{and}\quad \big\|{\cal E}_t^g[\operatorname{I}_t^{x;\tau_t^{x;N_n,\lambda_n}}]-V_t^x  \big\|_{L^2}\to 0. 
    \]
    \end{theorem}
    \begin{remark}\label{rem:convergence_issue}
        We note that the exploration benefit $\overline Y^{x;N,\lambda}-Y^{x;N}$ compensates for the penalization error $Y^x- Y^{x;N}$ (see \eqref{eq:compare} and Remark \ref{rem:penalized_BSDE}\;(iii)), which does not lead to a direct comparison between $\overline Y^{x;N,\lambda}$ and $Y^x$. Consequently, it is not immediate whether a stopping time defined, for every $N\in\mathbb{N}$ and $\lambda>0$, by
        \begin{align*}
            \overline{\tau}^{x;N,\lambda}_t :=& \inf \bigg\{s \ge t : \pi^{*,x;N,\lambda}_s \ge \frac{1}{2} \bigg\} \wedge T= 
            \inf \{s \ge t : \overline{Y}_s^{x;N,\lambda}\le R(X_s^x)\} \wedge T\in {\cal T}_t, 
        \end{align*} 
        converges to $\tau^{*;x}_t$ in \eqref{eq:robust_OST}, even along sequences $(N_n,\lambda_n)$ with $N_n\uparrow \infty$ and $\lambda_n \downarrow 0$.

        To construct a stopping time that converges to $\tau^{*,x}_t$, we first establish the bound $\overline Y_t^{x;N,\lambda}-Y_t^x\leq \lambda (T-t) \log 2$. This allows us to consider the value $\overline Y_t^{x;N,\lambda}-\lambda (T-t) \log 2$ which is dominated by $Y^x$ and also converges to $Y^x$ along $(N_n,\lambda_n)$ with $N_n\uparrow \infty$ and $\lambda_n \downarrow 0$. Then, by establishing the convergence of an increasing sequence of hitting times of continuous processes (see Lemma~\ref{lem:general_tau_conv}), we obtain in Theorem~\ref{thm:explore_stop_conv} that the robust exploratory stopping time $\tau_t^{x;N,\lambda}$ in \eqref{eq:explore_stopping} converges to $\tau^{*,x}_t$. This result also applies to the non-robust (classical) exploratory stopping framework of \cite{dai2024learning}. 
    \end{remark}
    \color{black}

    \section{Policy iteration theorem \& RL algorithm}\label{sec:policy_iter}
   A typical RL approach to finding the optimal strategy is based on \textit{policy iteration}, where the strategy is successively refined through iterative updates. In this section, we establish the policy iteration theorem based on the verification result in Theorem~\ref{thm:verification}, and then provide the corresponding reinforcement learning algorithm.
    
    Throughout this section, we fix a sufficiently large $N\in \mathbb{N}$ and a small $\lambda>0$ so that $\overline Y^{x;N,\lambda}$ serves as an approximation of $Y^{x}$ (see Remark\;\ref{rem:penalized_BSDE} and Theorem~\ref{thm:stability}). 
        
    For any $\pi^n\in \Pi$ and $n\in \mathbb{N}$, denote by $(\overline Y^{x;N,\lambda,\pi^{n}},\overline Z^{x;N,\lambda,\pi^{n}})\in \mathbb{S}^2(\mathbb{R})\times \mathbb{L}^2(\mathbb{R}^d)$ the unique solution of \eqref{eq:g_BSDE_explore} under the exploratory control $\pi^n$ (see Remark\;\ref{rem:explore_control_BSDE}\;(i)).
    Recall the logistic function $\operatorname{logit}(\cdot)$ in \eqref{eq:intro_optimal_explore}. 
    Then one can construct $\pi^{n+1}\in \Pi$ as 
        \begin{align}\label{eq:robust_iter_control}
    			\pi^{n+1}_t:= \operatorname{logit}\bigg(\frac{N}{\lambda}(R(X_t^x)-\overline Y_t^{x;N,\lambda,\pi^n})\bigg)
                ,\quad t\in[0,T].
        \end{align}

        \begin{theorem}\label{thm:policy_improvement}
            Suppose that Assumptions \ref{as:refer_semi} and \ref{as:reward} hold. Let $\overline Y^{x;N,\lambda}$ be the first component of the unique solution  of~\eqref{eq:g_BSDE_explore_optimal} (see Theorem \ref{thm:verification}). Let $\pi^1\in \Pi$ be given. Let $(\overline Y^{x;N,\lambda,\pi^{1}},\overline Z^{x;N,\lambda,\pi^{1}})$ be the unique solution of \eqref{eq:g_BSDE_explore} under~$\pi^1$. For every~$n\in\mathbb{N}$, let $\pi^{n+1}$ be defined iteratively according to \eqref{eq:robust_iter_control} and let $(\overline Y^{x;N,\lambda,\pi^{n+1}},\overline Z^{x;N,\lambda,\pi^{n+1}})$ be the unique solution of~\eqref{eq:g_BSDE_explore} under $\pi^{n+1}$. Then the following hold for every $n\in \mathbb{N}$:
            \begin{itemize}
                \item [(i)] $ \overline Y_t^{x;N,\lambda} \ge \overline Y_t^{x;N,\lambda,\pi^{n+1}}\geq \overline Y_t^{x;N,\lambda,\pi^{n}}$, $t\in[0,T]$, $\mathbb{P}$-a.s.; 
                \item [(ii)] 
                Set $\Delta({x;N,\lambda,\pi^1}):= \| \overline Y^{x;N,\lambda}-\overline Y^{x;N,\lambda,\pi^{1}}\|_{\mathbb{S}^2}^2$. There exists some constant ${C}>0$ (that depends on $N,T,d$ but not on $n,\lambda$) such that
                \begin{align*}
                    &\|\overline Y^{x;N,\lambda}-\overline Y^{x;N,\lambda,\pi^{n+1}}\|_{\mathbb{S}^2}^2 + \|\overline Z^{x;N,\lambda}-\overline Z^{x;N,\lambda,\pi^{n+1}}\|_{\mathbb{L}^2}^2
                    \leq \frac{{C}^n}{n!} \Delta({x;N,\lambda,\pi^1}),\\
                    &\|  {\pi}^{n+1} -{\pi}^{*}\|_{\mathbb{S}^2}^2 \leq \frac{N}{\lambda} \frac{{C}^{n-1}}{(n-1)!} \Delta({x;N,\lambda,\pi^1}).
                \end{align*}
            \end{itemize}
            In particular, 
            for every $t \in [0,T]$, $\overline Y^{x;N,\lambda,\pi^{n}}_t \uparrow \overline Y^{x;N,\lambda}_t$ and 
            $\pi^{n}_t \downarrow \pi^{*}_t$ $\mathbb{P}$-a.s.,  as~$n \to \infty$.
        \end{theorem}
        The proof of Theorem \ref{thm:policy_improvement} can be found in Section \ref{proof:thm:policy_improvement}.

        \begin{remark}
        \label{rem:policy_iteration_contribution}
        The policy iteration result in \cite[Theorem 4]{dai2024learning} builds on PDE arguments in a Markovian framework, whereas Theorem \ref{thm:policy_improvement} is based on BSDE arguments and shows convergence rates in a non-Markovian setting. In particular, \cite{dai2024learning} considers a bounded payoff and an initial iterate value function that is continuous and~bounded, whereas our result applies to payoffs without boundedness condition and does not require an initial value function; instead, the iteration is induced by any initial policy. 

        A BSDE-based policy iteration result for Dynkin games is studied in \cite[Theorem~3]{DaiDong2024-dynkin} under boundedness conditions on their BSDE paramters. While the iterative arguments using the comparison principle of BSDEs therein might be similar to ours, the presence of the generator $\overline Z^{x;N,\lambda,\pi^n}$ in our BSDE driver, induced by ambiguity, requires a more delicate analysis of a priori estimates and iterative arguments. In particular, our approach does not require boundedness of our BSDE parameters. 
        \end{remark} 

    \begin{remark}
            Let $N\in \mathbb{N}$ and $\lambda>0$ be given. Using $(\pi^{n})_{n\in\mathbb{N}}$ in \eqref{eq:robust_iter_control}, consider a sequence of stopping times $(\tau^{x;N,\lambda,n}_t)_{n\in \mathbb{N}}\subset {\cal T}_t$ defined, for every $n\in\mathbb{N}$, by 
            \begin{align*}
            \tau^{x;N,\lambda,n}_t
            :=&\inf\bigg\{s\ge t:\pi^{n+1}_s\geq \frac{1}{1+2^{N(T-s)}} \bigg\}\wedge T \\
            =&
            \inf\{s\ge t:\overline{Y}_s^{x;N,\lambda,\pi^n}-\lambda (T-s)\log 2\le R(X_s^x)\}\wedge T\quad t\in[0,T].
            \end{align*}
            From Theorem \ref{thm:policy_improvement} and definition of $\tau^{x;N,\lambda}_t$ in \eqref{eq:explore_stopping}, we can use the same arguments presented for Theorem \ref{thm:explore_stop_conv} 
            to have that for any $t\in[0,T]$, $\tau^{x;N,\lambda,n}_t\to\tau^{x;N,\lambda}_t$~$\mathbb{P}$-a.s., as $n\to \infty$. Moreover, for any $t\in[0,T]$, $\|{\cal E}_t^g[\operatorname{I}_t^{x;\tau_t^{x;N,\lambda,n}}]-{\cal E}_t^g[\operatorname{I}_t^{x;\tau_t^{x;N,\lambda}}]\|_{L^2}\to 0$~as~$n\to\infty$. 
        \end{remark}

        Let us mention some {Markovian} properties of the BSDEs arising in the policy iteration result given in Theorem \ref{thm:policy_improvement}, as well as how these properties can be leveraged to implement the policy iteration algorithm using {\it neural networks}.
        To that end, 
        in the remainder of this section, we consider the following specification:
        \begin{setting}\label{set:markov-BSDE}
            \begin{itemize}
                \item [(i)] The map $g$ given in Definition \ref{def:g_expect} is deterministic, i.e., for every $\omega^1,\omega^2\in \Omega$, $g(\omega^1,\cdot,\cdot)=g(\omega^2,\cdot,\cdot)$. 
                \item[(ii)] The baseline parameters $b^o$ and $\sigma^o$ appearing in \eqref{eq:refer_semi} are of the form given in Remark \ref{rem:suff_refer_semi}\;(ii), so that Assumption \ref{as:refer_semi} holds.
            \item[(iii)] The reward functions $R$ and $r$ satisfy all the conditions in Assumption~\ref{as:reward}\;(i). Furthermore, $r$ is continuous. Lastly, the discount rate process $(\beta_t)_{t\in[0,T]}$ is deterministic and bounded by the constant $C_\beta>0$ in Assumption~\ref{as:reward}\;(ii).
            \end{itemize}
        \end{setting}


        Denote by $\check\Pi$ the set of all Borel measurable maps 
        \(
            \check \pi:[0,T]\times \mathbb{R}^d\ni (t,\tilde x)\to \check \pi_t(\tilde x)\in[0,1],
        \)
        so that $\check \pi(X^x):=(\check{\pi}_t(X_t^x))_{t \in[0,T]}\in \Pi$, i.e., $\check\Pi$ is the closed loop policy~set.

        Under Setting \ref{set:markov-BSDE},
        set for every $\check \pi\in \check\Pi$ and $(t,\tilde x,y,z)\in[0,T]\times \mathbb{R}^d\times \mathbb{R}\times \mathbb{R}^d$,
        \begin{align}\label{eq:Markov_generator}
               \check F_t^{N,\lambda;\check \pi}(\tilde x,y,z) :=  r(\tilde x) - \beta_t y + g(t,z) + N(R(\tilde x)-y)\check\pi_t(\tilde x) - \lambda \mathcal{H}\big(\check\pi_t(\tilde x)\big),
        \end{align}
        so that $(\check F_t^{N,\lambda,\check \pi}(\cdot,\cdot,\cdot))_{t\in[0,T]}$ is deterministic and $\check F_\cdot^{N,\lambda,\check\pi}(\cdot,\cdot,\cdot)$ is Borel measurable.


        \begin{remark}\label{rem:optimal_visco}
            Under Setting~\ref{set:markov-BSDE}, recall $(\overline Y^{x;N,\lambda},\overline Z^{x;N,\lambda})$ satisfying \eqref{eq:g_BSDE_explore_optimal}; see also Theorem \ref{thm:verification}). Then set for every $(t,\tilde x,y,z)\in[0,T]\times \mathbb{R}^d\times \mathbb{R}\times \mathbb{R}^d$
            \begin{align*}
                \check{F}_t^{N,\lambda}(\tilde x,y,z):=r(\tilde x) - \beta_t y + g(t,z) + N(R(\tilde x)-y)+\lambda \log(e^{-\frac{N}{\lambda}\{R(\tilde x)-y\}}+1).
            \end{align*}
            Clearly, $\check{F}_t^{N,\lambda}(X_t^x,y,z)=\overline F^{x;N,\lambda}_t(y,z)$ for $(t,x,y,z)\in[0,T]\times \mathbb{R}^d\times \mathbb{R}\times \mathbb{R}^d$; see \eqref{eq:optimal_explore_gen}. Moreover, $\check{F}_\cdot^{N,\lambda}(\cdot,\cdot,\cdot)$ and $R(\cdot)$ satisfy the conditions (M1b) and ($\textrm{M1b}^c$) given in \cite{KHM}. Therefore, an application of \cite[Theorem~8.12]{KHM} ensures the existence of a viscosity solution\footnote{We refer to \cite[Definition 8.11]{KHM} for the definition of a viscosity solution of \eqref{eq:semi-linear-PDE} with setting the terminal condition $R\curvearrowright  \Psi$ and the generator $\check{F}^{N,\lambda}_\cdot \curvearrowright g$ therein.} $\check v^{N,\lambda}$ of the following PDE:,
            \begin{align}
            (\partial_tv+\mathcal{L} v)(t,x) + \check{F}^{N,\lambda}_t\big(x, v(t,x), ((\widetilde \sigma^o)^{\top}\nabla  v)(t,x)\big) =0,\;\;\; (t,x)\in [0,T)\times \mathbb{R}^d,\label{eq:semi-linear-PDE}
            \end{align}
            with $v(T,\cdot) = R(\cdot)$, where the infinitesimal operator $\mathcal{L}$ of $X^x$ under the measure~$\mathbb{P}$ is given by $\mathcal{L}v(t,x):= \frac{1}{2}\sum_{i,j=1}^d ((\widetilde \sigma^o)^{\top} \widetilde \sigma^o(t,x) )_{i,j} \frac{\partial^2v(t,x)}{\partial x_i \partial x_j} + \sum_{i=1}^d \widetilde b^o_i(t,x) \frac{\partial v(t,x)}{\partial x_i}$. In particular, it holds that $\overline Y_t^{x;N,\lambda}=\check v^{N,\lambda}(t,X_t^x)$, $\mathbb{P}\otimes dt$-a.e., for all $t\in[0,T]$.
        \end{remark}

        We now have a sequence of closed-loop policies in $\check \Pi$ deriving the policy iteration. 
        The proof is presented in Section \ref{proof:thm:policy_improvement}.
        \begin{corollary}\label{coro:Markovian_of_policy_iter_closed_loop}
            Under Setting~\ref{set:markov-BSDE}, let $\check\pi^{1}\in \check \Pi$ be given. 
            \begin{itemize}
                \item[(i)] There exists two sequences of Borel measurable functions $(v^{N,\lambda,n})_{n\in \mathbb{N}}$ and $(w^{N,\lambda,n})_{n\in \mathbb{N}}$ defined on $[0,T]\times \mathbb{R}^d$ (having values in $\mathbb{R}$ and $\mathbb{R}^d$, respectively) such that for every $n\in \mathbb{N}$ and every $t\in[0,T]$, $\mathbb{P}\otimes dt$-a.e.,
                \begin{align*}
                \overline Y_t^{x;N,\lambda,\check\pi^{n}(X^x)}=v^{N,\lambda,n}(t,X_t^{x}),\quad
                \overline Z_s^{x;N,\lambda,\check\pi^{n}(X^x)}= \big((\widetilde \sigma^o)^\top 
                w^{N,\lambda,n}\big) (t, X_t^{x}),
                \end{align*}
                with $\check \pi^{n}(X^x):=(\check \pi^{n}_t(X_t^x))_{t\in [0,T]}\in \Pi$, where for any $n\geq 2$, $\check \pi^{n}\in \check \Pi$ is defined iteratively as for $(t,\tilde x) \in [0,T]\times \mathbb{R}^d$
                \begin{align}\label{eq:robust_iter_control_deter}\check\pi^{n}_t(\tilde x):=\operatorname{logit}\Big(\frac{N}{\lambda}\big(R(\tilde x)-v^{N,\lambda,n-1}(t,\tilde x)\big)\Big).
                \end{align}
                \item[(ii)] If $ \check \pi_t^{1}(\cdot)$ is continuous on  $\mathbb{R}^d$ for any $t\in[0,T]$, one can find a sequence of functions $(v^{N,\lambda,n})_{n\in \mathbb{N}}$ which satisfies all the properties given in (i) and each $v^{N,\lambda,n}$, $n\in \mathbb{N}$, is a viscosity solution of the following PDE: 
                \[(\partial_t v+\mathcal{L} v)(t,x) + \check{F}^{N,\lambda,\check\pi ^n}_t(x, v(t,x), ((\widetilde \sigma^o)^{\top}\nabla  v)(t,x)) =0,\;\;(t,x)\in [0,T)\times \mathbb{R}^d,
                \]
                with $v(T,\cdot) = R(\cdot)$, where $\check\pi^n\in \check \Pi$ is defined iteratively as in \eqref{eq:robust_iter_control_deter}.
            \end{itemize}
          	\end{corollary}

        \begin{remark}
        \label{rem:markovian_closure_PI}
        Corollary~\ref{coro:Markovian_of_policy_iter_closed_loop} is a closed-loop, Markovian counterpart to the open-loop policy iteration in Theorem~\ref{thm:policy_improvement}. In particular, measurability and continuity on the initial policy ensure that the induced sequence of policies admits feedback-form representations in \eqref{eq:robust_iter_control_deter}. As noted in Remark~\ref{rem:policy_iteration_contribution}, \cite[Theorem~4]{dai2024learning} establishes a PDE-based policy iteration result. However, our result does not follow from it, since we do not assume boundedness of the reward functions and the associated iterative PDEs lack Lipschitz regularity in source terms. Moreover, the presence of $\overline Z^{x;N,\lambda,\check \pi^n(X^x)}$, induced by ambiguity, requires additional analysis. Our result is therefore obtained via a direct probabilistic approach, 
        based on \cite[Theorem 8.12]{KHM}. 
        \end{remark}

        The core logic of the policy iteration given in 
        Theorem \ref{thm:policy_improvement} and Corollary \ref{coro:Markovian_of_policy_iter_closed_loop} consists of two steps at each iteration. The first is the {\it policy update}, given in \eqref{eq:robust_iter_control} or \eqref{eq:robust_iter_control_deter}. The second is the {\it policy evaluation}, which corresponds to derive either the solution $(\overline Y^{x;N,\lambda,\pi^{n}},\overline Z^{x;N,\lambda,\pi^{n}})$ of the BSDE \eqref{eq:g_BSDE_explore} under the updated policy $\pi^n$, or equivalently, the solution $v^{N,\lambda,n}$ of the PDE under $\check \pi^n$ as given in Corollary \ref{coro:Markovian_of_policy_iter_closed_loop}\;(ii). 
        
        In what follows, we develop an RL scheme, relying on the {deep splitting method} of Beck et al.~\cite{Becker-Jentzen-Neufeld-Deep-Splitting-21} and Frey and K\"ock~\cite{deep_splitting_convergence}, to implement the policy evaluation step at each iteration. For this purpose, we first introduce some notation, omitting the dependence on $(N,\lambda)$ (even though the objects still depend on them).

        \begin{setting}\label{set:discrete} Denote by $I\in \mathbb{N}$ the number of steps in the time discretization and denote by $\Theta \subset \mathbb{R}^p$  (with some $p\in \mathbb{N}$) the parameter spaces for neural networks in.  
        \begin{enumerate}
            \item [(i)] Let $t_i = i\Delta t$ and $\Delta B_i:= B_{t_{i+1}} - B_{t_i}$ for $i=\{0,\dots,I-1\}$ with $\Delta t:= T / I$. Then the Euler scheme of \eqref{eq:refer_semi} under Setting \ref{set:markov-BSDE}\;(ii) is given by: $\check{X}^x_{0} := x$,
         \begin{align*}
            \check{X}^x_{i+1}:= \check{X}^x_{i} + \widetilde{b}^o(t_i,\check{X}^x_{i})\Delta t + \widetilde{\sigma}^o(t_i,\check{X}^x_{i}) \Delta B_i, \quad i\in\{0,\ldots,I-1\}.
         \end{align*}
            \item [(ii)] The initial closed-loop policy $\check \pi^1$ is given by \(\check\pi^1_i(\cdot):=\operatorname{logit}(\frac{N}{\lambda}(R(\cdot) - v^{0}_{i}(\cdot)))\), $i\in\{0,\dots,I-1\}$, 
        with some function (at least continuous) $v^{0}_{i}:\mathbb{R}^d\to \mathbb{R}$.
            \item [(iii)] For each $n\in \mathbb{N}$ and $i\in\{0,\dots,I-1\}$, let $v_i^n(\,\cdot\,;\vartheta^n_i):\mathbb{R}^d\to \mathbb{R}$ 
            be {\it neural realizations} of $v^{N,\lambda , n}(t_i,\cdot)$ 
            parameterized by 
            $\vartheta^n_i\in \Theta$ (e.g., feed-forward networks (FNNs) with $C^1$-regularity or Lipschitz continuous with weak~derivative). 
            \item [(vi)] For each $n\in \mathbb{N}$, the time-discretized, $n+1$-th updated, closed-loop policy $\check \pi^{n+1}(\cdot;\vartheta_i^n)$ (that depends on the parameter $\vartheta^n_i$ appearing in (iii)) is given by 
            \(\check\pi^{n+1}_i(\cdot ;\vartheta_i^n):=\operatorname{logit}(\frac{N}{\lambda}(R(\cdot) - v^{n}_{i}(\cdot;\vartheta^n_i)))\), $i\in\{0,\dots,I-1\}.$
        \item [(v)] For each $n\in \mathbb{N}$, set for every $(\tilde x,y,z)\in\mathbb{R}^d\times \mathbb{R}\times \mathbb{R}^d$,
        \begin{align*}
                \begin{aligned}
               \check F_i^{n}(\tilde x,y,z;\vartheta_i^{n-1}) &:=  r(\tilde x) - \beta_{t_i} y + g(t_i,z) + N(R(\tilde x)-y)\check\pi_i^n(\tilde x;\vartheta_i^{n-1})\\
               &\quad - \lambda \mathcal{H}\big(\check\pi_i^n(\tilde x;\vartheta_i^{n-1})\big),
               \end{aligned}
        \end{align*}
        with the convention that $\check \pi^1(\cdot;\vartheta_i^{0})\equiv\check \pi_i^1(\cdot)$ for any $\vartheta_i^0\in \Theta$ (see (ii)) so that $\check F^1_i(\cdot,\cdot,\cdot)$ is not parametrized over $\Theta$ but depends only on the form $\check \pi_i^1$.
        \end{enumerate}
        \end{setting}
         To apply the {deep splitting method}, one needs $ \widetilde{\sigma}^o(t_i,\check{X}^x_{i}) $ in the loss function calculation (given in \eqref{eq:LSTD-no-martingale}), which is unknown to an RL agent before learning the environment but can be learned from from the realized quadratic covariance of observed data\footnote{The mapping $\mathbb{R}^{d\times d}\ni A\mapsto A^{\frac{1}{2}}\in \mathbb{R}^{d\times d}$ denotes the symmetric positive-definite square root of a positive semidefinite matrix $A$.} 
        \[
            \Sigma({\check{X}^x_{i:i+1}}):= \frac{1}{\sqrt{\Delta t} }\big((\check{X}^x_{i+1} - \check{X}^x_{i})(\check{X}^x_{i+1} - \check{X}^x_{i})^\top\big)^{\frac{1}{2}},  
        \]
        so that 
         %
         $\Sigma({\check{X}^x_{i:i+1}})\Sigma({\check{X}^x_{i:i+1}})^\top\Delta t \to \widetilde{\sigma}^o(t_i,\check{X}^x_{i})\widetilde{\sigma}^o(t_i,\check{X}^x_{i})^\top\Delta t$ as $ \Delta t \downarrow 0$  in probability~$\mathbb{P}$; see e.g., \cite[Chapter I, Theorem~4.47]{jacod2013limit} and \cite[Section~6, Theorem 22]{Protter2005Stochastic}. 

        \begin{algorithm}[t]
        \caption{Policy iteration algorithm}
        \label{alg:exact_policy_iteration}
        \begin{algorithmic}[1]
        \REQUIRE  Batch size $M\in \mathbb{N}$; Number of policy iterations $\overline n \in \mathbb{N}$; Number of epochs $\overline \ell \in \mathbb{N}$ for policy evaluation; Learning rate $\alpha\in (0,1)$.
        \STATE Set the initial closed loop policy $\check\pi^1_i(\cdot)$, $i\in\{0,\ldots,I-1\}$, as in Setting \ref{set:discrete}\;(ii).\\
        \STATE Initialize $\vartheta_i^{0,*}\in \Theta$, $i\in\{0,1,\dots,I\}$. 
        \FOR{$n = 1, \ldots, \bar n$}
          \STATE Initialize $\vartheta^n_{i}\in \Theta$,  $i\in\{0,\ldots,I-1\}$, and $\vartheta_I^{n,*}\in \Theta$. 
            \FOR{$l = 1, \ldots, \bar \ell$}
                \STATE Generate $M$ trajectories of $(\check{X}^{x}_{i})_{i=0}^{I}$; see Setting \ref{set:discrete}\;(i).
            \FOR{ $i=I-1,\ldots,0 $  } 
                \STATE Minimize \eqref{eq:LSTD-no-martingale} over \(\vartheta^n_i \in \Theta \) by using SGD with learning rate $\alpha $.
        \ENDFOR
        \ENDFOR
          \STATE Denote by $\vartheta^{n,*}_i$ the lastly updated parameters at $t_i$, $i\in\{0,\ldots,I-1\}$.
        \ENDFOR
        \end{algorithmic}
        \end{algorithm}

        With all this notation set in place, for each iteration $n\in \mathbb{N}$, we present the policy evaluation as a {\it backward iterative} minimization problem. For any $i\in \{0,\dots,I-1\}$, let $ \vartheta^{n-1,*}_i , \; i\in \{0,\ldots,I-1\} $ denote the optimal parameter at time $t_i$ from iteration $n-1$, and let \( \vartheta^{n,*}_{i+1}\) denote the optimal parameter at time $t_{i+1}$ in $n$-iteration. Then the optimal parameter at time $t_i$ in $n$-iteration is given by \color{black} 
        \begin{align}\label{eq:loss_minimizat}
             \vartheta^{n,*}_i  \in \argmin_{\vartheta^n_i\in \Theta} \mathfrak{L}^n(\vartheta^n_i;\vartheta_i^{n-1,*}, \vartheta_{i+1}^{n,*}),  
        \end{align}
        where $\mathfrak{L}_i^n(\cdot;\vartheta_i^{n-1,*},\vartheta_{i+1}^{n,*}): \Theta  \to \mathbb{R} $ is the parameterized $L^2 $-loss function given by
         \begin{align}
            & \mathfrak{L}^n(\vartheta^n_i  ; \vartheta_i^{n-1,*}, \vartheta_{i+1}^{n,*})
            := \mathbb{E}\Big[\big| v^{n}_{i+1}(\check{X}^x_{i+1};\vartheta^{n,*}_{i+1}) - v^{n}_{i}(\check{X}^x_{i};\vartheta^{n}_{ i }) \nonumber \\
             & \quad \quad   +\check{F}^{n}_i\big(\check{X}^x_{i+1},v^{n}_{i+1}(\check{X}^x_{i+1};\vartheta^{n,*}_{i+1}), \Sigma({\check{X}^x_{i:i+1}}) \nabla v^{n}_{i+1}(\check{X}^x_{i+1};\vartheta^{n,*}_{i+1});\vartheta_i^{n-1,*}\big)\Delta t  \big|^2\Big],\label{eq:LSTD-no-martingale}
         \end{align}
         with the convention that ${v}^n_{I}(\check X_{I}^x;\vartheta_I^{n,*}) := R(\check X_{I}^x)$ with an arbitrary~$\vartheta_I^{n,*}\in \Theta$, and that $\check F^{1}_i$ is not parametrized over $\Theta$ (see Setting \ref{set:discrete}\;(v); hence $\vartheta_i^{0,*}\in \Theta$ is also an~arbitrary). 

        We numerically solve the problem given in \eqref{eq:loss_minimizat} by using stochastic gradient descent (SGD) algorithms (see, e.g., \cite[Section 4.3]{goodfellow2016deep}). Then we provide a pseudo-code in Algorithm~\ref{alg:exact_policy_iteration} to show how the policy iteration can be implemented. 

        \begin{remark} Note that the deep splitting method of \cite{Becker-Jentzen-Neufeld-Deep-Splitting-21,deep_splitting_convergence} is not the only neural realization of our policy evaluation; instead deep BSDEs\,/\,PDEs schemes of \cite{han18,pham2020deepBSDE,sirignano2018dgm} can be an alternative. More recently, several articles, including \cite{pham2022deepBSDE_erroranalysis,neufeld2024full}, provide the error analyses for such methods. To obtain a full error-analysis of our policy iteration algorithm, one would need to relax the standard Lipschitz and H\"older conditions on BSDE generators in the mentioned articles so as to cover the generator
        $\check F^{N,\lambda ,\check \pi^n}$ in \eqref{eq:Markov_generator}, 
        and then incorporate the policy evaluation errors from the neural approximations (under such relaxed conditions) into the convergence rate established in Theorem~\ref{thm:policy_improvement}. We defer this direction to a future~work.
        \end{remark}

    \section{Experiments}\label{sec:experiment}
    In this section,\footnote{All computations were performed using PyTorch on a Mac Mini with Apple M4 Pro processor and 64GB RAM. The complete code is available at: \url{https://github.com/GEOR-TS/Exploratory_Robust_Stopping_RL}.}
    we analyze some examples to support the applicability of Algorithm~\ref{alg:exact_policy_iteration}. 
    Let us fix $g(t,z)\equiv-\varepsilon|z|$ for $(t,z)\in[0,T]\times \mathbb{R}^d$, where $\varepsilon\geq 0$ represents the degree of ambiguity. By Remark~\ref{rem:represent_g_exp}, for any $\xi \in L^2(\mathcal{F}_{\tau};\mathbb{R}^d)$, it holds that $ \mathcal{E}^g_t[\xi] = \esssup_{\vartheta \in \mathcal{B}^{\varepsilon}} \mathbb{E}_t^{\mathbb{P}^{\vartheta}}[\xi]$, where $\mathcal{B}^{\varepsilon} $ includes all $\mathbb{F}$-progressively measurable processes $(\vartheta_t)_{t\in [0,T]} $ such that $|\vartheta_t| \le \varepsilon $ $\mathbb{P}\otimes dt$-a.e.. 
    
    In the training phase, 
    following Setting \ref{set:discrete}\;(vi),  
    we parametrize $v^{N,\lambda,n}(t_i, x)$ by 
    \[
        v^{n}_i(x ; \vartheta^n_i) = R(x) + \mathcal{NN}^1(x,R(x) ; \vartheta^n_i),\quad x\in \mathbb{R}^d,
    \]
    where $\mathcal{NN}^1(\cdot,\cdot; \vartheta^n_i) : \mathbb{R}^{d}\times\mathbb{R} \to \mathbb{R}$ denotes an FNN of depth $2$, width $20+d$, and $\mathrm{ReLU}$ activation, and $\vartheta^n_i  \in \Theta $ denotes the parameters of the FNN. 
    In all experiments, the number of policy iterations, epochs and the training batch size is set to $\overline n = 10 $,  $\overline \ell = 1000 $ and $2^{10}$, respectively. For numerical stability and training efficiency, we apply batch normalization before the input and at each hidden layer, together with Xavier normal initialization and the ADAM optimizer. 
    To make dependencies explicit, we denote by $(v^{N,\lambda,\star;\varepsilon}_i)_{i=0}^{I}$, obtained after sufficient policy iterations, under penalty factor $N$, temperature $\lambda $, and ambiguity degree $\varepsilon$.

    \subsection{Example 1}\label{ex:example1}
    We first conduct experiments on {1-dimensional} American put~/ call holder's stopping problems to illustrate the policy improvement, convergence, stability, and robustness of Algorithm~\ref{alg:exact_policy_iteration}. The simulation settings are as follows: under Setting \ref{set:discrete}, we let the running reward $r(\cdot) \equiv 0 $, the discounting factor $\beta_t \equiv r_{*}$, the volatility $\widetilde{\sigma}^o (t, x) = 0.4x$,  $x\in \mathbb{R}$, the initial price and strike price $x_0= \Gamma = 40$, and
        \begin{itemize}
            \item[(i)] ({Put})
            $T=1$, $I = 50 $, 
           the interest rate $ r_{*} = 0.06 $,  
            the payoff $R(x) = (\Gamma - x)^{+}$, the drift $ \widetilde{b}^o (t,x) = r_{*} x$;

            \item[(ii)] ({Call}) 
            $T=0.5$, $I = 100 $, 
            the dividend rates in the training simulator ${\delta}_{\mathrm{train}}=0.05 $ and in the testing simulator $ \delta $
             $\in \{0, 0.05, 0.1, 0.15, 0.2, 0.25\}$, the interest rate $ r_{*} = 0.05 $, the payoff $R(x) = (x - \Gamma)^{+}$, the drift $  \widetilde{b}^o(t,x) = (r_{*} - \delta ) x$. 
        \end{itemize}

    We first examine the policy improvement and convergence of Algorithm \ref{alg:exact_policy_iteration}. For the {\it put-type} stopping problem, we fix $\lambda = 1$ and $N=10 $, and consider several ambiguity degrees $\varepsilon \in \{ 0,0.2,0.4 \} $. The reference values $R^{\mathrm{ref}}_{\varepsilon} $ for $\varepsilon \in \{ 0,0.2,0.4 \} $ are obtained by solving \eqref{eq:g_BSDE_explore_optimal} for the corresponding optimal value function using the {\it deep backward scheme} of Hur\'e et al.~\cite{pham2020deepBSDE}, yielding $R^{\mathrm{ref}}_{0} = 5.302$, $R^{\mathrm{ref}}_{0.2} = 4.420$, $R^{\mathrm{ref}}_{0.4} = 3.725$.       
    The results illustrating the policy improvement and convergence are shown in Fig.\;\ref{fig:policy_improvement_convergence}, which align well with the theoretical findings in Theorem~\ref{thm:policy_improvement}. 
    
    Similarly, for the {\it call-type} stopping problem, we again fix $\lambda = 1 , N=10 $ and consider the same several ambiguity degrees. The reference values $R^{\mathrm{ref}}_{\varepsilon} $ computed by the deep backward scheme are $R^{\mathrm{ref}}_{0} = 4.378$, $R^{\mathrm{ref}}_{0.2} = 3.677$, $R^{\mathrm{ref}}_{0.4} = 3.130$. The corresponding policy improvement and convergence results are depicted in Figure~\ref{fig:policy_improvement_convergence}.

    \begin{figure}[tbp] \label{fig:policy_improvement_convergence}
    \begin{minipage}[t]{0.45 \linewidth} 
    \centering
    \includegraphics[height=3.8cm,width=5.8cm]{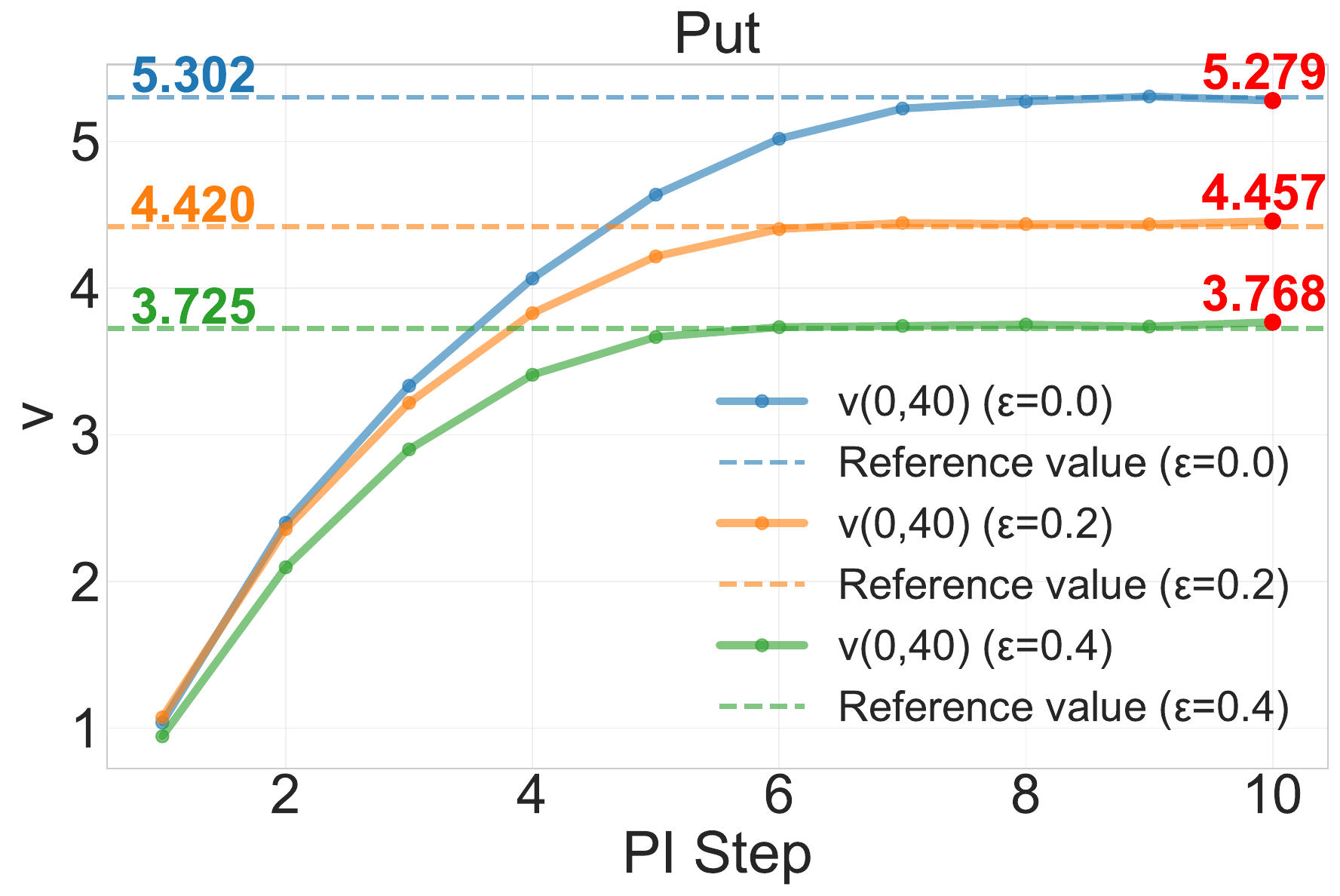}
    \label{fig:10}
    \end{minipage}%
    \begin{minipage}[t]{0.48\linewidth} 
    \centering
    \includegraphics[height=3.8cm,width=5.8cm]{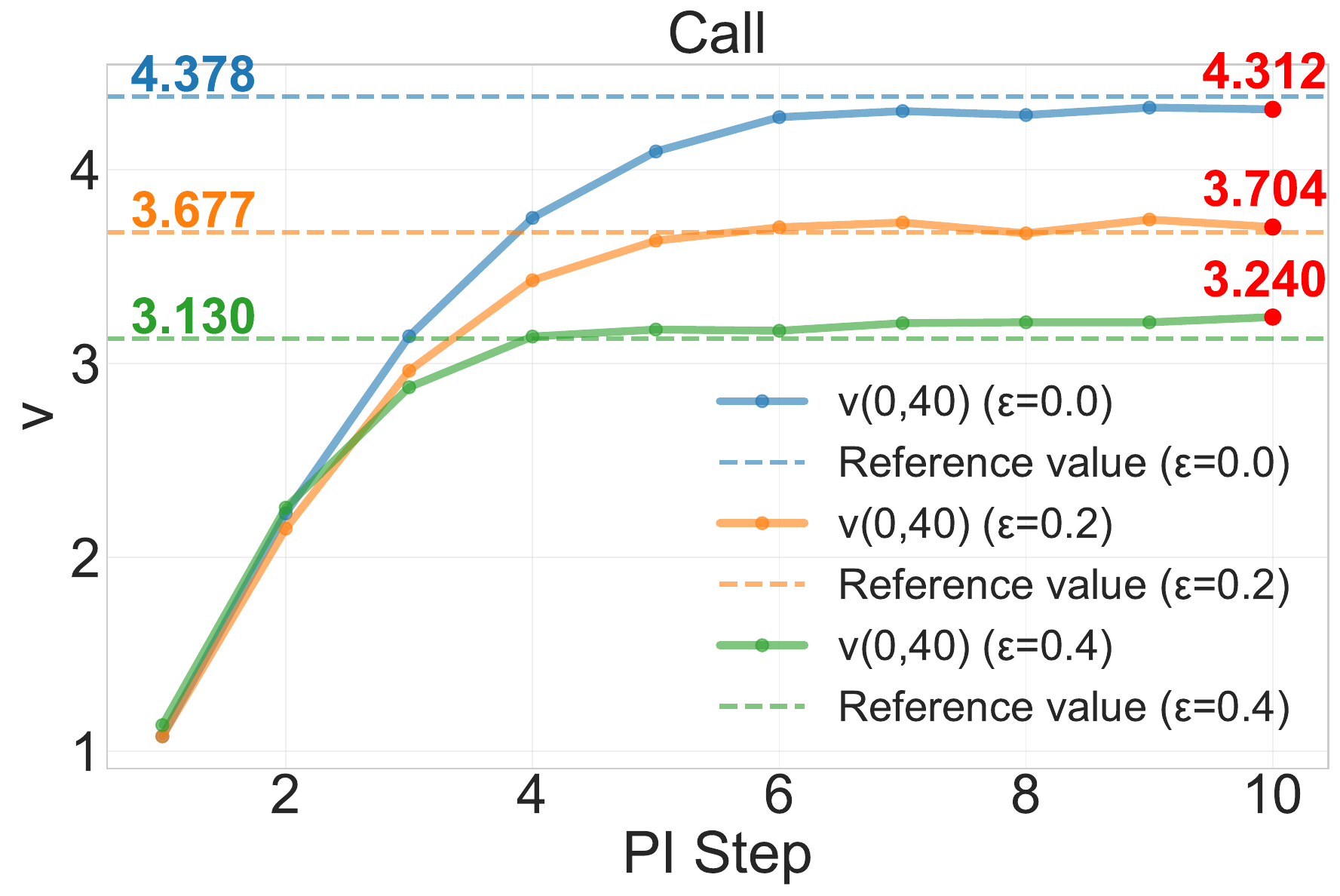}
    \label{fig:10}
    \end{minipage}%
    \caption{Policy improvement and convergence of Algorithm \ref{alg:exact_policy_iteration} in Example 1.}
    \end{figure}

    To examine the stability of Algorithm~\ref{alg:exact_policy_iteration}, we vary the penalty, temperature and ambiguity levels as $N \in \{5, 10, 20 \}$, $\lambda \in \{0.01, 1, 5 \} $, and $\varepsilon\in \{0,0.2,0.4\}$, and present the corresponding values of $v^{N,\lambda,\star;\varepsilon}_0$ in Table~\ref{tab:stability_put} (obtained after at-least 10 iterations of the policy improvement; see Fig.\;\ref{fig:policy_improvement_convergence}).
    These results align with the stability analysis w.r.t.~$\lambda$ given in Theorem~\ref{thm:stability} and the sensitivity analysis of robust optimization problems w.r.t.~ambiguity level examined in \cite[Theorem 2.13]{bartl2023sensitivity}, \cite[Corollary 5.4]{chen2024robust}.


    \begin{table}[t]
    \centering
    \small
    \caption{Stability of Algorithm~\ref{alg:exact_policy_iteration} w.r.t.~the penalty, temperature and ambiguity levels in Example 1.}
    \label{tab:stability_put}
    \begin{tabular}{c | c c c | c c c | c c c}
    \hline
    \multirow{3}{*}{$\varepsilon$} & \multicolumn{9}{c}{$v^{N,\lambda,\star;\varepsilon}_0(40)$} \\
    \cline{2-10}
     & \multicolumn{3}{c|}{$N = 5$} & \multicolumn{3}{c|}{$N = 10$} & \multicolumn{3}{c}{$N = 20$} \\
    \cline{2-10}
     & $\lambda=0.01$ & $\lambda=1$ & $\lambda=5$ & $\lambda=0.01$ & $\lambda=1$ & $\lambda=5$ & $\lambda=0.01$ & $\lambda=1$ & $\lambda=5$ \\
    \hline
    $0$ & $5.222$ & $5.278$ & $6.113$ & $5.233$ & $5.279$ & $5.788$ & $5.239$ & $5.296$ & $5.570$ \\
    $0.2$ & $4.311$ & $4.413$ & $5.258$ & $4.412$ & $4.457$ & $4.958$ & $4.425$ & $4.496$ & $4.765$ \\
    $0.4$ & $3.596$ & $3.671$ & $4.497$ & $3.702$ & $3.768$ & $4.221$ & $3.792$ & $3.814$ & $4.101$ \\
    \hline
    \end{tabular}
    \end{table}

    \begin{figure}[tbp] \label{fig:Robustness-RE-Call}
    \begin{minipage}[t]{0.33\linewidth} 
    \centering
    \includegraphics[height=2.8cm,width=4.2cm]{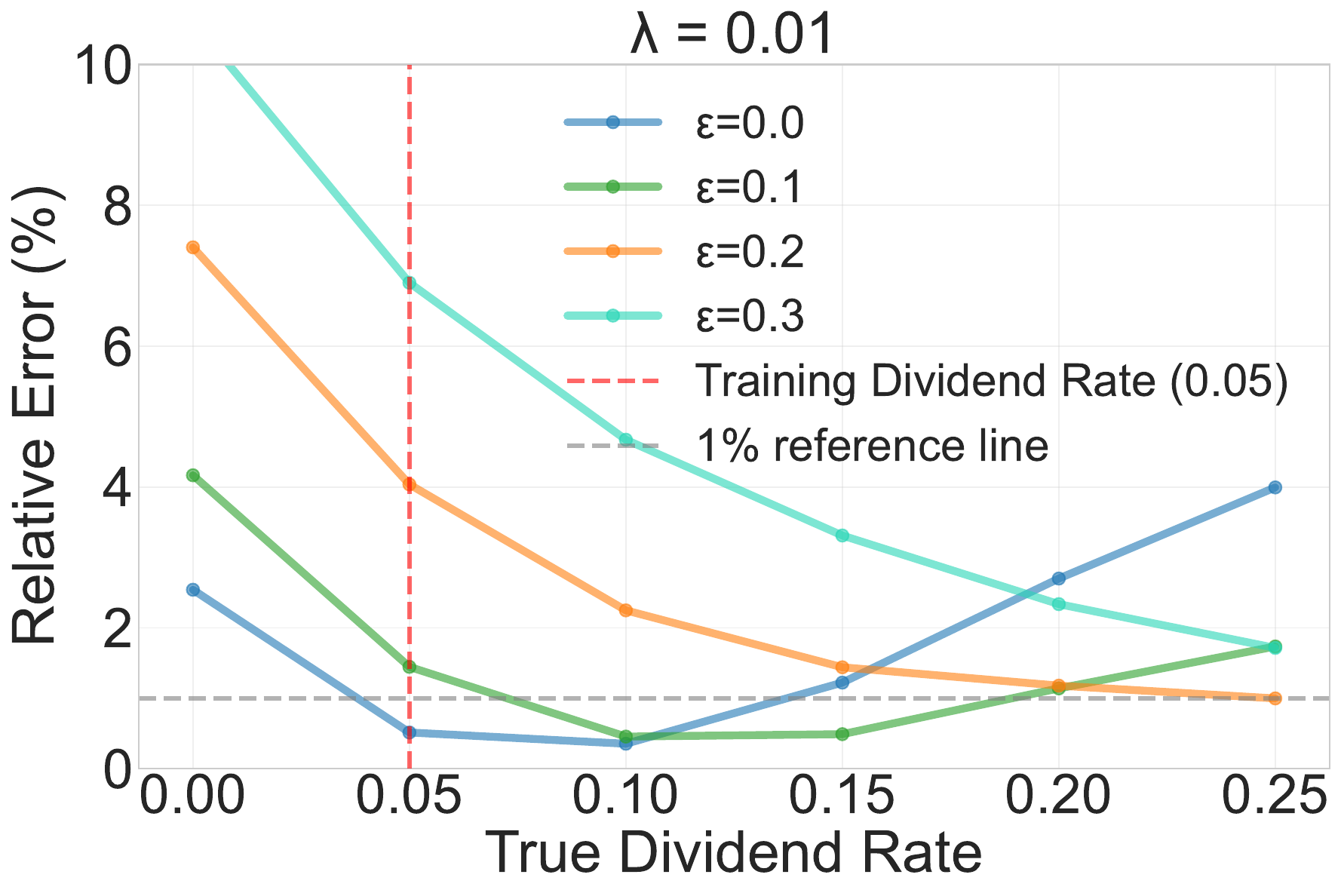}
    \label{fig:10_1}
    \end{minipage}%
    \begin{minipage}[t]{0.33\linewidth} 
    \centering
    \includegraphics[height=2.8cm,width=4.2cm]{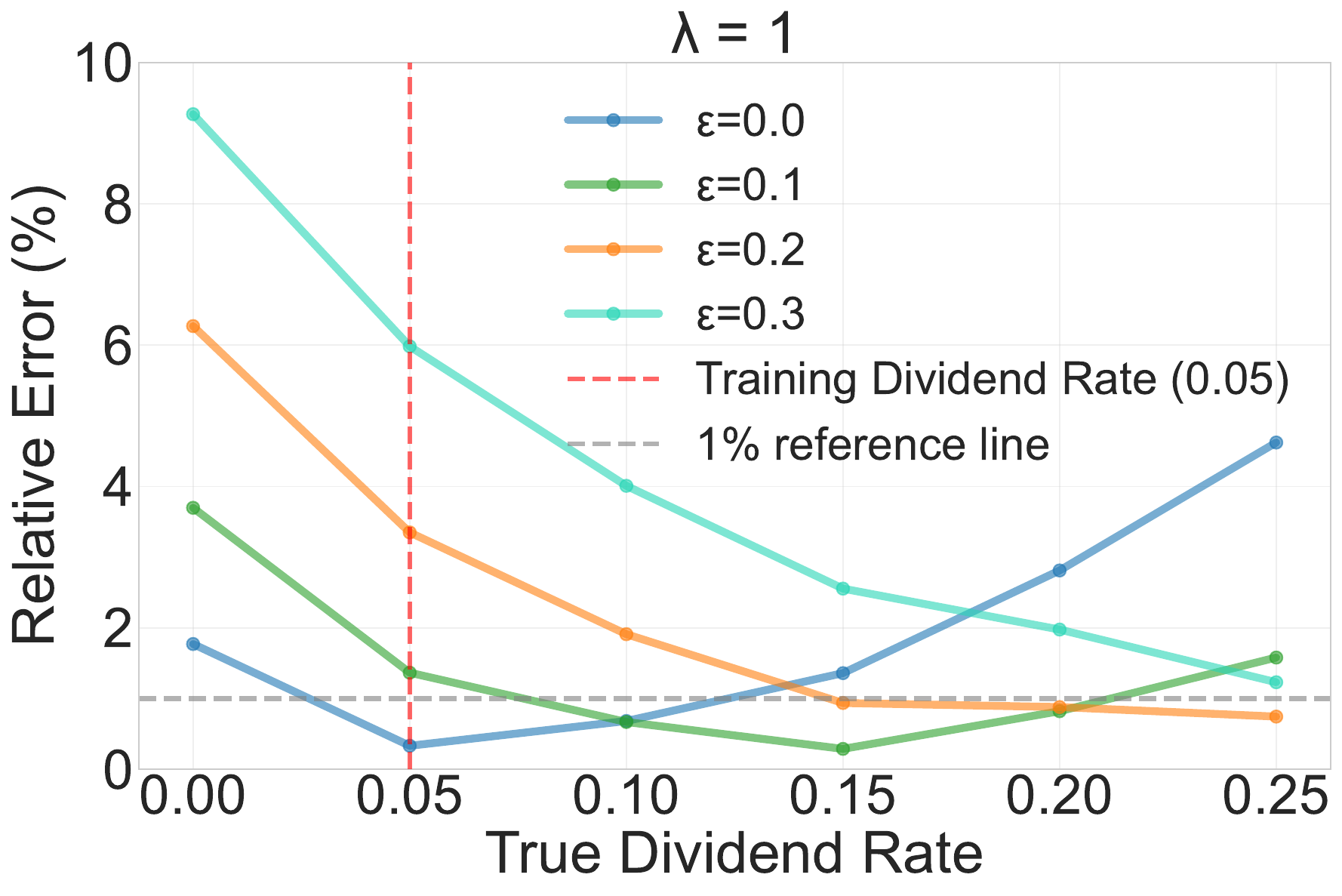}
    \label{fig:10}
    \end{minipage}%
    \begin{minipage}[t]{0.33\linewidth} 
    \centering
    \includegraphics[height=2.8cm,width=4.2cm]{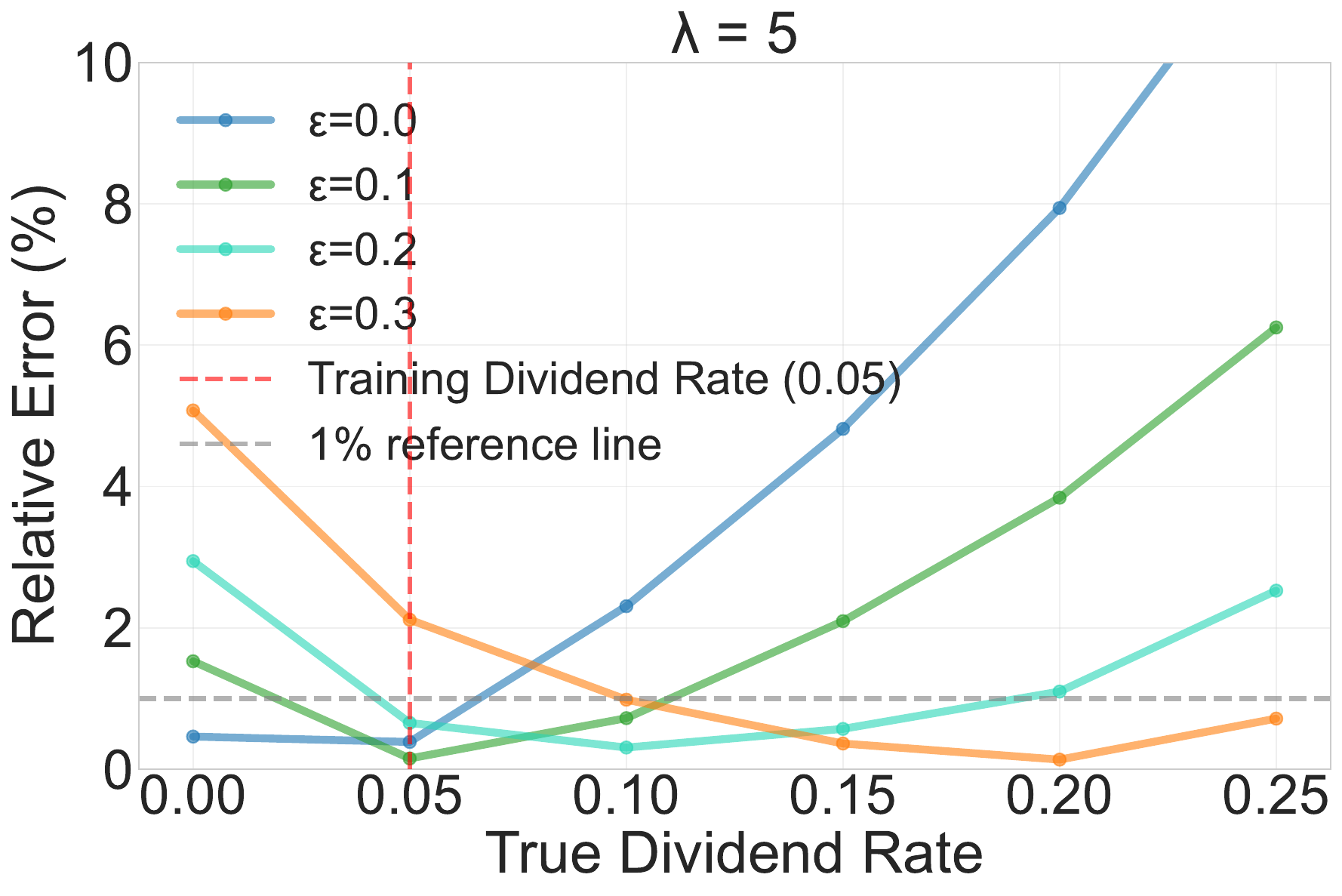}
    \label{fig:10_2}
    \end{minipage}
    \caption{Robustness performance under unknown testing environments in Example 1.}
    \end{figure}


    Lastly, we examine the robustness of Algorithm \ref{alg:exact_policy_iteration} in the call-type stopping problem. In particular, to assess the out-of-sample performance under an {\it unknown} testing environment, we re-simulate new state trajectories $(\check{X}^{x,\delta}_i)_{i=0}^I$ as in Setting~\ref{set:discrete}\;(i) under different dividend rates $\delta \in \{0, 0.05, 0.1, 0.15, 0.2, 0.25\}$, where the number of simulated trajectories is set to $ 2^{20}$.
    We fix $N=10 $ and consider configuration $\varepsilon \in \{0,0.1,0.2,0.3 \} $ both for 
    $\lambda \in \{0.01 , 1 , 5\}$. 
    Using the trained value functions $(v^{10,\lambda,\star;\varepsilon}_i(\cdot))_{i=0}^{I}$, the stopping time $\tau_\delta^{\varepsilon,\lambda}$ and the corresponding 
    discounted expected reward~$\check{R}^{\varepsilon,\lambda}_{\delta}$ under such unknown environment are defined~by 
    \begin{align}
         \tau^{\varepsilon,\lambda}_{\delta}  &:= \inf\big\{t_i :  v^{10,\lambda,\star;\varepsilon}_i(\check{X}^{x,\delta}_{i}) - {\bf 1}_{\{\lambda=0.01\}}  \lambda (T-t_i) \log2\le R(\check{X}^{x,\delta}_{i}), \; i=0,\ldots,I  \big\},\nonumber\\
        &\quad \check{R}^{\varepsilon,\lambda}_{\delta} := \mathbb{E} [ e^{-r_{*} \tau^{\varepsilon,\lambda}_{\delta} }R(\check{X}^{x,\delta}_{i}) ].\label{eq:explore_stopping_example1}
    \end{align}  
    

    \color{black}
    For each $\delta$, the corresponding American call option price represents the optimal value for the call-type stopping problem, which can be computed using the implicit finite-difference method of Forsyth and Vetzal~\cite{Forsyth-Penalty-02}. We therefore use the option prices computed by this method as reference values $R^{\mathrm{ref} }_{\delta} $ for each $\delta$, yielding $R^{\mathrm{ref} }_{0} = 4.954,$ $R^{\mathrm{ref} }_{0.05} = 4.410$, $R^{\mathrm{ref} }_{0.1} = 3.990$,  $R^{\mathrm{ref} }_{0.15} = 3.634$, $R^{\mathrm{ref} }_{2} = 3.324$, $R^{\mathrm{ref} }_{0.25} = 3.052 $. The relative errors are then computed as $  { |\check{R}^{\varepsilon,\lambda}_{\delta} - R^{\mathrm{ref}}_{\delta}  | }/{ R^{\mathrm{ref}}_{\delta} } $.

    Moreover, we note that the stopping rule in \eqref{eq:explore_stopping_example1} follows the robust exploratory stopping time $\tau_0^{x;N,\lambda}$ in \eqref{eq:explore_stopping} for the case $\lambda=0.01$, because it approximates the optimal stopping time $\tau^{*,x}_0$ in \eqref{eq:robust_OST} for sufficiently small $\lambda$ and large $N$ (see Theorem~\ref{thm:explore_stop_conv}). For larger values of $\lambda$, in particular $\lambda=1,5$, we heuristically adopt the stopping rule $\overline{\tau}_0^{x;N,\lambda}$ in Remark~\ref{rem:convergence_issue} without the adjustment term $\lambda (T-t_i)\log 2$, since the adjustment term becomes relatively large and may not only impact the approximation accuracy, but also diminish the exploration benefit.

    Last, we note in Fig.\;\ref{fig:Robustness-RE-Call} that when the dividend rate in the testing environment does not deviate significantly from that of the trained environment (near $\delta=0.05$), the non-robust value function (i.e., with $\varepsilon=0$) performs comparably well. However, as the discrepancy between the training and testing environments increases, the benefit of the robust framework becomes evident, as reflected by lower relative errors for higher ambiguity levels (i.e., $\varepsilon=0.2,0.3$). Notably, as the exploration parameter $\lambda$ increases, the robust framework performs better under larger discrepancies between the training and test environments. 
    \color{black}

    \subsection{Example 2} 
    \begin{figure}[tbp] \label{fig:policy_improvement_convergence_high_dim}
    \begin{minipage}[t]{0.45 \linewidth} 
    \centering
    \includegraphics[height=3.8cm,width=5.8cm]{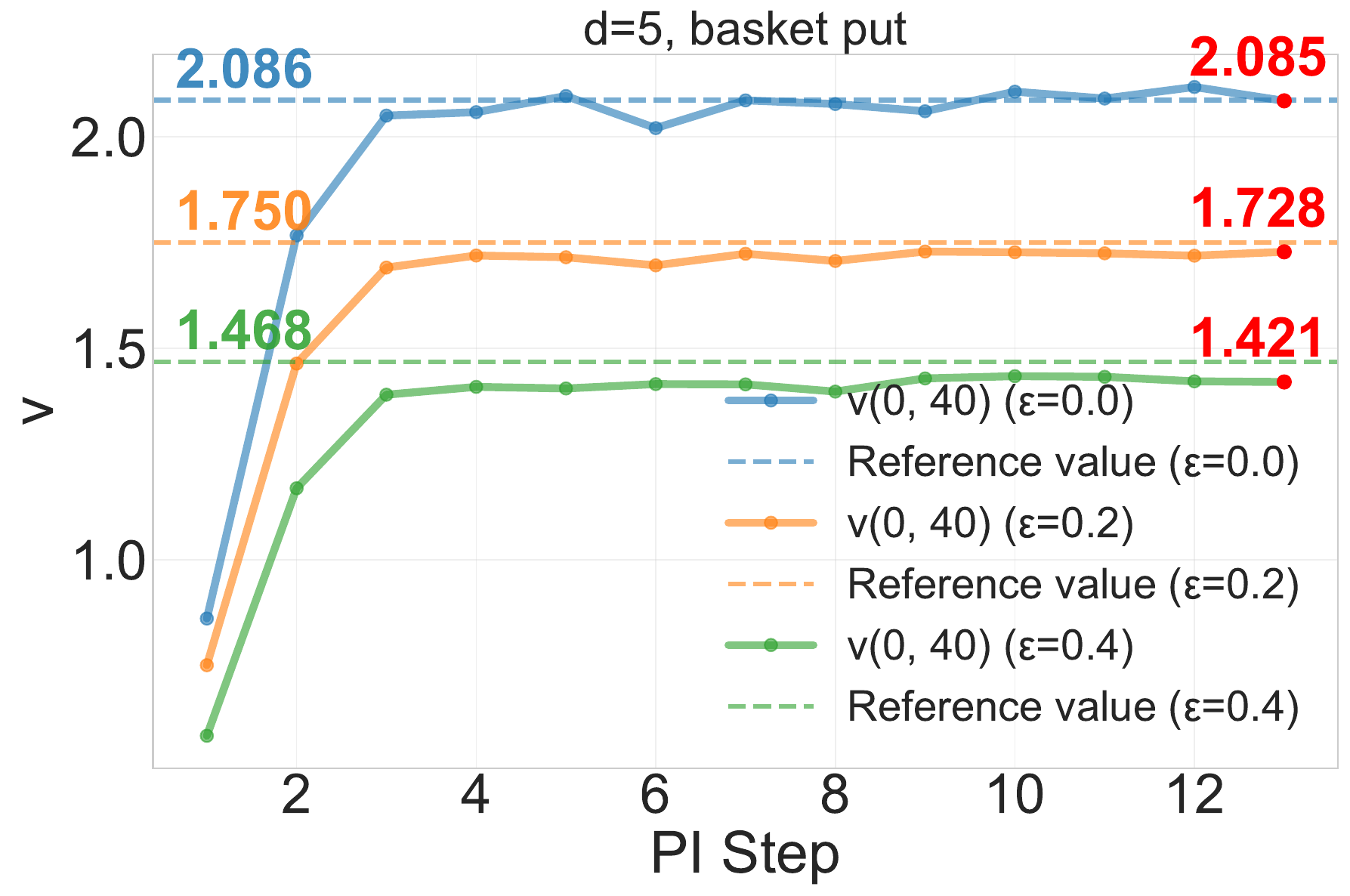}
    \label{fig:10}
    \end{minipage}%
    \begin{minipage}[t]{0.48\linewidth} 
    \centering
    \includegraphics[height=3.8cm,width=5.8cm]{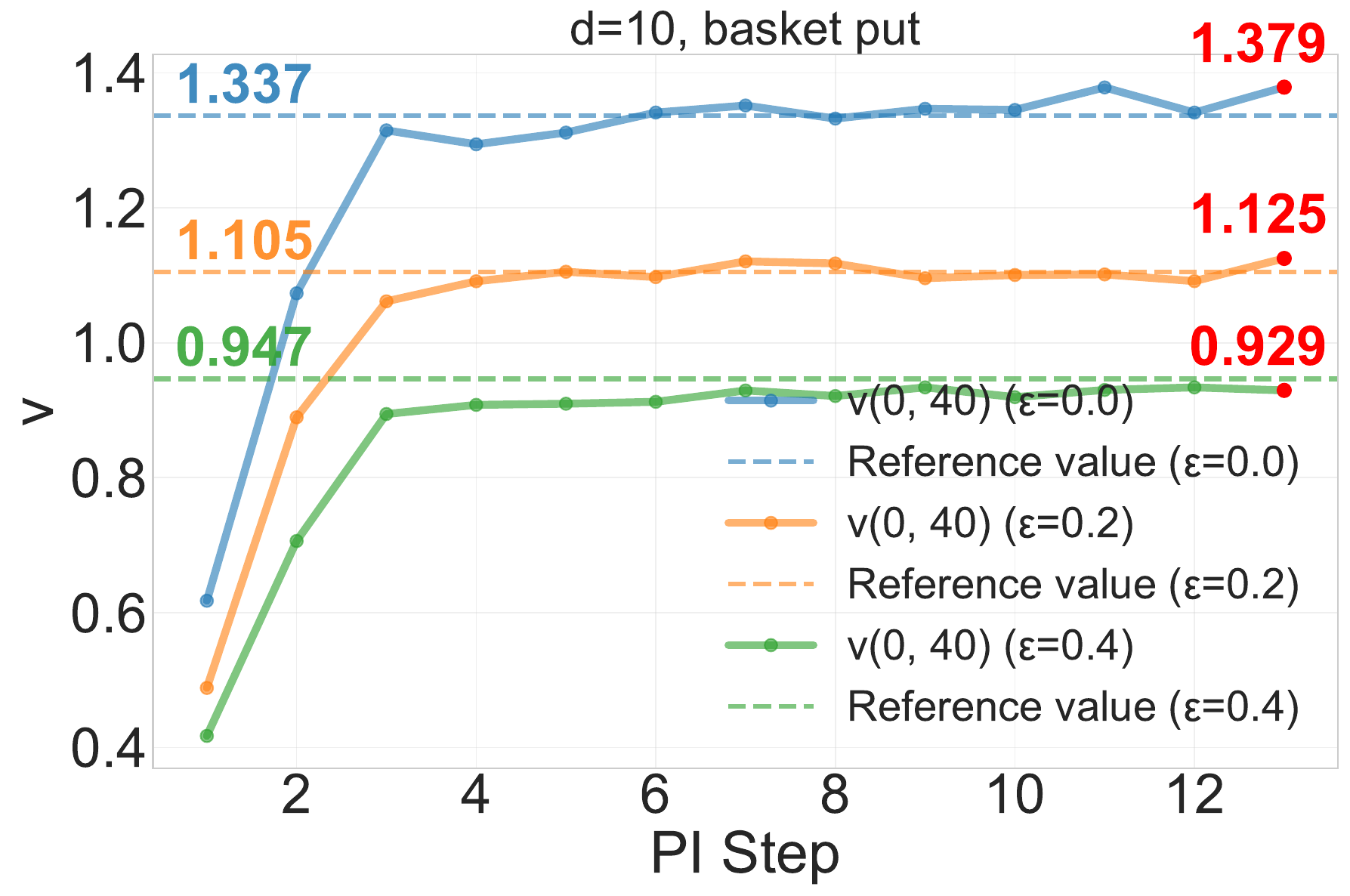}
    \label{fig:10}
    \end{minipage}%
    \caption{Policy improvement and convergence of Algorithm \ref{alg:exact_policy_iteration} in Example 2}
    \end{figure}

    We now implement Algorithm \ref{alg:exact_policy_iteration} in multi-dimensional settings with \(d\in\{5,10\}\). To examine policy improvement and convergence, we consider the American basket put stopping problem. The simulation settings are as follow: \(T=1,I=50,\lambda=1\), and \(N=10\). The discount factor is \(\beta_t\equiv 0.06\). The simulator is specified by the volatility and drift functions \(\widetilde{\sigma}^o(t,x)=\operatorname{diag}(0.4)x\) and \(\widetilde{b}^o(t,x)=0.06 x\), respectively, for \(x\in\mathbb R^d\), where $\operatorname{diag}(a)$ denotes the diagonal matrix with entries $a\in\mathbb{R}$. We consider ambiguity levels \(\varepsilon\in\{0,0.2,0.4\}\). Moreover, we set the running reward to \(r(\cdot)\equiv 0\), the initial price to \(x_0=(40,\dots,40)^{\top}\in\mathbb R^d\), and the strike price to \(\Gamma=40\). Last, the payoff is 
    \[
        R(x)=\mbox{$(\Gamma-\frac{1}{d}\sum_{i=1}^d x_i)^+,\quad x=(x_1,\ldots,x_d)^\top\in\mathbb R^d$}.
    \]
    
    As in Example~1, we compute the reference values \(R^{\mathrm{ref}}_{\varepsilon}\) for \(\varepsilon\in\{0,0.2,0.4\}\) using the deep backward scheme of \cite{pham2020deepBSDE}. For \(d=5\), we obtain \(R^{\mathrm{ref}}_{0}=2.086\), \(R^{\mathrm{ref}}_{0.2}=1.750\), and \(R^{\mathrm{ref}}_{0.4}=1.468\). For \(d=10\), we obtain \(R^{\mathrm{ref}}_{0}=1.337\), \(R^{\mathrm{ref}}_{0.2}=1.105\), and \(R^{\mathrm{ref}}_{0.4}=0.947\). 
    The results in Fig.~\ref{fig:policy_improvement_convergence_high_dim} illustrate policy improvement and convergence in high-dimensional settings, supporting the scalability of our algorithm.

    \begin{figure}[tbp] \label{fig:robust_high_dim}
    \begin{minipage}[t]{0.45 \linewidth} 
    \centering
    \includegraphics[height=3.6cm,width=5.4cm]{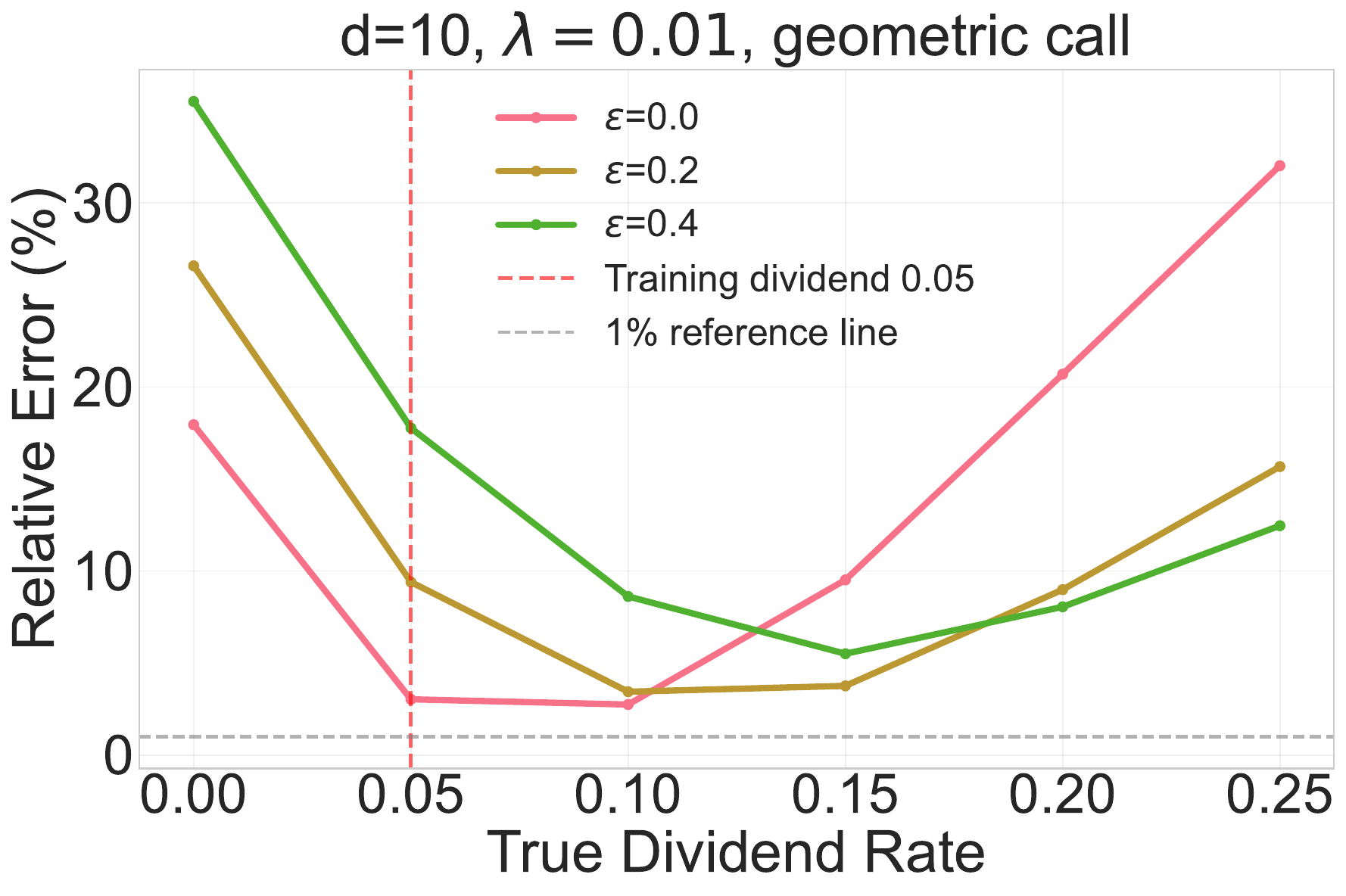}
    \label{fig:10}
    \end{minipage}%
    \begin{minipage}[t]{0.48\linewidth} 
    \centering
    \includegraphics[height=3.6cm,width=5.4cm]{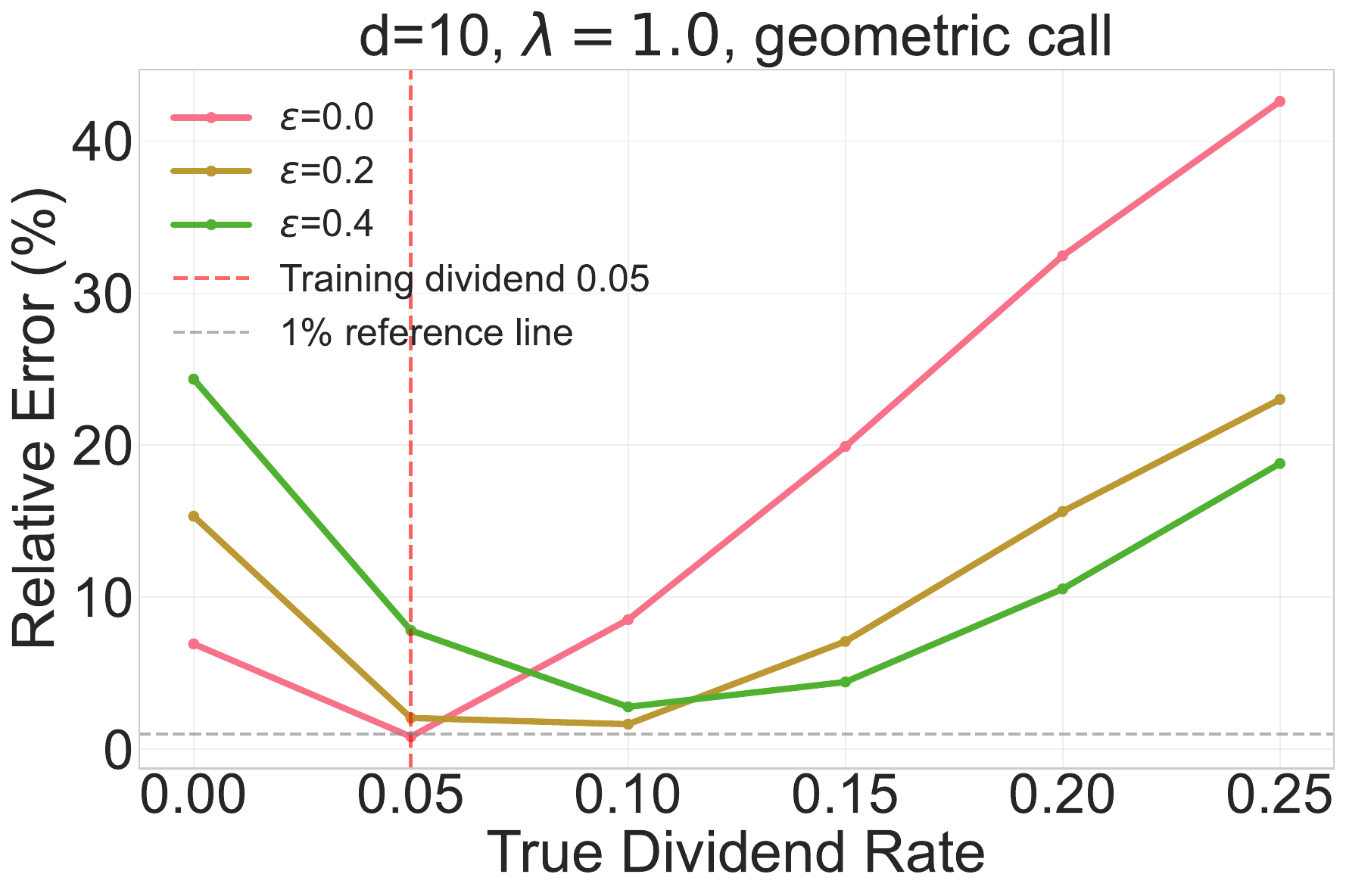}
    \label{fig:10}
    \end{minipage}%
     \caption{Robustness Performance under unknown testing environments in Example 2.}
    \end{figure}
    Next, for the robustness test, we consider the geometric call stopping problem with \(d=10\). The simulation settings are as follow: \(T=0.5\),\;\(I=100\), \(\lambda\in\{0.01,1\}\), and \(N=10\). The interest rate is specified by \(r_{*}=0.05\) and the ambiguity level varies over \(\varepsilon\in\{0,0.2,0.4\}\). The drift and volatility functions are specified by \(\widetilde{b}^{o}(t,x)=(r_{*}-\delta)x\) and \(\widetilde{\sigma}^{o}(t,x)=\operatorname{diag}(0.2)x\), $x\in\mathbb{R}^d$, respectively, where we let the dividend rate $\delta$ of the training simulator be \(\delta_{\mathrm{train}}=0.05\), whereas we let the dividend rate in the testing simulator vary over \(\delta\in\{0,0.05,0.1,0.15,0.2,0.25\}\). Last, the payoff is given by 
    \[
    \mbox{$R(x)=\big((\prod_{i=1}^d x_i)^{1/d}-\Gamma\big)^{+},\quad x=(x_1,\ldots,x_d)^{\top}\in\mathbb{R}^d.$}
    \]

    For the evaluation test, we adopt the stopping rule \(\tau^{\varepsilon,\lambda}_{\delta}\) and the expected reward~\(\check{R}_{\delta}^{\varepsilon,\lambda}\) from~\eqref{eq:explore_stopping_example1}, with the payoff replaced by the geometric call function above.

    Here we note that, due to the log-normal structure and the assumption of identical drift and volatility across dimensions, the 10-dimensional geometric call problem is equivalent to the one-dimensional problem with 
    \( \widetilde{b}^{o}(t,x) = 
    (r_{*}-\delta-\frac{(0.2)^2 }{2}(1-\frac1d )) x \)
    and 
    \(\widetilde{\sigma}^{o}(t,x) = \frac{0.2}{\sqrt d} x\), 
    $x \in \mathbb{R}.$ 
    Hence, we can apply the implicit finite-difference method of Forsyth and Vetzal~\cite{Forsyth-Penalty-02} to compute the reference values \(R^{\mathrm{ref}}_{\delta}\) in the one-dimensional setting for each \(\delta\). The corresponding reference values are given by \(R^{\mathrm{ref}}_{0}=1.061\), \(R^{\mathrm{ref}}_{0.05}=0.570\), \(R^{\mathrm{ref}}_{0.1}=0.350\), \(R^{\mathrm{ref}}_{0.15}=0.236\), \(R^{\mathrm{ref}}_{0.2}=0.172\), and \(R^{\mathrm{ref}}_{0.25}=0.134\). 
    The relative errors are defined by \(|\check{R}_{\delta}^{\varepsilon,\lambda}-R^{\mathrm{ref}}_{\delta}|/R^{\mathrm{ref}}_{\delta}\). 
    Fig.~\ref{fig:robust_high_dim} illustrates the robustness performance. As in Example~1, our robust RL algorithm performs better as the discrepancy between the training and testing environments increases.
    \color{black}

    \section{Proofs}\label{sec:proofs}
    \subsection{Proof of results in Section \ref{sec:robust_stopping}}\label{proof:pro:verification}
    \begin{proof}[Proof of Proposition \ref{pro:verification}]
    	{\it Step 1.} Fix $t\in[0,T]$ and let $\tau \in {\cal T}_t$. An application of It\^o's formula into $(e^{-\int_t^s \beta_udu}Y_s^{x})_{s\in[t,T]}$ ensures that 
    	\begin{align}\label{eq:Ito_formula}
    		\begin{aligned}
    			Y_t^x =&e^{-\int_t^\tau \beta_udu}Y_\tau^{x} +\int_t^\tau e^{-\int_t^s \beta_udu}\big(r(X_s^x)+g(s,Z_s^x)\big)ds\\
    			&-\int_t^\tau e^{-\int_t^s \beta_udu} Z_s^x dB_s+\int_t^\tau e^{-\int_t^s \beta_udu} dK_s^x.
    		\end{aligned}
    	\end{align}
    	
    	Since $\operatorname{I}_t^{x;\tau}\in L^2({\cal F}_\tau;\mathbb{R})$ (see Remark \ref{rem:wellposed_stopping}), $dK_s^x\geq 0$ for all $s\geq [t,\tau]$ (as $K^x$ is nondecreasing) and $Y_\tau^{x} \geq R(X_\tau^x)$ $\mathbb{P}$-a.s. (see Definition \ref{def:reflected_gBSDE}), it holds that $\mathbb{P}$-a.s.
    	\begin{align}
    		{\cal E}_t^g[\operatorname{I}_t^{x;\tau}]&\leq {\cal E}_t^g\bigg[Y_t^x-\int_t^\tau e^{-\int_t^s \beta_udu} g(s,Z_s^x)ds+\int_t^\tau e^{-\int_t^s \beta_udu} Z_s^x dB_s\bigg]\nonumber\\
    		&=Y_t^x +{\cal E}^g_t\bigg[-\int_t^\tau e^{-\int_t^s \beta_udu} g(s,Z_s^x)ds+\int_t^\tau e^{-\int_t^s \beta_udu} Z_s^x dB_s\bigg]\label{eq:trans0} \\
    		&
            =Y_t^x +{\cal E}^g_t\bigg[-\int_t^\tau g(s,e^{-\int_t^s \beta_udu} Z_s^x)ds+\int_t^\tau e^{-\int_t^s \beta_udu} Z_s^x dB_s\bigg]=:Y_t^x+ \operatorname{II}_t,\nonumber
    	\end{align}
    	where the first equality holds by the property of ${\cal E}_t^g[\cdot]$ given in \cite[Lemma 2.1]{coquet2002filtration}. 
        The second equality holds since $0<e^{-\int_t^s \beta_udu}\leq 1$ (as $\beta_t\geq 0$ for all $t\in[0,T]$; Assumption~\ref{as:reward}\;(ii)), and for every $\omega\in \Omega$ and $t\in[0,T]$, it holds that  $g(\omega,t,cz)= cg(\omega,t,z)$ for all $z\in \mathbb{R}^d$ and $c\in(0,1)$; see Definition \ref{def:g_expect}\;(iii).

        
        Moreover, since $Z^x\in \mathbb{L}^2(\mathbb{R}^d)$, $(\widetilde Z_s^x)_{s\in[t,T]}:=(e^{-\int_t^s \beta_udu} Z_s^x)_{s\in[t,T]}$ is $\mathbb{F}$-predictable and satisfies $\mathbb{E}[\int_t^T|\widetilde Z_s^x|^2ds]<\infty$. Therefore, the integrand of $\operatorname{II}_t$ is ${\cal E}^g$-martingale and its conditional $g$-expectation $\operatorname{II}_t$ equals zero; see \cite[Lemma\;5.5]{coquet2002filtration}. 
        \color{black}

	Combining this with \eqref{eq:trans0}, we obtain that ${\cal E}_t^g[\operatorname{I}_t^{x;\tau}]\leq Y_t^x$ $\mathbb{P}$-a.s.. Since $\tau\in {\cal T}_t$ is chosen some arbitrary, we have $V_t^x=\esssup_{\tau \in {\cal T}_t}{\cal E}_t^g[\operatorname{I}_t^{x;\tau}]\leq Y_t^x.$
	
    \vspace{0.5em}
	\noindent {\it Step 2.}  We now claim that $ Y_t^x\leq V_t^x$. Let $\tau_{t}^{*,x}\in {\cal T}_t$ be defined as in \eqref{eq:robust_OST}. Since $\int_0^{\tau_{t}^{*,x}}(Y_{s}^x- R(X_{s}^x))dK_s^{x}=0$ $\mathbb{P}$-a.s. (see Definition \ref{def:reflected_gBSDE}\;(iii)) 
	and $Y_{s}^x> R(X_{s}^x)$ for all $s \in(0, \tau_{t}^{*,x})$ (by definition of $\tau_{t}^{*,x}$),  it holds that 
	\begin{align}\label{eq:vanish_K}
		dK_s^{x}=0\quad \mbox{$\mathbb{P}$-a.s., for all $s\in(0,\tau_{t}^{*,x})$}.
	\end{align}
	
	Applying It\^o's formula as given in \eqref{eq:Ito_formula} and using \eqref{eq:vanish_K}, we obtain that $\mathbb{P}$-a.s.
	\begin{align}\label{eq:Ito_formula1}
            \begin{aligned}
		Y_t^x =&e^{-\int_t^{\tau_{t}^{*,x}} \beta_udu}Y_{\tau_{t}^{*,x}}^{x} +\int_t^{\tau_{t}^{*,x}} e^{-\int_t^s \beta_udu}\Big(r(X_s^x)+g(s,Z_s^x)\Big)ds\\
        &-\int_t^{\tau_{t}^{*,x}} e^{-\int_t^s \beta_udu} Z_s^x dB_s.
            \end{aligned}
	\end{align}
	
	By putting $\int_t^{\tau_{t}^{*,x}}e^{-\int_t^s \beta_udu}g(s,Z_s^x)ds-\int_t^{\tau_{t}^{*,x}} e^{-\int_t^s \beta_udu} Z_s^x  dB_s$ into the left-hand side of \eqref{eq:Ito_formula1} and taking the conditional $g$-expectation ${\cal E}_t^g[\cdot]$, $\mathbb{P}$-a.s.,
	\begin{align}
			&{\cal E}_t^g\bigg[\int_t^{\tau_{t}^{*,x}} e^{-\int_t^s \beta_udu}r(X_s^x)ds+ e^{-\int_t^{\tau_{t}^{*,x}} \beta_udu}Y_{\tau_{t}^{*,x}}^{x}\bigg] \nonumber \\
			&
            \;=Y_t^x+{\cal E}_t^g\bigg[-\int_t^{\tau_{t}^{*,x}}e^{-\int_t^s \beta_udu}g(s,Z_s^x)ds+\int_t^{\tau_{t}^{*,x}}  e^{-\int_t^s \beta_udu} Z_s^x dB_s \bigg]=Y_t^x, \label{eq:trans2} 
	\end{align}
    where the first equality follows from the property of ${\cal E}_t^g[\cdot]$ given in \cite[Lemma 2.1]{coquet2002filtration} and the second equality follows from the same arguments presented to show that the term $\operatorname{II}_t$ in \eqref{eq:trans0} equals zero; see also Definition \ref{def:g_expect}\;(iii). 
	
	Last, since $Y_{\tau_{t}^{*,x}}^{x}= R(X_{\tau_{t}^{*,x}}^x)$ $\mathbb{P}$-a.s. (by definition of~$\tau_{t}^{*,x}$ in \eqref{eq:robust_OST}), we have 
	\[
		Y_t^x= {\cal E}_t^g[\operatorname{I}_t^{x;\tau_{t}^{*,x}}]\leq V_t^x=\esssup_{\tau \in {\cal T}_t}{\cal E}_t^g[\operatorname{I}_t^{x;\tau}]\quad\mbox{$\mathbb{P}$-a.s.,} 
	\]
	as claimed. Therefore, $\tau_{t}^{*,x}$ is optimal to \eqref{eq:robust_stopping}. This completes the~proof. 
    \end{proof}
    \begin{proof}[Proof of Proposition \ref{pro:verification_control}]
		{\it Step 1.} Let $N\in \mathbb{N}$ and $\alpha\in {\cal A}$ be given. 
		Recalling $F^x$ given in~\eqref{eq:benchmark_gen}, we denote for every $(\omega,t,y,z)\in \Omega\times [0,T]\times \mathbb{R}\times \mathbb{R^d}$ by 
		\begin{align}\label{eq:controlled_generator}
			\widetilde F_t^{x;N,\alpha}(\omega,y,z):= F_t^x(\omega,y,z)+N\alpha_t(\omega)\,\big(R(X_t^x(\omega))-y\big).
		\end{align}

		Then  consider the following controlled BSDE: for $t\in[0,T]$
		\begin{align}\label{eq:controlled_BSDE}
			\widetilde Y_t^{x;N,\alpha}= R (X_T^x) +  \int_t^T \widetilde F^{x;N,\alpha}_s\big(\widetilde Y_s^{x;N,\alpha},\widetilde Z_s^{x;N,\alpha}\big) ds-\int_t^T \widetilde Z_s^{x;N,\alpha} dB_s.
		\end{align}
		
		Since $\alpha$ is uniformly bounded (noting that it has values only in $\{0,1\}$), one can deduce that the parameters of the BSDE \eqref{eq:controlled_BSDE} satisfies all the conditions given in \cite[Section 3]{pardoux1990adapted}. Hence, 
        there exists a unique solution $(\widetilde Y_t^{x;N,\alpha},\widetilde Z_t^{x;N,\alpha})_{t\in [0,T]}\in \mathbb{S}^2(\mathbb{R})\times \mathbb{L}^2(\mathbb{R}^d)$ to the controlled BSDE \eqref{eq:controlled_BSDE}. 
		
		We now claim that $\widetilde Y_t^{x;N,\alpha}={\cal E}_t^g[\operatorname{I}_t^{x;N,\alpha}]$ for all $t\in[0,T]$. Indeed, applying  It\^o's formula into $(e^{-\int_t^s (\beta_u+N\alpha_u)du}\widetilde Y_s^{x;N,\alpha})_{s\in[t,T]}$ and then taking ${\cal E}_t^g[\cdot]$ yield, 
		\begin{align*}
			&
            {\cal E}_t^g[\operatorname{I}_t^{x;N,\alpha}]-\widetilde Y_t^{x;N,\alpha}\\
            &\quad ={\cal E}_t^g\bigg[-\int_t^T e^{-\int_t^s (\beta_u+N\alpha_u)du}g(s,\widetilde Z_s^{x;N,\alpha} )ds+\int_t^T  e^{-\int_t^s (\beta_u+N\alpha_u)du} \widetilde Z_s^{x;N,\alpha}  dB_s\bigg],
		\end{align*}
		where we have used the property of ${\cal E}_t^g[\cdot]$ given in \cite[Lemma 2.1]{coquet2002filtration}.
		
        Moreover, by using the same arguments presented for the ${\cal E}^g$-martingale property in \eqref{eq:trans0} 
		(see the proof of Proposition \ref{pro:verification}) we can deduce that the conditional $g$-expectation appearing in the right-hand side of the above equals zero (i.e., the integrand therein is an ${\cal E}^g$-martingale). Hence the claim holds. \color{black}
		
		\vspace{0.5em}
		\noindent {\it Step 2.} It suffices to show that for every $t\in [0,T]$ $\mathbb{P}$-a.s., \(
			Y_t^{x;N}=\esssup_{\alpha \in {\cal A}} \widetilde Y_t^{x;N,\alpha}.
		\) 
		Indeed, it follows from Step 1 that for every $\alpha\in {\cal A}$ the parameters of the BSDE \eqref{eq:controlled_BSDE} satisfies the conditions given in 
		\cite[Section~3]{pardoux1990adapted}. Furthermore, the parameters of the BSDE \eqref{eq:g_BSDE} also satisfies the conditions (see Remark \ref{rem:penalized_BSDE}\;(i)).
		
		We recall that $F^{x;N}$ given in \eqref{eq:benchmark_gen_N} is the generator of \eqref{eq:g_BSDE} and that for each $\alpha \in {\cal A}$ $\widetilde F^{x;N,\alpha}$ given in \eqref{eq:controlled_generator} is the generator of \eqref{eq:controlled_BSDE}. Then for any $\alpha\in {\cal A}$, it holds that for all $(\omega,t,y,z)\in \Omega \times[0,T]\times \mathbb{R}\times \mathbb{R}^d$
		\begin{align*}
			F_t^{x;N}(\omega,y,z)= F_t^{x}(\omega,y,z)+N\max_{a\in \{0,1\}}\Big\{\big(R(X_t^x(\omega))-y\big)a\Big\} \geq  \widetilde F_t^{x;N,\alpha}(\omega, y,z).
		\end{align*}
	This ensures that for every $t\in[0,T]$,
		\begin{align}\label{eq:compare_generator}
			F^{x;N}_t(Y_t^{x;N},Z_t^{x;N})\geq \esssup_{\alpha \in {\cal A}} \widetilde F_t^{x;N,\alpha}(Y_t^{x;N},Z_t^{x;N}).
		\end{align}

		Moreover, let $\alpha^{*,x;N}$ be defined as in \eqref{eq:const_alpha_0}. Clearly, it takes values in $\{0,1\}$. Moreover, since $Y^{x;N}$ is in $\mathbb{S}^2(\mathbb{R})$ (see Remark \ref{rem:penalized_BSDE}\;(i)) and $(R(X_t^x))_{t\in[0,T]}$ are $\mathbb{F}$-progressively measurable (noting that $X^x$ is It\^o $(\mathbb{F},\mathbb{P})$-semimartingale and $R$ is continuous), $\alpha^{*,x;N}$ is $\mathbb{F}$-progressively measurable. Therefore, we have that $\alpha^{*,x;N}\in{\cal A}$.
		
		Moreover, by definition of $\alpha^{*,x;N}$, 
		\(
		\widetilde{F}_t^{x;N,\alpha^{*,x;N}}(Y_t^{x;N},Z_t^{x;N})=F_t^{x;N}(Y_t^{x;N},Z_t^{x;N}).
		\) 
		This implies that the inequality given in \eqref{eq:compare_generator} holds as equality. 
		
		Therefore, an application of \cite[Proposition 3.1]{el1997backward} ensures the claim to hold.
		
		\vspace{0.5em}
		\noindent {\it Step 3.} Lastly, it follows from \cite[Corollary 3.3]{el1997backward} that the process $\alpha^{*,x;N}\in{\cal A}$ is optimal for the problem given in Step 2., i.e., for all $t\in[0,T]$
		\(
		\esssup_{\alpha \in {\cal A}} \widetilde Y_t^{x;N,\alpha}=\widetilde Y_t^{x;N,\alpha^{*,x;N}}.
		\) 
		This completes the proof.
        \end{proof}

    \subsection{Proof of results in Section \ref{subsec:robust_exploratory}}\label{proof:thm:verification}
    \begin{proof}[Proof of Theorem \ref{thm:verification}]
        Let $N\in\mathbb{N}$ and $\lambda> 0$ be given.  We prove (i) by showing that the parameters of the BSDE \eqref{eq:g_BSDE_explore_optimal} satisfy all the conditions given in \cite[Section~3]{pardoux1990adapted} to ensure its existence and uniqueness to hold. 
        
        As $r$ is a Borel function and both $(\beta_t)_{t\in[0,T]}$ and $(g(t,z))_{t\in[0,T]}$ are $\mathbb{F}$-progressively measurable for all $z\in \mathbb{R}^d$, $(\overline F_t^{x;N,\lambda}(y,z))_{t\in[0,T]}$ given in \eqref{eq:optimal_explore_gen} is  $\mathbb{F}$-progressively measurable for all $(y,z)\in \mathbb{R}\times \mathbb{R}^d$. Moreover, since $g(\omega,t,0)=0$ for all $(\omega,t)\in\Omega\times [0,T]$ (see Definition\;\ref{def:g_expect}\;(iii)), by the growth conditions of $r$ and $R$ (see Assumption\;\ref{as:reward}\;(i)) and Remark \ref{rem:refer_semi}\;(i), it holds that $\|\overline F^{x;N,\lambda}_\cdot(0,0)\|_{\mathbb{L}^2}<\infty$ and $\|R(X_\cdot^x)\|_{\mathbb{L}^2}<\infty$. 
		

        By the regularity of $g$ given in Definition \ref{def:g_expect}\;(ii) and the boundedness of $(\beta_t)_{t\in[0,T]}$ (see Assumption \ref{as:reward}\;(ii)), for every $(\omega,t)\in\Omega\times [0,T]$, $y,\hat y\in \mathbb{R}$ and $z,\hat z\in \mathbb{R}^d$
	\begin{align}\label{eq:priori_bench_gen}
		\begin{aligned}
		|F_t^x(\omega,y,z)-F_t^x(\omega,\hat y,\hat z)|&\leq \beta_t(\omega) |y-\hat y| + |g(\omega,t,z)-g(\omega,t,\hat z)|\\
		&\leq (C_\beta+ \kappa)\big( |y-\hat y|+ |z-\hat z|\big).
		\end{aligned}
	\end{align}
        
        Moreover, since the map
		\begin{align}\label{eq:log_exp}
			h^{N,\lambda}:\mathbb{R}\ni s \to h^{N,\lambda}(s):= \lambda \log(\exp(-N\lambda^{-1}\,s)+1)\in (0,+\infty)
		\end{align}
		is strictly decreasing and {\color{blue} \color{black} $N$-Lipschitz continuous}, we are able to see that 
		for every $\omega\in \Omega $, $t\in[0,T]$, and $y,\hat y \in \mathbb{R}$
        \begin{align}
		|G^{x;N,\lambda}_t(\omega,y)-G^{x;N,\lambda}_t(\omega,\hat y)|
            &\leq N \Big|\big(R(X_t^x(\omega))-y\big)-\big(R(X_t^x(\omega))-\hat y\big)\Big|\nonumber\\
            &\quad +\Big|h^{N,\lambda}\big(R(X_t^x(\omega))-y\big)- h^{N,\lambda}\big((R(X_t^x(\omega))-\hat y\big)\Big|\label{eq:est_GNL}\\
			&\leq  2N |y-\hat y|. \nonumber
	\end{align} 

        From \eqref{eq:priori_bench_gen} and \eqref{eq:est_GNL} and the definition of $\overline F^{x;N,\lambda}$ given in \eqref{eq:optimal_explore_gen}, it follows that the desired priori estimate of $\overline F^{x;N,\lambda}$ holds. Hence an application of \cite[Theorem~3.1]{pardoux1990adapted} ensures the existence and uniqueness of the solution of \eqref{eq:g_BSDE_explore_optimal}, as claimed.

    \vspace{0.5em}
     \noindent We now prove (ii). By the representation given in \eqref{eq:g_exp_explore}, it  suffices to show that $\mathbb{P}$-a.s. \(
			\overline Y_t^{x;N,\lambda}=\esssup_{\pi \in \Pi} \overline Y_t^{x;N,\lambda,\pi}.\)
		
        Since  ${\cal H}$ is strictly convex on $[0,1]$ (see Remark \ref{rem:entropy}), it holds that 
		for every $(\omega,t,y,z)\in \Omega \times[0,T]\times \mathbb{R}\times \mathbb{R}^d$
		\begin{align}\label{eq:variational_form}
			\overline F_t^{x;N,\lambda}(\omega,y,z)= F_t^{x}(\omega,y,z)+\max_{a\in [0,1]}\Big\{N(R(X_t^x(\omega))-y)a-\lambda {\cal H}(a) \Big\}, 
		\end{align}
		where the equality holds by the first-order-optimality condition with the corresponding maximizer $
			a^*= (1+e^{-{N}{\lambda}^{-1}(R(X_t^x(\omega))-y)})^{-1}\in [0,1].$
		
		
		Then it follows from \eqref{eq:variational_form} that $\overline F_t^{x;N,\lambda}(\omega,y,z) \geq  \overline F_t^{x;N,\lambda,\pi}(\omega,y,z)$ for all $\pi\in \Pi$ and $(\omega,t,y,z)\in \Omega \times[0,T]\times \mathbb{R}\times \mathbb{R}^d$. This ensures that for every $t\in[0,T]$,
		\begin{align}\label{eq:compare_explore_generator}
			\overline F_t^{x;N,\lambda}\big(\overline Y_t^{x;N,\lambda},\overline Z_t^{x;N,\lambda}\big)\geq \esssup_{\pi \in {\cal A}} \overline F_t^{x;N,\lambda,\pi}(\overline Y_t^{x;N,\lambda},\overline Z_t^{x;N,\lambda}).
		\end{align}
		
		Moreover, let $\pi^{*,x;N,\lambda}:=(\pi^{*,x;N,\lambda}_t)_{t\in[0,T]}$ be defined as in \eqref{eq:robust_control}. Clearly, it takes values in $[0,1]$. Moreover, since $\overline Y^{x;N,\lambda}$ is in $\mathbb{S}^2(\mathbb{R})$ (see part (i)) and $(R(X_t^x))_{t\in[0,T]}$ are $\mathbb{F}$-progressively measurable (noting that $X^x$ is It\^o $(\mathbb{F},\mathbb{P})$-semimartingale and $R$ is continuous), $\pi^{*,x;N,\lambda}$ is $\mathbb{F}$-progressively measurable. Therefore, we have that $\pi^{*,x;N,\lambda}_t\in \Pi$.
		
		Furthermore, by \eqref{eq:variational_form} and definition of $\pi^{*,x;N,\lambda}$, it holds that
		\[
		\overline F_t^{x;N,\lambda,\pi^{*,x;N,\lambda}}(\overline Y_t^{x;N,\lambda},\overline Z_t^{x;N,\lambda})=\overline F_t^{x;N,\lambda}\big(\overline Y_t^{x;N,\lambda},\overline Z_t^{x;N,\lambda}\big),
		\]
		which implies that the inequality given in \eqref{eq:compare_explore_generator} holds as equality. 
		
		Therefore, an application of \cite[Proposition 3.1]{el1997backward} ensures the claim to hold.
        
        Moreover, a direct application of \cite[Corollary 3.3]{el1997backward} ensures that $\pi^{*,x;N,\lambda}$ is optimal for $\overline V^{x;N,\lambda}$ given in \eqref{eq:robust_explore_problem}. This completes the proof.
    \end{proof}
    \begin{proof}[Proof of Theorem \ref{thm:stability}]
    Let $N\in \mathbb{N}$ and $\lambda>0$ be given. Recall that $\overline F^{x;N,\lambda}$ and $F^{x;N}$, given in \eqref{eq:optimal_explore_gen} and \eqref{eq:benchmark_gen_N}, respectively, are the generators of the BSDEs \eqref{eq:g_BSDE_explore_optimal} and \eqref{eq:g_BSDE}, respectively. Then set for every $(\omega,t,y,z)\in \Omega \times [0,T]\times\mathbb{R}\times \mathbb{R}^d$ 
		\begin{align}
				\Delta \overline F_t^{x;N,\lambda}(\omega,y,z):=&\overline F_t^{x;N,\lambda}(\omega,y,z)-F_t^{x;N}(\omega,y,z)\nonumber\\
				=&h^{N,\lambda}(R(X_t^x(\omega))-y\big)+N \big(R(X_t^x(\omega))-y) {\bf 1}_{\{y>R(X_t^x(\omega))\}},\label{eq:compare_explore}
		\end{align}
		where we recall that the map $h^{N,\lambda}$ is given in \eqref{eq:log_exp}. 
		
		Since the map $h^{N,\lambda}$ is positive and satisfies that $h^{N,\lambda}(s)= -Ns + h^{N,\lambda}(-s)$ for all $s\in \mathbb{R}$, it holds that for every $(\omega,t,y,z)\in \Omega \times [0,T]\times\mathbb{R}\times \mathbb{R}^d$ 
		\begin{align}
			\Delta \overline F_t^{x;N,\lambda}(\omega,t,y,z) &\geq 
				\bigg[h^{N,\lambda}\big(R(X_t^x(\omega))-y\big)+N \big(R(X_t^x(\omega))-y\big) \bigg] {\bf 1}_{\{y>R(X_t^x(\omega))\}}\nonumber\\
				& = h^{N,\lambda}(-(R(X_t^x(\omega))-y)) {\bf 1}_{\{y>R(X_t^x(\omega))\}}\geq 0.\label{eq:lb_F}
		\end{align}
		
		Moreover, as the terminal conditions of \eqref{eq:g_BSDE_explore_optimal} and \eqref{eq:g_BSDE} are coincide, it follows~from the comparison principle of BSDEs (see, e.g., \cite[Theorem 2.2]{el1997backward}) that \eqref{eq:compare} holds.

		\vspace{0.5em}
		It remains to show that \eqref{eq:stability} holds. Set  for every $N\in \mathbb{N}$ and $\lambda>0$,
		\begin{align}
			\Delta{Y}^{x;N,\lambda} :=\overline{Y}^{x;N,\lambda}-Y^{x;N},\qquad  \Delta{Z}^{x;N,\lambda} :=\overline{Z}^{x;N,\lambda}-Z^{x;N}.
		\end{align} 
		
		Since the parameters of the BSDEs \eqref{eq:g_BSDE_explore_optimal} and \eqref{eq:g_BSDE} satisfy the conditions given in \cite[Section~5]{el1997backward} (with exponent $2$)  for all $N\in \mathbb{N}$ and $\lambda>0$, we are able to apply \cite[Proposition 5.1]{el1997backward} to have the following a priori estimates:\footnote{In \cite[Section~5]{el1997backward}, the filtration (denoted by $({\cal F}_t)$ therein) is set to be right-continuous and~complete (and hence not necessarily the Brownian filtration, as in our case). Nevertheless, we can still apply the stability result given in \cite[Proposition 5.1]{el1997backward}, since the martingales $M^i$, $i = 1, 2$, appearing therein are orthogonal to the Brownian motion. Consequently, the arguments remain valid when the general filtration is replaced with the Brownian one.} 
		for every $N\in \mathbb{N}$ and $\lambda>0$
		\begin{align}\label{eq:priori_peng}
			\|\Delta{Y}^{x;N,\lambda}\|_{\mathbb{S}^2}+\|\Delta{Z}^{x;N,\lambda}\|_{\mathbb{L}^2}\leq C  \mathbb{E} \bigg[\int_0^T|\Delta \overline F_t^{x;N,\lambda}(Y_t^{x,N},Z_t^{x;N}) |^2dt  \bigg]^{\frac{1}{2}},
		\end{align}
		with some $C>0$ (depending on $T$ but not on $N$,$\lambda$), and $\Delta \overline F{}^{x;N,\lambda}$ given in~\eqref{eq:compare_explore}.
		
		We note that $h^{N,\lambda}(s)=\lambda\log(\exp(-N\lambda^{-1} s)+1)\leq  \lambda\log 2$ for all $s\geq 0$. On the other hand, a simple calculation ensures for every $N\in \mathbb{N}$ and $\lambda>0$ that the map
		\[
		\overline h^{N,\lambda}:[0,\infty)\ni s\to \overline h^{N,\lambda}(s):= h^{N,\lambda}(-s)-Ns = \lambda \log(\exp({N}{\lambda}^{-1} s)+1)-Ns
		\]
		is (strictly) decreasing. This implies that $\overline h^{N,\lambda}(s)\leq  \overline h^{N,\lambda}(0)= \lambda \log 2$ for all $s\geq 0$. 
		
		From these observations and \eqref{eq:lb_F}, we have for every $N\in \mathbb{N}$, $\lambda>0$, and $t\in[0,T]$
		\begin{align}
				0\leq\Delta \overline F{}^{x;N,\lambda}_t(Y_t^{x,N},Z_t^{x;N})=& h^{N,\lambda}\Big(-\big(Y_t^{x,N}-R(X_t^x)\big)\Big) {\bf 1}_{\{Y_t^{x,N}\leq R(X_t^x)\}}\nonumber \\
				& + \overline h^{N,\lambda}\big(Y_t^{x,N}-R(X_t^x)\big) {\bf 1}_{\{Y_t^{x,N}>R(X_t^x)\}}
				\leq \lambda \log 2.\label{eq:ub_F}
		\end{align}
		
		Combining \eqref{eq:ub_F} with \eqref{eq:priori_peng} concludes that for every $N\in \mathbb{N}$ and $\lambda>0$ the estimate in \eqref{eq:stability} holds, 
		as claimed. This completes the proof. 
        \end{proof}

    The following lemma establishes the monotone convergence of the hitting times of an increasing sequence of continuous processes, and it will serve as a key ingredient in the proof of Theorem~\ref{thm:explore_stop_conv}.

    \begin{lemma}
    \label{lem:general_tau_conv}
    Let $ ({L})_{t\in[0,T]}$, $(\mathcal{Y}_t)_{t\in[0,T]}$, and $(\mathcal{Y}_t^n)_{t\in[0,T]}$, $n \in \mathbb N$, be real-valued, $\mathbb F$-progressively measurable continuous processes satisfying that
    \begin{itemize}
        \item[(i)] for every $n\in\mathbb N$ and $t\in[0,T]$,
        \(
        \mathcal{Y}_t^n\le \mathcal{Y}_t^{n+1}\le \mathcal{Y}_t\)
        $\mathbb P\text{-a.s.}$,
        \item[(ii)] 
        \(
         \Delta^n := \sup_{t\in[0,T]} (\mathcal{Y}_t-\mathcal{Y}_t^n) \downarrow 0
        \)
        $\mathbb P\text{-a.s.}$, 
        as $n \uparrow \infty$.
    \end{itemize}
    Moreover, we let 
    \(
    \tau_t^n:=\inf\{s\in[t,T]: \mathcal{Y}_s^n\le L_s\}\wedge T\), 
    \(
    \tau_t:=\inf\{s\in[t,T]: \mathcal{Y}_s\le L_s\}\wedge T,
    \)
    $t\in[0,T]$. 
    Then 
    it holds for every $t\in[0,T]$ that as $n\uparrow \infty$, 
    \(
    \tau_t^n \uparrow \tau_t
    \)
    $\mathbb P\text{-a.s.}$.
    \end{lemma}

    \begin{proof}
    By the condition (i), 
    it holds that 
    \(
    \tau_t^n\le \tau_t^{n+1}\le \tau_t
    \)
    $\mathbb P\text{-a.s.}$, for any $n\in \mathbb{N}$. 
    Moreover, since the filtration $\mathbb{F}$  is right-continuous (see Section \ref{sec:notation}), there exists an $\mathbb F$-stopping time $\bar\tau_t$ such that
    \(
     \bar\tau_t=\lim_{n\to \infty}\tau_t^n\le \tau_t,
    \)
    $\mathbb P\text{-a.s.}$. 
    
    It suffices to show that $\bar\tau_t\ge \tau_t$ $\mathbb{P}$-a.s. To this end, from the conditions (i) and~(ii), we consider a set $\Omega_0\subseteq \Omega$ satisfying that $\mathbb P(\Omega_0)=1$ and every $\omega\in\Omega_0$ satisfies 
        \[
        \mathcal{Y}_s^n(\omega) \le \mathcal{Y}_s^{n+1}(\omega) \le \mathcal{Y}_s(\omega)
        \quad \mbox{for all $s\in[t,T]$ and $n\in\mathbb N$},\quad \mbox{and $ \Delta^n(\omega) \downarrow 0 $ as $ n \uparrow \infty $.}
        \]

    Now, fix $\omega\in\Omega_0$.
    If $\tau_t(\omega)=t$, then, by definition of $\tau_t$, we have $\mathcal{Y}_t(\omega)\le L_t(\omega)$. This implies 
    \(
    \mathcal{Y}_t^n(\omega)\le L_t(\omega)
    \)
    for all $n\in \mathbb{N}$, 
    and therefore $\tau_t^n(\omega)=t$ for all $n\in \mathbb{N}$. As a consequence, we have
    \(
    \bar\tau_t(\omega)=t=\tau_t(\omega).
    \)
    
    Next, suppose that $\tau_t(\omega)>t$. Let $\varepsilon\in(0,\tau_t(\omega)-t)$. By definition of $\tau_t(\omega)$,
    \[
    \mathcal{Y}_s(\omega)>L_s(\omega)
    \quad \text{for all } s\in[t,\tau_t(\omega)-\varepsilon].
    \]
    Since the map $[t,\tau_t(\omega)-\varepsilon]\ni s \mapsto \mathcal{Y}_s(\omega)-L_s(\omega)$ is continuous, we can define
    \[
    m_{\omega,\varepsilon}
    :=
    \min_{s\in[t,\tau_t(\omega)-\varepsilon]}
    (\mathcal{Y}_s(\omega)-L_s(\omega))
    >0.
    \]
    We let $N_{\omega,\varepsilon}\in\mathbb{N}$ so that 
    \(
     \Delta^n(\omega) <m_{\omega,\varepsilon}
    \)
    for all $n\geq N_{\omega,\varepsilon}$. Then for every $n\geq N_{\omega,\varepsilon}$,
    \[
    \mathcal{Y}_s^n(\omega)
    \ge
    \mathcal{Y}_s(\omega)- \Delta^n(\omega)
    >
    L_s(\omega)\quad \mbox{for all $s\in[t,\tau_t(\omega)-\varepsilon]$.}
    \]
    
    This implies that
    \(
    \tau_t^n(\omega)\ge \tau_t(\omega)-\varepsilon
    \)
    for all $n\geq N_{\omega,\varepsilon}$. By letting $n\to \infty$, we have 
    \(
    \bar\tau_t(\omega)\ge \tau_t(\omega)-\varepsilon .
    \)
    Since $\varepsilon>0$ is arbitrary, we conclude that
    \(
    \bar\tau_t(\omega)\ge \tau_t(\omega).
    \)

    Since this inequality holds for any $\omega \in\Omega_0$ and we have $\mathbb P(\Omega_0)=1$ (by definition of $\Omega_0$), it follows that $\bar\tau_t\ge \tau_t$ $\mathbb{P}$-a.s.. 
    This completes the proof.
    \end{proof}

    \begin{proof}[Proof of Theorem \ref{thm:explore_stop_conv}]
    {\it Step 1.} We start with showing that for any $\lambda>0$, $N\in \mathbb{N}$ and $t\in[0,T]$, $\overline{Y}_t^{x;N,\lambda}-\lambda (T-t)\log2\leq Y_t^{x}$ $\mathbb{P}$-a.s..

    For notational simplicity, let $(\Delta \overline Y^{x;N,\lambda},\Delta \overline Z^{x;N,\lambda}):=(\overline{Y}^{x;N,\lambda}-Y^{x},\overline{Z}^{x;N,\lambda}-Z^{x})$ and $(\Delta Y^{x;N},\Delta Z^{x;N}):=({Y}^{x;N}-Y^{x},{Z}^{x;N}-Z^{x})$.  
    
    Using \eqref{eq:reflected_gBSDE} and \eqref{eq:g_BSDE_explore_optimal} as well as \eqref{eq:optimal_explore_gen}, 
    we have 
    \begin{align}
        \nonumber \Delta \overline{Y}_t^{x;N,\lambda}
        &=
        \int_t^T
        \big(
        -\beta_s \Delta \overline{Y}_s^{x;N,\lambda}
        +
        g(s,\overline Z_s^{x;N,\lambda})-g(s,Z_s^x)
        +
        G_s^{x;N,\lambda}(\overline{Y}_s^{x;N,\lambda})
        \big) ds \\
        &\quad 
        -\int_t^T \Delta \overline Z_s^{x;N,\lambda} dB_s
        -(K_{T}^x-K_t^x).\label{eq:explore_minus_reflected}
    \end{align}

    Moreover, using \eqref{eq:reflected_gBSDE} and \eqref{eq:g_BSDE}, we have
    \begin{align}
        -(K_{T}^x-K_t^x)
        &=
        \int_t^T
        \big(
        \beta_s \Delta Y_s^{x;N}
        -
        g(s,Z_s^{x;N})+g(s,Z_s^{x})
        -
        N(R(X_s^x)-Y_s^{x;N})^+
        \big) ds \nonumber \\
        &\quad +
        \int_t^T \Delta Z_s^N dB_s+\Delta Y_t^{x;N}.
        \label{eq:K_elimination}
    \end{align}

    Substituting \eqref{eq:K_elimination} into \eqref{eq:explore_minus_reflected} gives
    \begin{align}
        \Delta \overline{Y}_t^{x;N,\lambda}
        &=
        \Delta Y_t^{x;N}-
        \int_t^T
        \beta_s(\overline{Y}_s^{x;N,\lambda}-Y_s^{x;N})ds\nonumber \\
        &\quad +\int_t^T \big(G_s^{x;N,\lambda}(\overline Y_s^{x;N,\lambda})
        -
        N(R(X_s^x)-Y_s^{x;N})^+\big) ds \label{eq:main_difference_identity}\\
        &\quad
        +
        \int_t^T
        \big(
        g(s,\overline Z_s^{x;N,\lambda})-g(s,Z_s^{x;N})
        \big)ds -
        \int_t^T  (\overline{Z}_s^{x;N,\lambda}-Z_s^{x;N})dB_s. \nonumber
    \end{align}

    From Remark~\ref{rem:penalized_BSDE}\;(iii) and \eqref{eq:compare} in Theorem \ref{thm:stability}, we have that  for every $t\in[0,T]$ 
    \begin{align}\label{eq:monotone_gaps}
        \Delta Y_t^{x;N}\leq 0\quad \mbox{and}\quad \overline{Y}_t^{x;N,\lambda}-Y_t^{x;N}\geq 0\quad \mbox{$\mathbb{P}$-a.s.}.
    \end{align}
    Moreover, since 
    \(
    \underline G_s^{x;N,\lambda}(y):=G_s^{x;N,\lambda}(y)-N(R(X_s^x)-y)^+
    \le
    \lambda\log2\)
    for every $y\in \mathbb{R}$ and $s\in[t,T]$, 
    and $\overline{Y}_t^{x;N,\lambda}\geq Y_t^{x;N}$ $\mathbb{P}$-a.s., we have that for all $s\in[t,T]$,
    \begin{align}\label{eq:G_minus_penalty_bound}
    G_s^{x;N,\lambda}(\overline Y_s^{x;N,\lambda})
        -
        N(R(X_s^x)-Y_s^{x;N})^+\leq \underline G_s^{x;N,\lambda}(\overline Y_s^{x;N,\lambda})\leq \lambda\log2 \quad \mbox{$\mathbb{P}$-a.s..}
    \end{align}

    From \eqref{eq:monotone_gaps} and \eqref{eq:G_minus_penalty_bound} as well as the $\kappa$-Lipschitz continuity of $g$ (see Definition~\ref{def:g_expect}\;(ii)), it follows that $\Delta \overline{Y}_t^{x;N,\lambda}$ given in \eqref{eq:main_difference_identity} satisfies, $\mathbb{P}$-a.s.,
    \begin{align*}
    \Delta \overline{Y}_t^{x;N,\lambda}
    -
    \lambda (T-t)\log2
    \leq
    \int_t^T  \kappa |\overline{Z}_s^{x;N,\lambda}-Z_s^{x;N}| ds
    -
    \int_t^T (\overline{Z}_s^{x;N,\lambda}-Z_s^{x;N})dB_s.
    \end{align*}
    
    Since $Z^{x;N},\overline{Z}^{x;N,\lambda}\in \mathbb{L}^2(\mathbb{R}^d)$ (see Remark \ref{rem:penalized_BSDE}\;(i) and Theorem \ref{thm:verification}), the right-hand side of the above inequality is a ${\cal E}^{-\kappa}$-martingale. Hence, its conditional ${\cal E}^{-\kappa}$-expectation at $t$ equals zero $\mathbb{P}$-a.s.. This proves the claim.

    \noindent \textit{Step 2: Proof of (i).}~We show that $R(X_\cdot^x)$, $Y^x$, and $(Y_t^{x;N_n}-\lambda_n (T-t)\log2)_{t\in[0,T]}$, $n\in\mathbb{N}$, satisfy all the conditions of Lemma \ref{lem:general_tau_conv}, so that for every $t\in[0,T]$, 
    \begin{align}\label{eq:temp_stop}
    \underline\tau_t^n
    :=
    \inf\{
    s\in[t,T]:
    Y_s^{x;N_n}-\lambda_n (T-s)\log2 \le R(X_s^x)
    \}\wedge T
    \end{align}
    goes to $\tau_t^{*,x}$ $\mathbb{P}$-a.s., as $n\to \infty$ (see \eqref{eq:robust_OST}). 
    

    First, we note that all those processes are $\mathbb{F}$-progressively measurable, continuous processes because $X^x$, $Y^x$, and $Y^{x;N_n}$, $n\in\mathbb{N}$, are $\mathbb{F}$-progressively measurable and continuous (see \eqref{eq:refer_semi}, \eqref{eq:reflected_gBSDE}, \eqref{eq:g_BSDE}), and $R$ is continuous (see Assumption \ref{as:reward}\;(i)). 
    
    Since 
    \(
    N_n\uparrow\infty\) and \(\lambda_n\downarrow0
    \) 
    as $n\to \infty$, by Remark \ref{rem:penalized_BSDE}\;(iii), we have that for every $n\in \mathbb{N}$ and $t\in[0,T]$, $Y^{x;N_n}_t-\lambda_n (T-t)\log2 \leq Y^{x;N_{n+1}}_t-\lambda_{n+1} (T-t)\log2\leq Y_t^x$ $\mathbb{P}$-a.s.. Therefore, the condition (i) in Lemma~\ref{lem:general_tau_conv} holds.
    

    Next, we set \(\Delta^{x,n}:=\sup_{t\in[0,T]}(Y_t^x-(Y^{x;N_n}_t-\lambda_n (T-t)\log2 ))\ge0\), which 
    is non-increasing \(\P\)-a.s.. Define \(\Delta^{x,\infty}:=\lim_{n\to\infty}\Delta^{x,n}\). Then by Fatou's lemma,
    \[
    \mathbb E\big[|\Delta^{x,\infty}|^2\big]\leq \liminf_{n\to \infty}\mathbb E\big[|\Delta^{x,n}|^2\big]\le \liminf_{n\to \infty}\big(\|Y^x-Y^{N_n}\|_{\mathbb S^2}^2 + \lambda_nT\log 2\big)= 0,
    \]
    where the last equality follows from Remark \ref{rem:penalized_BSDE}.
    Thus, the condition (ii) holds. 

    Consequently, all the conditions in Lemma~\ref{lem:general_tau_conv} are satisfied, we have by the lemma that \(\underline\tau_t^n\uparrow \tau_t^{*,x}\) \(\P\)-a.s. as $n\to \infty$.  

    Last, we have by Step 1 and \eqref{eq:compare} that for any $n\in \mathbb{N}$ and~$t\in[0,T]$, 
    \[
        Y^{x;N_n}_t-\lambda_n (T-t)\log2 \leq \overline Y^{x;N_n,\lambda_n}_t-\lambda_n (T-t)\log2 \leq Y_t^x\quad \mbox{$\mathbb{P}$-a.s..}
    \]
    This implies that for any $n\in \mathbb{N}$ and $t\in[0,T]$, \(\underline\tau_t^n \le \tau_t^{x;N_n,\lambda_n}\le \tau_t^{*,x}\) (see \eqref{eq:temp_stop}). Thus, the convergence, \(\underline\tau_t^n\uparrow \tau_t^{*,x}\) \(\P\)-a.s., ensures the statement (i) to hold.
    
    
    \noindent \textit{Step 3: Proof of (ii).}
    Since $\tau_t^{x;N_n,\lambda_n}\to \tau_t^{*,x}$ $\mathbb{P}$-a.s. (by Step 2) and $R$ is continuous, by the properties of the integrand $\operatorname{I}_t^{x;\tau}$ for any $\tau\in{\cal T}_t$ and $t\in[0,T]$ given in Remark~\ref{rem:wellposed_stopping}, we have that for any $t\in[0,T]$, $\|\operatorname{I}_t^{x;\tau_t^{x;N_n,\lambda_n}}- \operatorname{I}_t^{x;\tau_t^{*,x}}\|_{L^2}\to 0$ as $n\to\infty$.

    Moreover, since \(\operatorname{I}_t^{x;\tau}\in L^2({\cal F}_T)\) for any \(\tau\in{\cal T}_t\) and \(t\in[0,T]\), by \cite[Theorem 3.1]{pardoux1990adapted} we can define
     \((\widetilde Y^{t,x;\tau},\widetilde Z^{t,x;\tau})\in
    \mathbb S^2(\mathbb R)\times \mathbb L^2(\mathbb R^d)\) as a unique solution of
    \begin{align}\label{eq:BSDE_g_eval_stop}
        \widetilde Y_s^{t,x;\tau}
        =
       \operatorname{I}_t^{x;\tau} + \int_s^T g(u,\widetilde Z_u^{t,x;\tau})du
        -\int_s^T \widetilde Z_u^{t,x;\tau} dB_u,
        \quad s\in[0,T]. 
    \end{align}
    Clearly, it holds that for all $t\in[0,T]$, $\mathbb{P}$-a.s.,
    \[
        \widetilde Y_t^{t,x;\tau_t^{x;N_n,\lambda_n}}={\cal E}_t^g[\operatorname{I}_t^{x;\tau_t^{x;N_n,\lambda_n}}]\quad \mbox{and}\quad \widetilde Y_t^{t,x;\tau_t^{x;*}}={\cal E}^g_t[\operatorname{I}_t^{x;\tau_t^{*,x}}]=V_t^x.
    \]
    
    Since $g$ is $\kappa$-Lipschitz, the parameters of the BSDEs \eqref{eq:BSDE_g_eval_stop} satisfy the conditions given in \cite[Section~5]{el1997backward} (with exponent $2$) under any $\tau\in {\cal T}_t$ and $t\in[0,T]$, we are able to apply \cite[Proposition 5.1]{el1997backward} to have the following a priori estimate: there exists $C >0 $ (depends on $\kappa$ and $T$, but not on $n$) such that for any \( n\in \N \) and $t\in[0,T]$,
    \begin{align*}
        \big\|{\cal E}_t^g[\operatorname{I}_t^{x;\tau_t^{x;N_n,\lambda_n}}]-V_t^x  \big\|_{L^2}\leq \| \widetilde Y^{t,x;\tau_t^{x;N_n,\lambda_n}}- \widetilde Y^{t,x;\tau_t^{*,x}} \|_{\mathbb{S}^2} 
        \le C\|\operatorname{I}_t^{x;\tau_t^{x;N_n,\lambda_n}}- \operatorname{I}_t^{x;\tau_t^{*,x}}\|_{L^2}. 
    \end{align*}
    Since the last term vanishes as $ n \to \infty $, this completes the proof.
    \end{proof}
    \color{black}

    \subsection{Proof of results in Section \ref{sec:policy_iter}}\label{proof:thm:policy_improvement} 
    \begin{proof}[Proof of Theorem \ref{thm:policy_improvement}] We start by proving (i). Let $n\in \mathbb{N}$ be given. Since $\overline{Y}^{x;N,\lambda}_t\geq \overline{Y}^{x;N,\lambda,\pi}_t$ $\mathbb{P}$-a.s., for all $t\in[0,T]$ and $\pi\in \Pi$ (see Theorem \ref{thm:verification}\;(ii)), it suffices to show that $\overline Y_t^{x;N,\lambda,\pi^{n+1}}\geq \overline Y_t^{x;N,\lambda,\pi^{n}}$, $\mathbb{P}$-a.s., for all $t\in[0,T]$.
        
        For notational simplicity, let $(\overline Y^{n},\overline Z^{n}):=(\overline Y^{x;N,\lambda,\pi^{n}},\overline Z^{x;N,\lambda,\pi^{n}})$, $(\overline Y^{n+1},\overline Z^{n+1}):=(\overline Y^{x;N,\lambda,\pi^{n+1}},\overline Z^{x;N,\lambda,\pi^{n+1}}).$ 
        In analogy, let 
        $\overline F^{n}:=\overline F^{x;N,\lambda,\pi^{n}}$, $\overline F^{n+1}:=\overline F^{x;N,\lambda,\pi^{n+1}}$. 
        
        Then we set for every~$t\in[0,T]$ 
    \begin{align*}
        \phi_t:=(\overline F_t^{n+1}-\overline F_t^{n})(\overline Y_t^{n},\overline Z_t^{n}),\quad 
        \Delta Y_t:= \overline Y_t^{{n+1}}-\overline Y_t^{{n}},\quad 
        \Delta Z_t:=(\Delta Z_{t,1},\dots,\Delta Z_{t,d})^\top, 
    \end{align*}
    with $\Delta Z_{t,i}:=\overline Z_{t,i}^{{n+1}}-\overline Z_{t,i}^{{n}}$ for $i=1,\dots,d,$ where $\overline Z_{t,i}^{{n+1}}$ and $\overline Z_{t,i}^{{n}}$ denote the $i$-th component of $\overline Z_{t}^{{n+1}}$ and $\overline Z_{t}^{{n}}$, respectively. 
    
    Moreover, we denote for every $t\in[0,T]$ and $i=1,\dots,d$,
    \begin{align*}
    \begin{aligned}
    n_t:=&\frac{1}{\Delta Y_t}\Big(\overline F_t^{{n+1}}(\overline Y_t^{n+1},\overline Z_t^{n+1})-\overline F_t^{{n+1}}(\overline Y_t^{{n}},\overline Z_t^{{n+1}})\Big){\bf  1}_{\{\Delta Y_t\neq 0\}},\\
    m_{t,i}:=&\frac{1}{\Delta Z_{t,i}}\Big(\overline F^{{n+1}}_t(\overline Y_t^{n},(\overline Z_{t,1}^{n},\dots,\overline Z_{t,i-1}^{n},\overline Z_{t,i}^{n
    +1},\overline Z_{t,i+1}^{n
    +1},\dots,\overline Z_{t,d}^{n
    +1})^\top)  \\
    &\quad\quad\quad - \overline F^{{n+1}}_t(\overline Y_t^{n},(\overline Z_{t,1}^{n},\dots,\overline Z_{t,i-1}^{n},\overline Z_{t,i}^{n},\overline Z_{t,i+1}^{n
    +1},\dots,\overline Z_{t,d}^{n
    +1})^\top)\Big){\bf 1}_{\{\Delta Z_{t,i}\neq 0\}}.
    \end{aligned}
    \end{align*}

    Clearly, $(\Delta Y,\Delta Z)$ satisfies the following BSDE: for $t\in[0,T]$,
    \[
    \Delta Y_t = \int_t^T ( n_s\Delta Y_s + m_s^\top \Delta Z_s  + \phi_s )ds - \int_t^T \Delta Z_s  d B_s.
    \]
    Moreover, by construction \eqref{eq:robust_iter_control}, ${\pi}_t^{n+1} = \mathrm{argmax}_{a\in [0,1]}\{ N(R(X_t^x)-\overline Y_t^{n}) a-\lambda {\cal H}(a)\}$, for all $t\in [0,T].$ This ensures that $\phi_t\geq 0$ for all $t\in[0,T]$. 

    Clearly, it holds that $n_t=-(\beta_t+N\pi_t^{n+1}){\bf  1}_{\{\Delta Y_t\neq 0\}}$ for all $t\in[0,T]$. Moreover, by Assumption \ref{as:reward}\;(ii) and the fact that $\pi^{n+1}\in\Pi$ has values in $[0,1]$, $(n_t)_{t\in[0,T]}$ is uniformly bounded. Furthermore, by the Lipschitz property of $g$ (see Definition\;\ref{def:g_expect}\;(ii)), for every $i=1,\dots,d$, $(m_{t,i})_{t\in[0,T]}$ is uniformly bounded by $\kappa>0$. 

    Therefore, by letting {\color{blue} \color{black} $\Gamma_t := \exp( \int_0^t m_sdB_s + \int_0^t(n_s - \frac{1}{2}|m_s|^2 )ds)$} for $t\in[0,T]$, applying It\^o's formula into $(\Gamma_t\Delta Y_t)_{t\in[0,T]}$ and taking the conditional expectation $\mathbb{E}_t[\cdot]$, 
    \[
    \Delta Y_t = \Gamma_t^{-1} \mathbb{E}_t\bigg[\int_t^T \Gamma_s \phi_s ds \bigg],\quad \mbox{$\mathbb{P}$-a.s.,}\quad \mbox{for all}\;\;t\in[0,T].
    \]
    Since $\phi \ge 0$, we have $\Delta Y_t \ge 0$ $\mathbb{P}$-a.s., for all $t\in[0,T]$. 
    Therefore, the part~(i) holds.

    \vspace{0.5em}
    \noindent We now prove (ii). 
    Set for every $n\in \mathbb{N}$
    \[
    \overline F:=\overline F^{{x;N,\lambda}},\quad  \Delta^{n+1}\overline F :=\overline F - \overline F^{{n+1}},\quad \overline Y := \overline{Y}^{x;N,\lambda},\quad \Delta^n \overline{Y}_t := \overline{Y}_t - \overline{Y}^{n}_t 
    \]
    In analogy, set $\overline Z := \overline{Z}^{x;N,\lambda}$ and $\Delta^n \overline{Z}_t := \overline{Z}_t - \overline{Z}^{n}$. 

    We first note that for any $n\in \mathbb{N}$, $\omega \in \Omega $, $t\in[0,T]$, $y,\hat y \in \mathbb{R}$ and $z,\hat z\in \mathbb{R}^d$
    \[
        \begin{aligned}
		|\overline F_t^{{n+1}}(\omega,y,z)-\overline F_t^{{n+1}}(\omega,\hat{y},\hat{z})|&\leq (\beta_t(\omega)+N) |y-\hat y| + |g(\omega,t,z)-g(\omega,t,\hat z)| \\
		&\leq (C_\beta+ \kappa + N)\big( |y-\hat y|+ |z-\hat z|\big).
		\end{aligned}
    \]

    Set $C_1:=C_\beta+ \kappa + N>0$. By the a priori estimate in \cite[Theorem~4.2.3]{zhangjianfeng17}, there exists some $C_2>0 $ (that depends on $C_1,T,d$ but not on $n,\lambda$), such that\footnote{For any $t\in[0,T]$ and $Y\in \mathbb{S}^2(\mathbb{R})$, denote by $\|Y\|_{\mathbb{S}^2_t}^2:=\mathbb{E}[\sup_{s\in[t,T]}|Y_s|^2]$. In analogy, for any $Z\in \mathbb{L}^2(\mathbb{R}^d)$, denote by $ \|Z\|^2_{\mathbb{L}^2_t}:=\mathbb{E}[\int_t^T|Z_s|^2ds] $.} 
    \begin{align*}
                \|\Delta^{n+1} \overline Y\|_{\mathbb{S}^2_t}^2 + \|\Delta^{n+1} \overline Z\|_{\mathbb{L}^2_t}^2 & \le C_2 \mathbb{E}\bigg[\int_t^T \big| \Delta^{n+1}\overline{F}_s( \overline{Y}_s,\overline{Z}_s ) \big| ds \bigg]^2 \\
                & \le C_2T  \int_t^T \mathbb{E} \Big[\big| \Delta^{n+1}\overline{F}_s( \overline{Y}_s,\overline{Z}_s ) \big|^2\Big] ds\quad \mbox{for all $t\in[0,T]$},
    \end{align*}
    where we have used the Jensen's inequality with exponent $2$ for the last inequality.
    
    Moreover, by setting $L^{n}_s:= \frac{N}{\lambda}( R(X^x_s) - \overline{Y}^{n}_s)$ and $ L_s:= \frac{N}{\lambda}( R(X^x_s) - \overline{Y}_s)$ and noting that $\pi_s^{n+1}=(1+e^{-L_s^n})^{-1}$,  
    we compute that for all $s\in[t,T]$
    \begin{align*}
        \big| \Delta^{n+1}\overline{F}_s( \overline{Y}_s,\overline{Z}_s ) \big|
        & = \lambda \bigg| ( L_s - L^{n}_s ) - \frac{L_s - L^{n}_s}{1+e^{-L^{n}_s}} + \log(1+e^{-L^{n}_s} ) - \log(1+e^{-L_s} )  \bigg| \\
        & \le 3\lambda |L_s - L^{n}_s| =3N \big| \Delta^{n}\overline{Y}_s  \big| 
    \end{align*}
    where we have used the fact that $|\log(1+e^x) - \log(1+e^y)| \le |x-y|$ for all $x,y\in \mathbb{R} $.

    By setting $C_3:=9C_2TN^2>0$, we have shown that for all $t\in[0,T]$
    \begin{align}\label{eq:one_step_integ}
        \|\Delta^{n+1} \overline Y\|_{\mathbb{S}^2_t}^2 + \|\Delta^{n+1} \overline Z\|_{\mathbb{L}^2_t}^2 \le {C}_3 \int_t^T \mathbb{E} \Big[ \big| \Delta^{n} \overline{Y}_s \big|^2 \Big] ds \le C_3 \int_t^T \|\Delta^{n} \overline Y\|_{\mathbb{S}^2_s}^2 ds.
    \end{align}
    

    By using the same arguments presented for \eqref{eq:one_step_integ} iteratively, 
    \begin{align*}
         &\|\Delta^{n+1} \overline Y\|_{\mathbb{S}^2}^2 + \|\Delta^{n+1} \overline Z\|_{\mathbb{L}^2}^2 \le C_3 \int_0^T \|\Delta^{n} \overline Y\|_{\mathbb{S}^2_{t_n}}^2 dt_n\\
         &\qquad \le  ({C_3})^2 \int_0^T \int_{t_{n}}^T \|\Delta^{n-1} \overline Y\|_{\mathbb{S}^2_{t_{n-1}}}^2 dt_{n-1} \; dt_n \\
         &\qquad  \le \cdots  \le ({C}_3)^{n} \int_0^T \int_{t_{n}}^T \cdots \int_{t_2}^T \|\Delta^{1} \overline Y\|_{\mathbb{S}^2_{t_{1}}}^2 dt_1  \cdots dt_{n-1} \; dt_n \\
         &\qquad \le  ({C}_3)^{n} \|\Delta^{1} \overline Y\|_{\mathbb{S}^2}^2 \int_0^T \int_{t_{n}}^T \cdots \int_{t_2}^T 1\; dt_1  \cdots dt_{n-1} \; dt_n  = ({C}_3)^{n} \frac{T^n}{n!} \|\Delta^{1} \overline Y\|_{\mathbb{S}^2}^2,
    \end{align*}
    together with the 1-Lipschitz continuity of the logistic function $(1+e^{-x})^{-1}$, 
    we have
    \begin{align*}
        \| \pi^{n+1} - \pi^{*}\|^2_{\mathbb{S}^2} 
         \le \frac{N}{\lambda} \mathbb{E}\bigg[\sup_{t\in[0,T]}| \overline{Y}^{x;N,\lambda,\pi^{n}}_t - \overline{Y}^{x;N,\lambda}_t |^2 \bigg]= \frac{N}{\lambda} \|\Delta^n \overline{Y} \|_{\mathbb{S}^2}.
    \end{align*}
    The monotonicity of $\pi^{n+1}_t $ as $n\uparrow \infty $ is obvious from the logistic function form on $\overline{Y}^{x;N,\lambda,\pi^{n}} $, which completes the proof. 
    \end{proof} 

    Let us consider the following controlled forward-backward SDEs for any $\check \pi \in \check \Pi$: for any $(t,x)\in[0,T]\times \mathbb{R}^d$ and $s\in[0,T]$,
        \begin{align}\label{eq:controlled_fbsde}
            \begin{aligned}
            \check Y_s^{t,x;N,\lambda,\check\pi}&=R (\check X_T^{t,x}) + \int_s^T
		\check F_u^{N,\lambda,\check\pi}(\check X_u^{t,x},\check Y_u^{t,x;N,\lambda,\check\pi},\check Z_u^{t,x;N,\lambda,\check\pi}){\bf 1}_{\{u\geq t\}}du\\
        &\quad -\int_s^T \check Z_u^{t,x;N,\lambda,\check\pi} dB_u.
            \end{aligned}
        \end{align}
    	where $\check X_s^{t,x}=x+(\int_t^s\widetilde b^o(s,\check X_s^{t,x})ds+\widetilde \sigma^o(s,\check X_s^{t,x})dB_s){\bf1}_{\{s\geq t\}}$.
    	
        One can deduce that there exists a unique solution $(\check Y^{t,x;N,\lambda,\check\pi},\check Z^{t,x;N,\lambda,\check\pi})$ to \eqref{eq:controlled_fbsde} (see Remark \ref{rem:explore_control_BSDE}). In particular, 
        since $\check X^{0,x}=X^{x}$ (see \eqref{eq:refer_semi} and Remark~\ref{rem:suff_refer_semi}\;(ii)), $(\check Y^{0,x;N,\lambda,\check\pi},\check Z^{0,x;N,\lambda,\check\pi})$ is the unique solution $(\overline Y^{x;N,\lambda,\check \pi(X^x)},\overline Z^{x;N,\lambda,\check \pi(X^x)})$ to \eqref{eq:g_BSDE_explore} under $\check \pi(X^x)=(\check \pi_t(X_t^x))_{t\in[0,T]}\in \Pi$.

        Then we observe the following Markovian representation of \eqref{eq:controlled_fbsde}. %
        \begin{lemma}\label{lem:Markovian_under_closed_loop_policy}
            Under Setting~\ref{set:markov-BSDE}, let $\check \pi\in \check\Pi $ be given. 
            \begin{itemize}
                \item[(i)] There exist two Borel measurable functions $v^{N,\lambda,\check \pi}:[0,T]\times \mathbb{R}^d\to \mathbb{R}$ and $w^{N,\lambda,\check \pi}:[0,T]\times \mathbb{R}^d\to \mathbb{R}^d$ such that for every $t\leq s\leq T$, $\mathbb{P}\otimes ds$-a.e.,
                \begin{align}\label{eq:deter_YZ}
                \check Y_s^{t,x;N,\lambda,\check\pi}=v^{N,\lambda,\check \pi}(s,\check X_s^{t,x}),\quad \check Z_s^{t,x;N,\lambda,\check\pi}=\big((\widetilde \sigma^o)^\top w^{N,\lambda,\check \pi}\big)(s,\check X_s^{t,x}),
                \end{align}
                where $(\check Y^{t,x;N,\lambda,\check\pi},\check Z^{t,x;N,\lambda,\check\pi})$  is the unique solution of \eqref{eq:controlled_fbsde}. 
                \item[(ii)] Furthermore, if $\check\pi_t(\cdot)$ is continuous on $\mathbb{R}^d$ for any $t\in[0,T]$,  
                one can find a function $v^{N,\lambda,\check \pi}:[0,T]\times \mathbb{R}^d\to \mathbb{R}$ which satisfies the property given in \eqref{eq:deter_YZ} and is a 
                viscosity solution of the following PDE: 
                \[(\partial_t v+\mathcal{L} v)(t,x) + \check{F}^{N,\lambda,\check\pi }_t(x, v(t,x), ((\widetilde \sigma^o)^{\top}\nabla  v)(t,x)) =0,\quad (t,x)\in [0,T)\times \mathbb{R}^d,
                \]
                with $v(T,\cdot) = R(\cdot)$, where the infinitesimal operator $\mathcal{L}$ is defined as in Remark~\ref{rem:optimal_visco}. 
                In particular, $\check v^{N,\lambda , \check \pi}$ 
                is locally Lipschitz w.r.t.~$x$ and H\"older continuous w.r.t.~$t$ (Hence, it is continuous on $[0,T]\times \mathbb{R}^d$). 
            \end{itemize}
        \end{lemma}
        \begin{proof}
            We start with proving (i). According to \cite[Theorem 8.9]{KHM}, it suffices to show that the generator $\check F_\cdot^{N,\lambda,\check \pi}(\cdot,\cdot,\cdot) $ given in \eqref{eq:Markov_generator} satisfies the condition (M1b) given in \cite{KHM} (noting that $\check X^{t,x}$ given in \eqref{eq:controlled_fbsde} satisfies (M1f) therein; see Remark \ref{rem:refer_semi}). Note that 
        $\beta_t$ and $\check\pi_t(x)$ are uniformly bounded (see Setting \ref{set:markov-BSDE}), and $g$ is uniformly Lipschitz w.r.t.~$z$ (see Definition \ref{def:g_expect}). Therefore, $\check F_\cdot^{N,\lambda,\check\pi}(\cdot,\cdot,\cdot)$ is uniformly Lipschitz w.r.t.~$(y,z)$ with the corresponding Lipschitz constant depending only on $C_{\beta}, \lambda, N $ (not on $t,x$). Moreover, for all $(t,x)\in[0,T]\times \mathbb{R}^d$,
        \[
            |\check F_t^{N,\lambda,\check\pi}(x,0,0)| \le |r(x)| + N|R(x)\check\pi_t(x)| + \lambda \big|{\cal H}\big(\check\pi_t(x)\big)\big|.
        \]
        Note that $|{\cal H}(\check\pi_t(\cdot))|$ is bounded by $\log 2 $ (see Remark \ref{rem:entropy}), and $r(\cdot)$ and $R(\cdot)$ are linearly growing. Therefore, there exists a constant $C $ only depends on $C_{r,R}, N, \lambda $ (not on $(t,x)$) such that \(|\check F^{N,\lambda,\check\pi}(t,x,0,0)| \le C(1+|x| )\) for all $(t,x)\in [0,T]\times \mathbb{R}^d$. Thus, (M1b) holds true. 

        We now prove (ii). As $r(x),R(x),\check\pi_t(x)$ are continuous w.r.t $x$ for all $t \in [0,T] $, 
        the condition ($\mathrm{M1b^c}$) given in \cite{KHM} holds true. Therefore, an application of \cite[Theorem~8.12]{KHM} ensures 
        that 
        \(
        v^{N,\lambda,\check \pi}(t,x):=\check{Y}_t^{t,x;N,\lambda,\check\pi}\) for \((t,x)\in[0,T]\times \mathbb{R}^d
        \)
        is a viscosity solution
        of the PDE given in the statement (ii); see \eqref{eq:controlled_fbsde}. Moreover, using the flow property of $\{\check X_s^{t,x};t\leq s\leq T,x\in \mathbb{R}^d\}$ and the uniqueness of the solution of \eqref{eq:controlled_fbsde}, we have for $t\leq s\leq T$, $\mathbb{P}\otimes ds$-a.e., \(v^{N,\lambda,\check \pi}(s,\check X_s^{t,x})=\check Y_s^{s,\check X_s^{t,x};N,\lambda, \check \pi}= \check Y_s^{t,x;N,\lambda, \check \pi},\) that is, the property in \eqref{eq:deter_YZ} holds.
        Lastly, the regularity of $v^{N,\lambda,\check \pi}$ follows from the argument in the proof of \cite[Theorem~8.12]{KHM}, which employs the $L^p$-estimation techniques in
        the proof of \cite[Lemma~2.1 and Corollary~2.10]{pardoux2005backward}. 
        \end{proof}

    \begin{proof}[Proof of Corollary~\ref{coro:Markovian_of_policy_iter_closed_loop}] Part (i) follows immediately from an iterative application of Lemma \ref{lem:Markovian_under_closed_loop_policy}\;(i). In a similary manner, Part (ii) is obtained by iteratively applying Lemma~\ref{lem:Markovian_under_closed_loop_policy}\;(ii). Indeed, as $\check\pi_t^1(\cdot)$ is continuous, the corresponding function $v^{N,\lambda,1}$ satisfies all the properties in Part (i) and is also a viscosity solution of the PDE given in the statement (with the generator $\check F_\cdot^{N,\lambda ,\check \pi^1})$. In particular, it is continuous on $[0,T]\times \mathbb{R}^d$, the next iteration policy $\check \pi_t^2(\cdot)$ ,$t\in[0,T]$, (defined as in \eqref{eq:robust_iter_control_deter}) is also continuous on $\mathbb{R}^d$. The same argument can therefore be applied at each subsequent iteration. This completes the proof.
    \end{proof}


\bibliographystyle{siamplain}
\bibliography{main}

@article{bartl2023sensitivity,
	title={Sensitivity of robust optimization problems under drift and volatility uncertainty},
	author={Bartl, Daniel and Neufeld, Ariel and Park, Kyunghyun},
	journal={Finance Stoch., arXiv:2311.11248},
    note={forthcoming},
	year={2025+}
}

@article{Becker-Jentzen-Neufeld-Deep-Splitting-21,
	author = {Beck, Christian and Becker, Sebastian and Cheridito, Patrick and Jentzen, Arnulf and Neufeld, Ariel},
	title = {Deep Splitting Method for Parabolic {PDEs}},
	journal = {SIAM J. Sci. Comput.},
	volume = {43},
	number = {5},
	pages = {A3135-A3154},
	year = {2021}
}

@article{blackwell1983extension,
	title={An extension of {Skorohod's} almost sure representation theorem},
	author={Blackwell, David and Dubins, Lester E},
	journal={Proc. Amer. Math. Soc.},
	volume={89},
	number={4},
	pages={691--692},
	year={1983}
}

@techreport{DaiDong2024-dynkin,
  title  = {Learning an Optimal Investment Policy with Transaction Costs via a Randomized Dynkin Game},
  author = {Dai, Min and Dong, Yuchao},
  year   = {2024},
  month  = {6},
  note   = {Available at SSRN 4871712},
  type   = {Working Paper},
  url    = {https://ssrn.com/abstract=4871712},
  doi    = {10.2139/ssrn.4871712}
}

@article{coquet2002filtration,
	title={Filtration-consistent nonlinear expectations and related {$g$}-expectations},
	author={Coquet, Fran{\c{c}}ois and Hu, Ying and M{\'e}min, Jean and Peng, Shige},
	journal={Probab. Theory Relat. Fields},
	volume={123},
	number={1},
	pages={1--27},
	year={2002},
	publisher={Springer}
}

@article{dai2024learning,
	author={Dai, Min and Sun, Yu and Xu, Zuo Quan and Zhou, Xun Yu},
	title = {Learning to Optimally Stop Diffusion Processes, with Financial Applications},
    journal = {Manag. Sci.},
    note = {Articles in Advance},
    year = {2026},
    doi = {10.1287/mnsc.2024.07614}
}

@article{DongRLstopping,
	author = {Dong, Yuchao},
	title = {Randomized Optimal Stopping Problem in Continuous Time and Reinforcement Learning Algorithm},
	journal = {SIAM J. Control Optim.},
	volume = {62},
	number = {3},
	pages = {1590-1614},
	year = {2024}
}

@article{el1997reflected,
	title={Reflected solutions of backward {SDE}, and related obstacle problems for {PDEs}},
	author={El Karoui, Nicole and Kapoudjian, Christophe and Pardoux, Etienne and Peng, Shige and Quenez, Marie-Claire},
	journal={Ann. Probab.},
	volume={25},
	number={2},
	pages={702--737},
	year={1997},
	publisher={Institute of Mathematical Statistics}
}

@article{el1997backward,
	title={Backward stochastic differential equations in finance},
	author={El Karoui, Nicole and Peng, Shige and Quenez, Marie Claire},
	journal={Math. Finance},
	volume={7},
	number={1},
	pages={1--71},
	year={1997},
	publisher={Wiley Online Library}
}

@article{ferrari2022optimal,
	title={Optimal consumption with {Hindy--Huang--Kreps} preferences under nonlinear expectations},
	author={Ferrari, Giorgio and Li, Hanwu and Riedel, Frank},
	journal={Adv. Appl. Probab.},
	volume={54},
	number={4},
	pages={1222--1251},
	year={2022},
	publisher={Cambridge University Press}
}

@article{Forsyth-Penalty-02,
	author = {Forsyth, P. A. and Vetzal, K. R.},
	title = {Quadratic Convergence for Valuing {American} Options Using a Penalty Method},
	journal = {SIAM J. Sci. Comput.},
	volume = {23},
	number = {6},
	pages = {2095-2122},
	year = {2002}
}

@article{pham2022deepBSDE_erroranalysis,
	author = {Germain, Maximilien and Pham, Huy\^{e}n and Warin, Xavier},
	title = {Approximation Error Analysis of Some Deep Backward Schemes for Nonlinear PDEs},
	journal = {SIAM J. Sci. Comput.},
	volume = {44},
	number = {1},
	pages = {A28-A56},
	year = {2022}
}

@book{goodfellow2016deep,
	title={Deep learning},
	author={Goodfellow, Ian and Bengio, Yoshua and Courville, Aaron and Bengio, Yoshua},
	volume={1},
	number={2},
	year={2016},
	publisher={MIT press Cambridge}
}

@article{han18,
	author = { Jiequn Han  and Arnulf Jentzen  and Weinan E },
	title = {Solving high-dimensional partial differential equations using deep learning},
	journal = {Proc. Natl. Acad. Sci.,},
	volume = {115},
	number = {34},
	pages = {8505-8510},
	year = {2018}
}

@article{pham2020deepBSDE,
	author = {Huré, Côme and Pham, Huyên and Warin, Xavier},
	year = {2020},
	month = {01},
	pages = {1},
	title = {Deep backward schemes for high-dimensional nonlinear {PDEs}},
	volume = {89},
	journal = {Math. Comp.}
}

@inbook{KHM,
	title = {Chapter Eight. BSDEs And Applications},
	author = {Nicole {El Karoui} and Said Hamad\`ene and Anis Matoussi},
	editor = {René Carmona},
	publisher = {Princeton University Press},
	address = {Princeton},
	pages = {267--320},
	year = {2009},
    note = {in Indifference Pricing: Theory and Applications}
}

@article{lepeltier2005penalization,
	title={Penalization method for reflected backward stochastic differential equations with one r.c.l.l. barrier},
	author={Lepeltier, J-P and Xu, Mingyu},
	journal={Stat. Probab. Lett.},
	volume={75},
	number={1},
	pages={58--66},
	year={2005},
	publisher={Elsevier}
}

@book{mao2007stochastic,
	title={Stochastic differential equations and applications},
	author={Mao, Xuerong},
	year={2007},
	publisher={Elsevier}
}

@article{neufeld2024full,
	title={Full error analysis of the random deep splitting method for nonlinear parabolic {PDEs} and {PIDEs}},
	author={Neufeld, Ariel and Schmocker, Philipp and Wu, Sizhou},
	journal={arXiv preprint arXiv:2405.05192},
	year={2024}
}

@inproceedings{pardoux2005backward,
	title={Backward stochastic differential equations and quasilinear parabolic partial differential equations},
	author={Pardoux, Etienne and Peng, Shige},
	booktitle={Stochastic Partial Differential Equations and Their Applications: Proceedings of IFIP WG 7/1 International Conference University of North Carolina at Charlotte, NC June 6--8, 1991},
	pages={200--217},
	year={2005},
	organization={Springer}
}

@article{pardoux1990adapted,
	title={Adapted solution of a backward stochastic differential equation},
	author={Pardoux, Etienne and Peng, Shige},
	journal={Syst. Control Lett.},
	volume={14},
	number={1},
	pages={55--61},
	year={1990},
	publisher={Elsevier}
}

@article{peng1997backward,
	title={Backward {SDE} and related $g$-expectation},
	author={Peng, Shige},
	journal={Pitman research notes in mathematics series},
	pages={141--160},
	year={1997},
	publisher={Longman Scientific \& Technical}
}

@article{peng2005smallest,
	title={The smallest {$g$}-supermartingale and reflected {BSDE} with single and double {$L^2$} obstacles},
	author={Peng, Shige and Xu, Mingyu},
	journal={Ann. Inst. H. Poincar{\'e} Probab. Statist.},
	volume={41},
	number={3},
	pages={605--630},
	year={2005}
}

@article{sirignano2018dgm,
	title={{DGM}: A deep learning algorithm for solving partial differential equations},
	author={Sirignano, Justin and Spiliopoulos, Konstantinos},
	journal={J. Comput. Phys.},
	volume={375},
	pages={1339--1364},
	year={2018},
	publisher={Elsevier}
}

@article{wang2020reinforcement,
	title={Reinforcement learning in continuous time and space: A stochastic control approach},
	author={Wang, Haoran and Zariphopoulou, Thaleia and Zhou, Xun Yu},
	journal={J. Mach. Learn. Res.},
	volume={21},
	number={198},
	pages={1--34},
	year={2020}
}

@article{garcia2015comprehensive-safeRL,
  title   = {A Comprehensive Survey on Safe Reinforcement Learning},
  author  = {Garc{\'i}a, Javier and Fern{\'a}ndez, Fernando},
  journal = {J. Machine Learn. Res.},
  volume  = {16},
  number  = {42},
  pages   = {1437--1480},
  year    = {2015}
}

@InProceedings{pinto2017robust,
  title = 	 {Robust Adversarial Reinforcement Learning},
  author =       {Lerrel Pinto and James Davidson and Rahul Sukthankar and Abhinav Gupta},
  booktitle = 	 {Proceedings of the 34th International Conference on Machine Learning},
  pages = 	 {2817--2826},
  year = 	 {2017},
  editor = 	 {Precup, Doina and Teh, Yee Whye},
  volume = 	 {70},
  series = 	 {Proceedings of Machine Learning Research},
  month = 	 {06--11 Aug},
  publisher =    {PMLR},
}

@misc{christiano2016transfer,
      title={Transfer from Simulation to Real World through Learning Deep Inverse Dynamics Model}, 
      author={Paul Christiano and Zain Shah and Igor Mordatch and Jonas Schneider and Trevor Blackwell and Joshua Tobin and Pieter Abbeel and Wojciech Zaremba},
      year={2016},
      eprint={1610.03518},
      archivePrefix={arXiv},
}

@misc{tankov2025optimalstoppingdivestmenttiming,
      title={Optimal stopping and divestment timing under scenario ambiguity and learning}, 
      author={Andrea Mazzon and Peter Tankov},
      year={2025},
      eprint={2408.09349},
      archivePrefix={arXiv},
      primaryClass={q-fin.MF},
      url={https://arxiv.org/abs/2408.09349}, 
}

@article{wang2020continuous,
	title={Continuous-time mean--variance portfolio selection: A reinforcement learning framework},
	author={Wang, Haoran and Zhou, Xun Yu},
	journal={Math. Finance},
	volume={30},
	number={4},
	pages={1273--1308},
	year={2020},
	publisher={Wiley Online Library}
}

@book{zhangjianfeng17,
	author = {Jianfeng Zhang},
	title={Backward Stochastic Differential Equations},
	year={2017},
	publisher={Springer New York},
	address={New York}
}

@ARTICLE{SuttonBarto1998,
	author={Sutton, R.S. and Barto, A.G.},
	journal={IEEE Trans. Neural Netw.}, 
	title={Reinforcement Learning: An Introduction}, 
	year={1998},
	volume={9},
	number={5},
	pages={1054-1054},
}

@article{Silver2016Mastering,
  author = {Silver, D. and Huang, A. and Maddison, C. and others},
  title = {Mastering the game of {Go} with deep neural networks and tree search},
  journal={Nature},
  volume={529},
  number={7587},
  pages={484--489},
  year={2016},
  publisher={Nature Publishing Group}
}

@article{Silver2017MasteringWithout,
  author = {Silver, D. and Schrittwieser, J. and Simonyan, K. and others},
  title = {Mastering the game of {Go} without human knowledge},
  journal={Nature},
  volume={550},
  number={7676},
  pages={354--359},
  year={2017},
  publisher={Nature Publishing Group UK London}
}

@article{Mnih2015HumanLevel,
  author = {Mnih, V. and Kavukcuoglu, K. and Silver, D. and others},
  title = {Human-level control through deep reinforcement learning},
  journal = {Nature},
  volume = {518},
  pages = {529--533},
  year = {2015}
}

@article{small-q,
	author  = {Yanwei Jia and Xun Yu Zhou},
	title   = {q-Learning in Continuous Time},
	journal={J. Mach. Learn. Res.},
	year    = {2023},
	volume  = {24},
	number  = {161},
	pages   = {1--61}
}

@article{Policy-evaluation,
	title={Policy evaluation and temporal-difference learning in continuous time and space: A martingale approach},
  author={Jia, Yanwei and Zhou, Xun Yu},
	journal = {J. Mach. Learn. Res.},
  volume={23},
  number={154},
  pages={1--55},
  year={2022}
}

@article{autodriver,
	author = {Levine, Sergey and Finn, Chelsea and Darrell, Trevor and Abbeel, Pieter},
	title = {End-to-end training of deep visuomotor policies},
	year = {2016},
	volume = {17},
	number = {1},
	journal = {J. Mach. Learn. Res.},
	pages = {1334–1373}
}

@article{wald1948optimum,
	title={Optimum character of the sequential probability ratio test},
	author={Wald, Abraham and Wolfowitz, Jacob},
	journal={Ann. Math. Stat.},
	pages={326--339},
	year={1948},
	publisher={JSTOR}
}

@article{dybvig1995dusenberry,
	title={Dusenberry's ratcheting of consumption: optimal dynamic consumption and investment given intolerance for any decline in standard of living},
	author={Dybvig, Philip H},
	journal={Rev. Econ. Stud.},
	volume={62},
	number={2},
	pages={287--313},
	year={1995},
	publisher={Wiley-Blackwell}
}

@book{peskir2006optimal,
	title={Optimal stopping and free-boundary problems},
	author={Peskir, Goran and Shiryaev, Albert},
	year={2006},
	publisher={Springer}
}

@article{epstein2003recursive,
	title={Recursive multiple-priors},
	author={Epstein, Larry G and Schneider, Martin},
	journal={J. Econ. Theory},
	volume={113},
	number={1},
	pages={1--31},
	year={2003},
	publisher={Elsevier}
}

@article{marinacci1999limit,
	title={Limit laws for non-additive probabilities and their frequentist interpretation},
	author={Marinacci, Massimo},
	journal={J. Econ. Theory},
	volume={84},
	number={2},
	pages={145--195},
	year={1999},
	publisher={Elsevier}
}

@article{Policy-gradient,
	author  = {Yanwei Jia and Xun Yu Zhou},
	title   = {Policy Gradient and Actor-Critic Learning in Continuous Time and Space: Theory and Algorithms},
	journal={J. Mach. Learn. Res.},
  volume={23},
  number={275},
  pages={1--50},
  year={2022}
}

@article{exploratoryHJB,
	author = {Tang, Wenpin and Zhang, Yuming Paul and Zhou, Xun Yu},
	title = {Exploratory {HJB} Equations and Their Convergence},
	journal = {SIAM J. Control Optim.},
	volume = {60},
	number = {6},
	pages = {3191-3216},
	year = {2022}
}

@article{han2023choquet,
	title={Choquet regularization for continuous-time reinforcement learning},
	author={Han, Xia and Wang, Ruodu and Zhou, Xun Yu},
	journal = {SIAM J. Control Optim.},
	volume={61},
	number={5},
	pages={2777--2801},
	year={2023},
	publisher={SIAM}
}

@article{Liwu2024,
	title = {Reinforcement learning for continuous-time mean-variance portfolio selection in a regime-switching market},
	journal = {J. Econ. Dyn. Control},
	volume = {158},
	pages = {104787},
	year = {2024},
	author = {Bo Wu and Lingfei Li},
}

@article{DaiLearning2023,
	author = {Dai, Min and Dong, Yuchao and Jia, Yanwei},
	title = {Learning equilibrium mean-variance strategy},
	journal = {Math. Finance},
	volume = {33},
	number = {4},
	pages = {1166-1212},
	year = {2023}
}

@article{DaiRecursive20203,
	author = {Dai, Min and Dong, Yuchao and Jia, Yanwei and Zhou, Xunyu},
	year = {2023},
	month = {01},
	pages = {},
	title = {Learning Merton's Strategies in an Incomplete Market: Recursive Entropy Regularization and Biased Gaussian Exploration},
	journal = {SSRN Electronic Journal},
	doi = {10.2139/ssrn.4668480}
}

@article{bayraktar2011optimal_a,
	title={Optimal stopping for non-linear expectations—{Part I}},
	author={Bayraktar, Erhan and Yao, Song},
	journal={Stochastic Process. Appl.},
	volume={121},
	number={2},
	pages={185--211},
	year={2011},
	publisher={Elsevier}
}

@article{bayraktar2011optimal_b,
	title={Optimal stopping for non-linear expectations—{Part II}},
	author={Bayraktar, Erhan and Yao, Song},
	journal={Stochastic Process. Appl.},
	volume={121},
	number={2},
	pages={212--264},
	year={2011},
	publisher={Elsevier}
}

@article{nutz2015optimal,
	title={OPTIMAL STOPPING UNDER ADVERSE NONLINEAR EXPECTATION AND RELATED GAMES},
	author={Nutz, Marcel and Zhang, Jianfeng},
	journal={Ann. Appl. Probab.},
	volume={25},
	number={5},
	pages={2503--2534},
	year={2015}
}

@article{PW23,
	author = {Park, K. and Wong, H. Y. },
	title = {Robust Retirement with Return Ambiguity: Optimal {$G$}-Stopping Time in Dual Space},
	journal={SIAM J. Control Optim.},
	volume = {61},
	number ={3},
	year =	 {2023},
	pages ={1009--1037}
}

@article{riedel2009optimal,
	title={Optimal stopping with multiple priors},
	author={Riedel, Frank},
	journal={Econometrica},
	volume={77},
	number={3},
	pages={857--908},
	year={2009},
	publisher={Wiley Online Library}
}

@article{bartl2024numerical,
	title={Numerical method for nonlinear {Kolmogorov PDEs} via sensitivity analysis},
	author={Bartl, Daniel and Neufeld, Ariel and Park, Kyunghyun},
	journal={Appl. Math. Optim., arXiv:2403.11910},
    note={forthcoming},
	year={2026+}
}

@article{park2023irreversible,
	title={Irreversible Consumption Habit under Ambiguity: Singular Control and Optimal {$G$}-Stopping Time},
	author={Park, Kyunghyun and Chen, Kexin and Wong, Hoi Ying},
	journal={Ann. Appl. Probab.},
	volume={35},
	number={4},
	pages={2471--2525},
	year={2025},
	publisher={Institute of Mathematical Statistics}
}

@article{chen2024robust,
	title={Robust dividend policy: Equivalence of {Epstein-Zin} and {Maenhout} preferences},
	author={Chen, Kexin and Park, Kyunghyun and Wong, Hoi Ying},
	journal={arXiv preprint arXiv:2406.12305},
	year={2024}
}

@article{klibanoff2005smooth,
	title={A smooth model of decision making under ambiguity},
	author={Klibanoff, Peter and Marinacci, Massimo and Mukerji, Sujoy},
	journal={Econometrica},
	volume={73},
	number={6},
	pages={1849--1892},
	year={2005},
	publisher={Wiley Online Library}
}

@article{frikha2025entropy,
	title={An Entropy Regularized {BSDE} Approach to {Bermudan} Options and Games},
	author={Frikha, Noufel and Li, Libo and Chee, Daniel},
	journal={arXiv preprint arXiv:2509.18747},
	year={2025}
}

@article{becker2019deep,
	title={Deep optimal stopping},
	author={Becker, Sebastian and Cheridito, Patrick and Jentzen, Arnulf},
	journal={J. Mach. Learn. Res.},
	volume={20},
	number={74},
	pages={1--25},
	year={2019}
}

@article{becker2021solving,
	title={Solving high-dimensional optimal stopping problems using deep learning},
	author={Becker, Sebastian and Cheridito, Patrick and Jentzen, Arnulf and Welti, Timo},
	journal={Eur. J. Appl. Math.},
	volume={32},
	number={3},
	pages={470--514},
	year={2021},
	publisher={Cambridge University Press}
}

@article{reppen2025neural,
	title={Neural optimal stopping boundary},
	author={Reppen, Andres Max and Soner, Halil Mete and Tissot-Daguette, Valentin},
	journal={Math. Finance},
	volume={35},
	number={2},
	pages={441--469},
	year={2025},
	publisher={Wiley Online Library}
}

@article{deep_splitting_convergence,
author = {Frey, R\"{u}diger and K\"{o}ck, Verena},
title = {Convergence Analysis of the Deep Splitting Scheme: The Case of Partial Integro-Differential Equations and the Associated Forward Backward {SDEs} with Jumps},
	journal = {SIAM J. Sci. Comput.},
volume = {47},
number = {1},
pages = {A527-A552},
year = {2025}
}

@article{RobustRL-model-mismatch,
  title={Reinforcement learning under model mismatch},
    author = {Roy, Aurko and Xu, Huan and Pokutta, Sebastian},
  journal={Adv. Neural Inf. Process. Syst.},
  volume={30},
  year={2017}
}

@article{RobustRL-deep,
    author={Zhang, Huan and Chen, Hongge and Xiao, Chaowei and Li, Bo and Liu, Mingyan and Boning, Duane and Hsieh, Cho-Jui},
      title={Robust deep reinforcement learning against adversarial perturbations on state observations},
  journal={Adv. Neural Inf. Process. Syst.},
  volume={33},
  pages={21024--21037},
  year={2020}
}

@article{Guo2025DeepSeekR1-nature,
  author = {Guo, D. and Yang, D. and Zhang, H. and others},
  title = {DeepSeek-R1 incentivizes reasoning in {LLMs} through reinforcement learning},
  journal = {Nature},
  volume = {645},
  pages = {633--638},
  year = {2025}  
}

@book{Protter2005Stochastic,
  author = {Protter, Philip E.},
  title = {Stochastic Integration and Differential Equations},
  series = {Stochastic Modelling and Applied Probability},
  publisher = {Springer},
  address = {Berlin, Heidelberg},
  year = {2005},
  edition = {2},
  pages = {XIII, 415}
}

@book{jacod2013limit,
  title={Limit theorems for stochastic processes},
  author={Jacod, Jean and Shiryaev, Albert},
  volume={288},
  year={2013},
  publisher={Springer Science \& Business Media}
}

@article{blanchet2023double,
  title={Double pessimism is provably efficient for distributionally robust offline reinforcement learning: Generic algorithm and robust partial coverage},
  author={Blanchet, Jose and Lu, Miao and Zhang, Tong and Zhong, Han},
  journal={Adv. Neural Inf. Process. Syst.},
  volume={36},
  pages={66845--66859},
  year={2023}
}

@article{panaganti2022robust,
  title={Robust reinforcement learning using offline data},
  author={Panaganti, Kishan and Xu, Zaiyan and Kalathil, Dileep and Ghavamzadeh, Mohammad},
  journal={Adv. Neural Inf. Process. Syst.},
  volume={35},
  pages={32211--32224},
  year={2022}
}

@article{morimoto2005robust,
  title={Robust reinforcement learning},
  author={Morimoto, Jun and Doya, Kenji},
  journal={Neural Comput.},
  volume={17},
  number={2},
  pages={335--359},
  year={2005},
  publisher={MIT Press}
}

@article{dianetti2024exploratory,
  title={Exploratory optimal stopping: A singular control formulation},
  author={Dianetti, Jodi and Ferrari, Giorgio and Xu, Renyuan},
  journal={arXiv preprint arXiv:2408.09335},
  year={2024}
}

@article{huang2025convergence,
  title={Convergence of policy iteration for entropy-regularized stochastic control problems},
  author={Huang, Yu-Jui and Wang, Zhenhua and Zhou, Zhou},
    journal = {SIAM J. Control Optim.},
  volume={63},
  number={2},
  pages={752--777},
  year={2025},
  publisher={SIAM}
}

@article{pocock1977group,
  title={Group sequential methods in the design and analysis of clinical trials},
  author={Pocock, Stuart J},
  journal={Biometrika},
  volume={64},
  number={2},
  pages={191--199},
  year={1977},
  publisher={Oxford University Press}
}

@book{gittins2011multi,
  title={Multi-armed bandit allocation indices},
  author={Gittins, John and Glazebrook, Kevin and Weber, Richard},
  year={2011},
  publisher={John Wiley \& Sons}
}

@article{pierskalla1976survey,
  title={A survey of maintenance models: the control and surveillance of deteriorating systems},
  author={Pierskalla, William P and Voelker, John A},
  journal={Naval Research Logistics Quarterly},
  volume={23},
  number={3},
  pages={353--388},
  year={1976},
  publisher={Wiley Online Library}
}
\end{document}